\documentclass[a4paper,12pt,reqno]{amsart}
\usepackage[T1]{fontenc}
\usepackage[utf8]{inputenc}
\usepackage{amssymb,dsfont,mathrsfs}
\usepackage[margin=1in]{geometry}
\usepackage{enumitem}
\usepackage{graphicx}

\numberwithin{equation}{section}

\theoremstyle{plain}
\newtheorem{theorem}{Theorem}[section]
\newtheorem{lemma}[theorem]{Lemma}
\newtheorem{corollary}[theorem]{Corollary}
\newtheorem{proposition}[theorem]{Proposition}

\theoremstyle{definition}
\newtheorem{definition}[theorem]{Definition}
\newtheorem{example}[theorem]{Example}
\newtheorem{remark}[theorem]{Remark}

\DeclareMathOperator{\re}{Re}
\DeclareMathOperator{\im}{Im}

\DeclareMathOperator{\Arg}{Arg}

\DeclareMathOperator{\Int}{Int}
\DeclareMathOperator{\Cl}{Cl}
\DeclareMathOperator{\sign}{sign}
\DeclareMathOperator{\esssupp}{ess\,supp}

\newcommand{\C}{\mathds{C}}
\newcommand{\hp}{\mathds{H}}
\newcommand{\disk}{\mathds{D}}
\newcommand{\R}{\mathds{R}}
\newcommand{\ind}{\mathds{1}}

\newcommand{\sub}{\subseteq}
\newcommand{\ph}{\varphi}
\newcommand{\ro}{\varrho}
\newcommand{\eps}{\varepsilon}

\newcommand{\dom}{D}
\newcommand{\zero}{Z}

\newcommand{\pr}{\mathbf{P}}
\newcommand{\ex}{\mathbf{E}}

\newcommand{\tset}[1]{\{ #1 \}}
\newcommand{\abs}[1]{\left| #1 \right|}

\newcommand{\thet}{\vartheta}

\newcommand{\ol}[1]{\overline{#1}}
\newcommand{\ul}[1]{\underline{#1}}
\newcommand{\ignore}[1]{}

\usepackage{ifthen}
\newcommand{\formula}[2][nolabel]%
{%
 \ifthenelse{\equal{#1}{nolabel}}%
 {\begin{align*} \begin{aligned} #2 \end{aligned} \end{align*}}%
 {%
  \ifthenelse{\equal{#1}{}}%
  {\begin{align} #2 \end{align}}%
  {\begin{align} \label{#1} \begin{aligned} #2 \end{aligned} \end{align}}%
 }%
}

\begin{document}

%
%

\title[Fluctuation theory for Lévy processes with completely monotone jumps]{Fluctuation theory for Lévy processes with completely monotone jumps}
\author{Mateusz Kwaśnicki}
\thanks{Work supported by the Polish National Science Centre (NCN) grant no.\@ 2015/19/B/ST1/01457}
\address{Mateusz Kwaśnicki \\ Faculty of Pure and Applied Mathematics \\ Wrocław University of Science and Technology \\ Wybrze{\.z}e Wyspia\'nskiego 27 \\ 50-370 Wroc{\l}aw, Poland}
\email{mateusz.kwasnicki@pwr.edu.pl}
\date{\today}

\begin{abstract}
We study the Wiener--Hopf factorization for L\'evy processes $X_t$ with completely monotone jumps. Extending previous results of L.C.G.~Rogers, we prove that the space-time Wiener--Hopf factors are complete Bernstein functions of both the spatial and the temporal variable. As a corollary, we prove complete monotonicity of: (a) the tail of the distribution function of the supremum of $X_t$ up to an independent exponential time; (b) the Laplace transform of the supremum of $X_t$ up to a fixed time $T$, as a function of $T$. The proof involves a detailed analysis of the holomorphic extension of the characteristic exponent $f(\xi)$ of $X_t$, including a peculiar structure of the curve along which $f(\xi)$ takes real values.
\end{abstract}

\maketitle


%
%

\section{Introduction}
\label{sec:intro}

This is the first in a series of papers, where we study a class of one-dimensional Lévy processes $X_t$ with \emph{completely monotone jumps}, introduced by L.C.G.~Rogers in~\cite{bib:r83}. The main objective of this article is to provide a detailed description of characteristic (Lévy--Khintchine) exponents $f$ of these Lévy processes and their Wiener--Hopf factors $\kappa^+(\tau, \xi)$, $\kappa^-(\tau, \xi)$. In particular, we extend the result of~\cite{bib:r83}, which asserts that $\kappa^+(\tau, \xi)$ and $\kappa^-(\tau, \xi)$ are complete Bernstein functions of $\xi$ (or, equivalently, that the ladder height processes associated with $X_t$ have completely monotone jumps). Our main result states that $\kappa^+(\tau, \xi)$ and $\kappa^-(\tau, \xi)$ are additionally complete Bernstein functions of $\tau$ (that is, also the ladder time processes have completely monotone jumps). In fact, we prove an even stronger statement.

\begin{theorem}
\label{thm:main}
Suppose that $X_t$ is a one-dimensional Lévy process with completely monotone jumps, possibly killed at a uniform rate. Let $\kappa^+(\tau, \xi)$ and $\kappa^-(\tau, \xi)$ denote its space-time Wiener--Hopf factors, that is, the characteristic (Laplace) exponents of bi-variate ladder processes associated to $X_t$. Then $\kappa^+(\tau, \xi)$ and $\kappa^-(\tau, \xi)$ are complete Bernstein functions of both $\xi$ (for each fixed $\tau \ge 0$) and $\tau$ (for each fixed $\xi \ge 0$). Furthermore, if $0 \le \tau_1 \le \tau_2$ and $0 \le \xi_1 \le \xi_2$, then
\formula[eq:main:1]{
 & \frac{\kappa^+(\tau_1, \xi)}{\kappa^+(\tau_2, \xi)} \, , && \frac{\kappa^-(\tau_1, \xi)}{\kappa^-(\tau_2, \xi)} \, , && \frac{\kappa^+(\tau, \xi_1)}{\kappa^+(\tau, \xi_2)} \, , && \frac{\kappa^-(\tau, \xi_1)}{\kappa^-(\tau, \xi_2)}
}
are complete Bernstein functions of $\xi$ and $\tau$, respectively. Finally, if $\xi_1, \xi_2 \ge 0$, then
\formula[eq:main:2]{
 & \kappa^+(\tau, \xi_1) \kappa^-(\tau, \xi_2)
}
is a complete Bernstein function of $\tau$.
\end{theorem}

A simple application of Theorem~\ref{thm:main} is given in Corollary~\ref{cor:cm}. When $X_t$ is a compound Poisson process, the expression~\eqref{eq:main:2} can be given two meanings, and both lead to a complete Bernstein function of $\tau$; see Section~\ref{sec:xwh} for further details.

A brief introduction to fluctuation theory for Lévy processes, which includes the basic expressions for $\kappa^+(\tau, \xi)$ and $\kappa^-(\tau, \xi)$, is given in Section~\ref{sec:levy}. The notions of Lévy processes with completely monotone jumps and complete Bernstein functions are discussed in Section~\ref{sec:rogers}. Theorem~\ref{thm:main} is proved in Section~\ref{sec:xwh}, after a number of intermediate results in Sections~\ref{sec:real} and~\ref{sec:wh}. One lengthy and technical proof is moved to Section~\ref{sec:proof}.

Characteristic exponents $f$ of Lévy processes described in Theorem~\ref{thm:main} form a class of holomorphic functions with remarkable properties, for which we propose the name \emph{Rogers functions}. We show that $f$ is in this class if and only if $f$ extends to a function holomorphic in the right complex half-plane $\hp = \{\xi \in \C : \re \xi > 0\}$, and $\re(f(\xi) / \xi) \ge 0$ for all $\xi \in \hp$. We also give various equivalent definitions of the class of Rogers function, provide integral representations for their Wiener--Hopf factors, and describe highly regular structure of the set $\Gamma_f = \{\xi \in \hp : f(\xi) \in (0, \infty)\}$, which we call the \emph{spine} of $f$.

The results of this paper will be applied in a number of forthcoming works. We will use them in the study of hitting times of points for $X_t$, in the spirit of~\cite{bib:k12} and~\cite{bib:jk14}, which cover the symmetric case. The Wiener--Hopf formulae of the present paper will be the key ingredient in the derivation of integral expressions for the distribution of the running supremum of $X_t$ and for the transition density of $X_t$ killed upon leaving a half-line; this will extend the results of~\cite{bib:k11} and~\cite{bib:kmr12}, which deal with symmetric processes, as well as~\cite{bib:kk18}, where asymmetric strictly stable Lévy processes are studied. In the symmetric case, the harmonic extension technique for the generator of $X_t$ was discussed in~\cite{bib:km18}; a similar description in the asymmetric setting is expected. Characteristic (Laplace) exponents of the bi-variate ladder processes of $X_t$ form an interesting class of bi-variate complete Bernstein functions, which deserves a separate study. Finally, for a large class of processes $X_t$ discussed in Theorem~\ref{thm:main}, the distribution of $X_t$ turns out to be bell-shaped, and conversely, all known bell-shaped distributions are of that form; see~\cite{bib:k18}.

The Wiener--Hopf factorisation for Lévy processes is currently a very active field of study. Stable Lévy processes, which are prime examples of Lévy processes with completely monotone jumps, have been studied in this context since the pioneering work of Darling~\cite{bib:d56}; other classical references are~\cite{bib:b73,bib:d87,bib:h69}. For a sample of recent works on that subject, see~\cite{bib:bdp08,bib:c15,bib:ds10,bib:gj10,bib:gj12,bib:hk11,bib:ku11,bib:ku13,bib:k16:deep,bib:k18:deep,bib:k19:deep,bib:m13,bib:pz17}. Related fluctuation identities for more general classes of Lévy processes are provided, for example, in~\cite{bib:dr11,bib:fjw15,bib:hk16,bib:ku10,bib:kkp13,bib:kpw14,bib:lm08}; see also the references therein. Our work is also related to recent progress in the theory of non-self-adjoint operators related to Lévy processes, studied in~\cite{bib:kk18,bib:ps17,bib:ps19,bib:psz19,bib:pz17}.

Over the last decade, fluctuation theory stimulated the study of potential and spectral theory for symmetric Lévy processes (and in particular those with completely monotone jumps); see, e.g.,~\cite{bib:bbkrsv09,bib:bgr13,bib:bgr14,bib:bgr14a,bib:cks16,bib:cgt17,bib:g12,bib:g13,bib:gr17,bib:km14,bib:km18a,bib:ksv12,bib:ksv13,bib:ksv14,bib:ksv14a,bib:ksv18,bib:kkms10,bib:k11,bib:k12a,bib:k12,bib:ssv10} and the references therein. There are, however, very few papers where similar problems are studied for asymmetric processes, see~\cite{bib:gs19,bib:gs19a,bib:kk17,bib:p14,bib:y12}. The results of the present article may therefore stimulate the development of potential theory and spectral theory for asymmetric Lévy processes.

A preliminary version of this article, which contained less general results and used different methods, appeared as an unpublished paper~\cite{bib:k:pre}.

We conclude the introduction with a number of notational remarks. For clarity, we often write, for example, $X_t$ for the process $t \mapsto X_t$, or $\kappa^+(\tau, \xi)$ for the function $(\tau, \xi) \mapsto \kappa^+(\tau, \xi)$. We write $\mu(-ds)$ for the mirror image of a measure $\mu(ds)$, that is, for the measure $E \mapsto \mu(-E)$. The interior of a set $E \sub \C$ is denoted by $\Int E$, while its closure by $\Cl E$. The right complex half-plane $\{\xi \in \C : \re \xi > 0\}$ is denoted by $\hp$, while $\disk$ stands for the unit disk $\{\xi \in \C : |\xi| < 1\}$ in the complex plane.

We generally use symbols $x, y$ for spatial variables, $t, s$ for temporal variables, $\xi, \eta$ for Fourier variables corresponding to $x, y$, and $\tau, \sigma$ for Laplace variables corresponding to $t, s$. We also use $r, s, t, z$ as auxiliary variables, for example in integrals. Symbols $a, b, c$ denote constants or parameters. 

The notion of a Rogers function $f(\xi)$ and its domain $\dom_f$ is introduced in Section~\ref{sec:rogers:def}. If $f$ is a Rogers function, the notation $\Gamma_f$, $\zero_f$, $\zeta_f(r)$ and $\lambda_f(r)$ is introduced in Section~\ref{sec:real:spine}, while the symbols $\Gamma_f^\star$, $\dom_f^+$ and $\dom_f^-$ are defined in Section~\ref{sec:real:symspine}.

%
%

\section{Essentials of flutuation theory for Lévy processes}
\label{sec:levy}

\subsection{Lévy processes}

Throughout this work we assume that $X_t$ is a one-dimensional Lévy process. In other words, $X_t$ is a real-valued stochastic process with independent and stationary increments, càdlàg paths, and initial value $X_0 = 0$. We allow $X_t$ to be killed at a uniform rate, that is, the probability of $X_t$ being alive at time $t$ is equal to $e^{-c t}$ for some $c \ge 0$.

We denote by $f(\xi)$ the characteristic (Lévy--Khintchine) exponent of $X_t$, that is,
\formula[eq:lk0]{
 \ex \exp(i \xi X_t) & = \exp(-t f(\xi))
}
for $t \ge 0$ and $\xi \in \R$. By the Lévy--Khintchine formula,
\formula[eq:lk]{
 f(\xi) & = a \xi^2 - i b \xi + c + \int_{\R \setminus \{0\}} (1 - e^{i \xi x} + i \xi (1 - e^{-|x|}) \sign x) \nu(dx) ,
}
where $a \ge 0$ is the Gaussian coefficient, $b \in \R$ is the drift, $c \ge 0$ is the rate at which $X_t$ is killed and $\nu(dx)$ is the Lévy measure, a non-negative Borel measure on $\R \setminus \{0\}$ such that $\int_{\R \setminus \{0\}} \min\{1, x^2\} \nu(dx) < \infty$; $\nu(dx)$ describes the intensity of jumps of $X_t$. If $\nu(dx)$ is absolutely continuous, we denote its density function again by the same symbol $\nu(x)$.

A Lévy process $X_t$ is a \emph{compound Poisson process} if the paths of $X_t$ are piece-wise constant with probability one. This is the case if and only if $a = 0$, $\nu$ is a finite measure and $b = \int_{\R \setminus \{0\}} (1 - e^{-|x|}) \sign x \, \nu(dx)$.

For a general discussion of Lévy processes, we refer to~\cite{bib:a04,bib:bnmr01,bib:b96,bib:k06,bib:s99}.

\subsection{Fluctuation theory}

Fluctuation theory studies the properties of the supremum and the infimum functionals of a Lévy process $X_t$, which are defined by
\formula{
 \ol{X}_t & = \sup \tset{X_s : s \in [0, t]} , & \ul{X}_t & = \inf \tset{X_s : s \in [0, t]} ,
}
as well as times at which these extremal values are attained, denoted by
\formula{
 \ol{T}_t & = \inf \tset{s \in [0, t] : X_s = \ol{X}_t} , & \ul{T}_t & = \inf \tset{s \in [0, t] : X_s = \ul{X}_t} .
}
These times are known to be almost surely unique, unless $X_t$ is a compound Poisson process.

The suprema and infima of Lévy processes are intimately connected with the \emph{ascending} and \emph{descending ladder processes}: bi-variate Lévy processes with non-decreasing coordinates, whose range is almost surely equal to $\{(\ol{X}_t, \ol{T}_t) : t \in [0, \infty)\}$ and $\{(-\ul{X}_t, \ul{T}_t) : t \in [0, \infty)\}$, respectively. The characteristic (Laplace) exponents of the ladder processes are denoted by $\kappa^+(\tau, \xi)$ and $\kappa^-(\tau, \xi)$.

For a detailed introduction to the fluctuation theory of Lévy processes, we refer to~\cite{bib:bnmr01,bib:b96,bib:d07,bib:f74,bib:k06,bib:s99}. Here we limit our attention to the results that are used in this work.

\subsection{Complete monotonicity results}

The tri-variate Laplace transforms of random vectors $(\ol{T}_s, \ol{X}_s)$ and $(\ul{T}_s, \ul{X}_s)$ have the following description in terms of $\kappa^+(\tau, \xi)$ and $\kappa^-(\tau, \xi)$: if $\tau, \sigma \ge 0$ and $\re \xi \ge 0$, then
\formula[eq:pr]{
 \ex \exp(-\xi \ol{X}_S - \tau \ol{T}_S) & = \frac{\kappa^+(\sigma, 0)}{\kappa^+(\tau + \sigma, \xi)} \, , \\
 \ex \exp(\xi \ul{X}_S - \tau \ul{T}_S) & = \frac{\kappa^-(\sigma, 0)}{\kappa^-(\tau + \sigma, \xi)} ,
}
where $S$ is independent from the process $X_t$ and exponentially distributed with intensity $\sigma = (\ex S)^{-1}$. The above formulae are known as Pecherskii--Rogozin identities, established in~\cite{bib:pr69,bib:r66}; see also~\cite{bib:f74,bib:ku10a}. Thus, as an almost straightforward consequence of Theorem~\ref{thm:main}, we obtain the following result. We stress that complete monotonicity of $\pr(\ol{X}_S > x)$ already follows from~\cite{bib:r83}.

\begin{corollary}
\label{cor:cm}
Suppose that $X_t$ is a one-dimensional Lévy process with completely monotone jumps, possibly killed at a uniform rate. Let $S$ be an exponentially distributed random variable independent from the process $X_t$. Then for all $\xi > 0$,
\formula[eq:ft:cm]{
 \pr(\ol{X}_S > x) , && \qquad \pr(\ol{T}_S > s) && \qquad \text{and} \qquad && \ex \exp(-\xi \ol{X}_t)
}
are completely monotone functions on $(0, \infty)$ of $x$, $s$ and $t$, respectively. If $X_t$ converges almost surely to $-\infty$, then additionally $\pr(\ol{X}_\infty > x)$ and $\pr(\ol{T}_\infty > s)$ are completely monotone functions on $(0, \infty)$ of $x$ and $s$, respectively. Similar statements hold for $\ul{X}_t$ and $\ul{T}_t$.
\end{corollary}

\begin{proof}
In the following argument we use well-known properties of complete Bernstein, Stieltjes and completely monotone functions that are discussed in Section~\ref{sec:cbf}.

Suppose that $S$ is exponentially distributed with intensity $\sigma$. The Laplace transforms of the expressions given in~\eqref{eq:ft:cm} are then
\formula{
 \int_0^\infty e^{-\xi x} \pr(\ol{X}_S > x) dx & = \frac{1 - \ex \exp(-\xi \ol{X}_S)}{\xi} = \frac{1}{\xi} \biggl(1 - \frac{\kappa^+(\sigma, 0)}{\kappa^+(\sigma, \xi)}\biggr) , \\
 \int_0^\infty e^{-\tau s} \pr(\ol{T}_S > s) ds & = \frac{1 - \ex \exp(-\tau \ol{T}_S)}{\tau} = \frac{1}{\tau} \biggl(1 - \frac{\kappa^+(\sigma, 0)}{\kappa^+(\tau + \sigma, 0)}\biggr) , \\
 \int_0^\infty e^{-\sigma t} \ex \exp(-\xi \ol{X}_t) dt & = \frac{\ex \exp(-\xi \ol{X}_S)}{\sigma} = \frac{\kappa^+(\sigma, 0)}{\sigma \kappa^+(\sigma, \xi)} \, ,
}
where $\tau, \xi, \sigma > 0$.

By Theorem~\ref{thm:main}, $\kappa^+(\sigma, \xi) / \kappa^+(\sigma, 0)$ is a complete Bernstein function of $\xi$. Therefore, $\kappa^+(\sigma, 0) / \kappa^+(\sigma, \xi)$ is a Stieltjes function of $\xi$. This function is equal to $1$ at $\xi = 0$, and therefore $1 - \kappa^+(\sigma, 0) / \kappa^+(\sigma, \xi)$ is again a complete Bernstein function of $\xi$. We conclude that $\xi^{-1} (1 - \kappa^+(\sigma, 0) / \kappa^+(\sigma, \xi))$ is a Stieltjes function of $\xi$, and consequently $\pr(\ol{X}_S > x)$ is a completely monotone function of $x$.

Similarly, $\pr(\ol{T}_S > s)$ is a completely monotone function of $s$, and an even shorter argument proves that $\ex \exp(-\xi \ol{X}_t)$ is a completely monotone function of $t$.
\end{proof}

\subsection{Fristedt--Pecherski--Rogozin formulae}

There are essentially two types of expressions available for $\kappa^+(\tau, \xi)$ and $\kappa^-(\tau, \xi)$. The \emph{Fristedt--Pecherski--Rogozin formulae} state that if $\tau \ge 0$ and $\re \xi \ge 0$, then
\formula[eq:f]{
 \kappa^+(\tau, \xi) & = \exp\biggl(\int_0^\infty \int_{(0, \infty)} \frac{e^{-t} - e^{-\tau t - \xi x}}{t} \, \pr(X_t \in dx) dt\biggr) , \\
 \kappa^-(\tau, \xi) & = \exp\biggl(\int_0^\infty \int_{(-\infty, 0)} \frac{e^{-t} - e^{-\tau t + \xi x}}{t} \, \pr(X_t \in dx) dt\biggr) ;
}
see~\cite{bib:f74,bib:pr69,bib:r66}. We will take these identities as definitions of $\kappa^+(\tau, \xi)$ and $\kappa^-(\tau, \xi)$. It is also convenient to introduce an additional function
\formula[eq:ft:bullet]{
 \kappa^\circ(\tau) & = \exp\biggl(\int_0^\infty \frac{e^{-t} - e^{-\tau t}}{t} \, \pr(X_t = 0) dt\biggr) ,
}
which is equal to $1$ for all $\tau > 0$ unless $X_t$ is a compound Poisson process. A more analytic approach leads to the \emph{Baxter--Donsker formulae}, first obtained in~\cite{bib:bd57}, and discussed later in this section.

As a simple consequence of~\eqref{eq:f}, we have the Wiener--Hopf factorisation
\formula[eq:wh]{
 \tau + f(\xi) & = (1 + f(0)) \kappa^\circ(\tau) \kappa^+(\tau, -i \xi) \kappa^-(\tau, i \xi)
}
whenever $\tau \ge 0$ and $\xi \in \R$; indeed,
\formula{
 \kappa^\circ(\tau) \kappa^+(\tau, -i \xi) \kappa^-(\tau, i \xi) & = \exp\biggl(\int_0^\infty \int_{\R} \frac{e^{-t} - e^{-\tau t + i \xi x}}{t} \, \pr(X_t \in dx) dt\biggr) \\
 & = \exp\biggl(\int_0^\infty \frac{e^{-t (1 + f(0))} - e^{-t (\tau + f(\xi))}}{t} \, dt\biggr) = \exp\biggl(\log \frac{\tau + f(\xi)}{1 + f(0)}\biggr) ,
}
where the first equality follows from~\eqref{eq:f} and~\eqref{eq:ft:bullet}, the second one is an application of~\eqref{eq:lk0}, and the third one follows from Frullani's integral.

We remark that some authors decide to exclude compound Poisson processes from their considerations, some incorporate $\kappa^\circ(\tau)$ into $\kappa^+(\tau, \xi)$ or $\kappa^-(\tau, \xi)$ instead. Since the former approach is not completely general, while the latter one breaks the symmetry between the Wiener--Hopf factors, we decide to keep $\kappa^\circ(\tau)$ in the notation, following, for example,~\cite{bib:tt02,bib:tt08}.

\subsection{Baxter--Donsker formulae}

The expressions for $\kappa^+(\tau, \xi)$ and $\kappa^-(\tau, \xi)$ discussed below are similar to those obtained by Baxter and Donsker in~\cite{bib:bd57}. We derive them from~\eqref{eq:f}, although in fact the article~\cite{bib:bd57} predated the works of Pecherski--Rogozin and Fristedt. They correspond directly to the meaning of the term \emph{Wiener--Hopf factorisation} in analysis, and they can also be proved by solving a Riemann--Hilbert problem for $\log(\tau + f(\xi))$ with a fixed $\tau > 0$.

Suppose that $\tau_1, \tau_2 > 0$ and $\re \xi_1, \re \xi_2 > 0$. By~\eqref{eq:f}, we have
\formula{
 \frac{\kappa^+(\tau_1, \xi_1) \kappa^+(\tau_2, \xi_2)}{\kappa^+(\tau_1, \xi_2) \kappa^+(\tau_2, \xi_1)} & = \exp\biggl(\int_0^\infty \int_{(0, \infty)} \frac{(e^{-\tau_2 t} - e^{-\tau_1 t}) (e^{-\xi_1 x} - e^{-\xi_2 x})}{t} \, \pr(X_t \in dx) dt\biggr) .
}
Observe that
\formula{
 (e^{-\xi_1 x} - e^{-\xi_2 x}) \ind_{(0, \infty)}(x) & = \frac{1}{2 \pi} \int_{-\infty}^\infty e^{i x z} \biggl(\frac{1}{\xi_1 + i z} - \frac{1}{\xi_2 + i z}\biggr) dz ,
}
and that $(\xi_1 + i z)^{-1} - (\xi_2 + i z)^{-1}$ is absolutely integrable over $z \in \R$. By Fubini's theorem, for $t > 0$ we have
\formula{
 \int_{(0, \infty)} (e^{-\xi_1 x} - e^{-\xi_2 x}) \pr(X_t \in dx) & = \frac{1}{2 \pi} \int_{\R} \int_{-\infty}^\infty e^{i x z} \biggl(\frac{1}{\xi_1 + i z} - \frac{1}{\xi_2 + i z}\biggr) dz \, \pr(X_t \in dx) \\
 & = \frac{1}{2 \pi} \int_{-\infty}^\infty e^{-t f(z)} \biggl(\frac{1}{\xi_1 + i z} - \frac{1}{\xi_2 + i z}\biggr) dz .
}
Furthermore, for $z \in \R$, by Frullani's integral,
\formula{
 \int_0^\infty \frac{e^{-\tau_2 t} - e^{-\tau_1 t}}{t} \, e^{-t f(z)} dt & = \log \frac{\tau_1 + f(z)}{\tau_2 + f(z)} \, .
}
By combining the above results, we obtain a variant of the Baxter--Donsker formula:
\formula[eq:bd:0]{
 \frac{\kappa^+(\tau_1, \xi_1) \kappa^+(\tau_2, \xi_2)}{\kappa^+(\tau_1, \xi_2) \kappa^+(\tau_2, \xi_1)} & = \exp\biggl(\frac{1}{2 \pi} \int_{-\infty}^\infty \biggl(\frac{1}{\xi_1 + i z} - \frac{1}{\xi_2 + i z}\biggr) \log \frac{\tau_1 + f(z)}{\tau_2 + f(z)} \, dz\biggr) .
}
This can be simplified in the following way: by~\eqref{eq:f} and dominated convergence theorem,
\formula[eq:lim]{
 \lim_{\tau \to \infty} \frac{\kappa^+(\tau, \xi_1)}{\kappa^+(\tau, \xi_2)} & = 1 & \qquad \text{and} \qquad && \lim_{\xi \to \infty} \frac{\kappa^+(\tau_1, \xi)}{\kappa^+(\tau_2, \xi)} & = 1 .
}
Combining the first of the above identities with~\eqref{eq:bd:0}, we easily find that
\formula[eq:bd:p]{
 \frac{\kappa^+(\tau, \xi_1)}{\kappa^+(\tau, \xi_2)} & = \exp\biggl(\frac{1}{2 \pi} \int_{-\infty}^\infty \biggl(\frac{1}{\xi_1 + i z} - \frac{1}{\xi_2 + i z}\biggr) \log(\tau + f(z)) \, dz\biggr) ,
}
Indeed, the integral of $(\xi_1 + i z)^{-1} - (\xi_2 + i z)^{-1}$ over $\R$ is $1$, so we may replace $\tau_2 + f(z)$ in~\eqref{eq:bd:0} by $1 + f(z)/\tau_2$. Passing to the limit as $\tau_2 \to \infty$ and applying dominated convergence theorem, we obtain~\eqref{eq:bd:p} with $\tau = \tau_1$. We remark that if $f(\xi) / \xi$ is integrable in the neighbourhood of $0$, then we may pass to the limit as $\xi_2 \to 0^+$ in~\eqref{eq:bd:p} to get the expression found in~\cite{bib:bd57}; however, the more general form~\eqref{eq:bd:p} is more convenient for our needs.

The corresponding expression for $\kappa^-(\tau, \xi)$ reads
\formula[eq:bd:m]{
 \frac{\kappa^-(\tau, \xi_1)}{\kappa^-(\tau, \xi_2)} & = \exp\biggl(\frac{1}{2 \pi} \int_{-\infty}^\infty \biggl(\frac{1}{\xi_1 - i z} - \frac{1}{\xi_2 - i z}\biggr) \log(\tau + f(z)) \, dz\biggr) .
}
In a similar manner, using the identity
\formula{
 \ind_{\{0\}}(x) + e^{-\xi_1 x} \ind_{(0, \infty)}(x) + e^{-\xi_2 x} \ind_{(-\infty, 0)}(x) & = \frac{1}{2 \pi} \int_{-\infty}^\infty e^{i x z} \biggl(\frac{1}{\xi_1 + i z} + \frac{1}{\xi_2 - i z}\biggr) dz ,
}
we find that
\formula{
 \frac{\kappa^+(\tau_1, \xi_1) \kappa^-(\tau_1, \xi_2)}{\kappa^+(\tau_2, \xi_1) \kappa^-(\tau_2, \xi_2)} & = \exp\biggl(\frac{1}{2 \pi} \int_{-\infty}^\infty \biggl(\frac{1}{\xi_1 + i z} + \frac{1}{\xi_2 - i z}\biggr) \log \frac{\tau_1 + f(z)}{\tau_2 + f(z)} \, dz\biggr) ,
}
and by setting $\tau_1 = \tau$ and considering the limit as $\tau_2 \to \infty$, we obtain
\formula[eq:bd:pm]{
 & \kappa^\circ(\tau) \kappa^+(\tau, \xi_1) \kappa^-(\tau, \xi_2) \\
 & \hspace{3em} = \frac{1}{1 + f(0)} \, \exp\biggl(\frac{1}{2 \pi} \int_{-\infty}^\infty \biggl(\frac{1}{\xi_1 + i z} + \frac{1}{\xi_2 - i z}\biggr) \log(\tau + f(z)) dz\biggr) ;
}
we omit the details and refer to the proof of Proposition~\ref{prop:r:bd} for an analogous argument. These are the main expressions that we will work with.

\subsection{Main idea of the proof}

Our goal is to mimic the argument used in~\cite{bib:kmr13}, which can be summarised as follows. Suppose that $X_t$ is symmetric, so that $f(\xi)$ is real-valued for $\xi \in \R$, and assume additionally that $f(\xi)$ is an increasing and differentiable function of $\xi$ on $(0, \infty)$. In this case, integrating by parts in~\eqref{eq:bd:p}, we obtain
\formula{
 \frac{\kappa^+(\tau, \xi_1)}{\kappa^+(\tau, \xi_2)} & = \exp\biggl(\frac{1}{2 \pi i} \int_{-\infty}^\infty \frac{f'(z)}{\tau + f(z)} \, \log \frac{\xi_1 + i z}{\xi_2 + i z} \, dz\biggr) \\
 & = \exp\biggl(\frac{1}{\pi} \int_0^\infty \frac{f'(z)}{\tau + f(z)} \, \Arg \frac{\xi_1 + i z}{\xi_2 + i z} \, dz\biggr) \\
 & = \exp\biggl(\frac{1}{\pi} \int_{f(0^+)}^{f(\infty^-)} \frac{1}{\tau + r} \, \Arg \frac{\xi_1 + i f^{-1}(r)}{\xi_2 + i f^{-1}(r)} \, dr\biggr) ,
}
where $f^{-1}(r)$ denotes inverse function of $f(\xi)$ on $(0, \infty)$. It is well-known that the right-hand side defines a complete Bernstein function of $\tau$ if $0 \le \xi_1 \le \xi_2$, and the essential part of Theorem~\ref{thm:main} follows in the symmetric case.

The above approach does not work for asymmetric processes, because then $f(\xi)$ takes complex values. For this reason, we restrict our attention to processes with completely monotone jumps, discussed in detail in the next section. In this case $f(\xi)$ has a holomorphic extension to $\C \setminus i \R$, and there is a unique line $\Gamma_f$ along which $f(\xi)$ takes real values. Our strategy is to deform the contour of integration in~\eqref{eq:bd:p} to $\Gamma_f$ and only then integrate by parts.

Implementation of the above plan requires a detailed study of the class of characteristic exponents of Lévy processes with completely monotone jumps: the \emph{Rogers functions}. Definitions and basic properties of Rogers functions are studied in the next section. Section~\ref{sec:real} contains a detailed analysis of the \emph{spine} $\Gamma_f$ of a Rogers function $f$, while Section~\ref{sec:wh} provides various expressions for the Wiener--Hopf factors of a Rogers function. Our main result, Theorem~\ref{thm:main}, is proved in Section~\ref{sec:xwh}. Finally, Section~\ref{sec:proof} contains a rather technical proof of one of intermediate results in Section~\ref{sec:real}.

%
%

\section{Lévy processes with completely monotone jumps and Rogers functions}
\label{sec:rogers}

Some of the properties discussed below are not used in the proof of Theorem~\ref{thm:main}. However, we gather them here to facilitate referencing in forthcoming works.

\subsection{Definition of Rogers functions}
\label{sec:rogers:def}

Recall that a function $\nu(x)$ on $(0, \infty)$ is said to be \emph{completely monotone} if $(-1)^n \nu^{(n)}(x) \ge 0$ for all $x > 0$ and $n = 0, 1, 2, \ldots\,$ By Bernstein's theorem, $\nu(x)$ is completely monotone on $(0, \infty)$ if and only if it is the Laplace transform of a non-negative Borel measure on $(0, \infty)$, known as the \emph{Bernstein measure} of $\nu(x)$. The following class of Lévy processes appears to have been first studied by Rogers in~\cite{bib:r83}.

\begin{definition}
\label{def:cm}
A Lévy process $X_t$ has \emph{completely monotone jumps} if the Lévy measure $\nu(dx)$ of $X_t$ is absolutely continuous with respect to the Lebesgue measure, and there is a density function $\nu(x)$ such that $\nu(x)$ and $\nu(-x)$ are completely monotone functions of~$x$ on $(0, \infty)$.
\end{definition}

We propose the name \emph{Rogers functions} for the class of characteristic exponents of Lévy processes with completely monotone jumps, possibly killed at a uniform rate. Among numerous equivalent characterisations of this class, we take the following one as the definition.

\begin{definition}
\label{def:rogers}
A function $f(\xi)$ holomorphic in the right complex half-plane $\hp = \{\xi \in \C : \re \xi > 0\}$ is a \emph{Rogers function} if $\re (f(\xi) / \xi) \ge 0$ for all $\xi \in \hp$.
\end{definition}

Recall that a function $f(\xi)$ holomorphic in $\hp$ is a \emph{Nevanlinna--Pick function} if and only if $\re f(\xi) \ge 0$ for all $\xi \in \hp$. Therefore, $f(\xi)$ is a Rogers function if and only if $f(\xi) / \xi$ is a Nevanlinna--Pick function. Despite this close relation between the classes of Rogers and Nevanlinna--Pick functions, we believe the former one deserves a separate name. The reasons are twofold: key properties of Rogers functions needed in the proof of our main theorem are not shared by Nevanlinna--Pick functions, and the proposed new name allows for more compact statements of our results.

The following theorem provides four equivalent definitions of a Rogers function. Its proof is a mixture of standard arguments from the theory of complete Bernstein and Stieltjes functions (as in~\cite{bib:ssv10} or~\cite{bib:k11}) and the argument given in~\cite{bib:r83} (see formulae~(14)--(18) therein).

\begin{theorem}
\label{thm:rogers}
Suppose that $f(\xi)$ is a continuous function on $\R$, satisfying $f(-\xi) = \overline{f(\xi)}$ for all $\xi \in \R$. The following conditions are equivalent:
\begin{enumerate}[label=\rm (\alph*)]
\item\label{it:r:a}
$f(\xi)$ extends to a Rogers function;
\item\label{it:r:b}
$f(\xi)$ is the characteristic exponent of a Lévy process with completely monotone jumps, possibly killed at a uniform rate;
\item\label{it:r:c}
we have
\formula[eq:r:int]{
 f(\xi) & = a \xi^2 - i b \xi + c + \frac{1}{\pi} \int_{\R \setminus \{0\}} \biggl(\frac{\xi}{\xi + i s} + \frac{i \xi \sign s}{1 + |s|}\biggr) \frac{\mu(ds)}{|s|}
}
for all $\xi \in \R$, where $a \ge 0$, $b \in \R$, $c \ge 0$ and $\mu(ds)$ is a Borel measure on $\R \setminus \{0\}$ such that $\int_{\R \setminus \{0\}} |s|^{-3} \min\{1, s^2\} \mu(ds) < \infty$;
\item\label{it:r:d}
either $f(\xi) = 0$ for all $\xi \in \R$ or
\formula[eq:r:exp]{
 f(\xi) & = c \exp\biggl(\frac{1}{\pi} \int_{-\infty}^\infty \biggl(\frac{\xi}{\xi + i s} - \frac{1}{1 + |s|}\biggr) \frac{\ph(s)}{|s|} \, ds\biggr)
}
for all $\xi \in \R$, where $c > 0$ and $\ph(s)$ is a Borel function on $\R$ with values in $[0, \pi]$.
\end{enumerate}
\end{theorem}

\begin{remark}
\label{rem:rogers}
\begin{enumerate}[label=\rm (\alph*)]
\item
In Theorem~\ref{thm:rogers}\ref{it:r:a}, $f$ extends to a holomorphic function in $\C \setminus i \R$, satisfying $f(-\overline{\xi}) = \overline{f(\xi)}$. Further extensions are sometimes given by~\eqref{eq:r:int} or~\eqref{eq:r:exp}. In particular, formula~\eqref{eq:r:exp} defines a holomorphic function on the set
\formula[eq:r:dom]{
 \dom_f & = \C \setminus (-i \esssupp \ph) ,
}
where $\esssupp \ph$ denotes the essential support of $\ph$. We call $\dom_f$ the \emph{domain} of $f$, and we keep the notation $\dom_f$ throughout the article. Note that if $s \in \R \setminus \esssupp \ph$, then $f(-i s)$ is a positive real number, and also formula~~\eqref{eq:r:int} extends to $\dom_f$.

We always identify the function $f$ (defined originally on $\R$, or even on $(0, \infty)$) and its holomorphic extension given by~\eqref{eq:r:exp} to the set $\dom_f$. We remark that this is the maximal holomorphic extension of $f$ which takes values in $\C \setminus (-\infty, 0]$. However, $f$ may extend to a holomorphic function in an even larger set. For example, if $f(\xi) = \xi^2$, then $\ph(s) = \pi$ for almost all $s \in \R$ and $\dom_f = \C \setminus i \R$, despite the fact that $f$ extends to an entire function.
\item
In Theorem~\ref{thm:rogers}\ref{it:r:b}, $f(\xi)$ has the representation~\eqref{eq:lk}, where $\nu(dx)$ has a density function $\nu(x)$ with respect to the Lebesgue measure, and $\nu(x)$ and $\nu(-x)$ are completely monotone functions of $x$ on $(0, \infty)$.
\item
The measures $\ind_{(0, \infty)}(s) \mu(ds)$ and $\ind_{(0, \infty)}(s) \mu(-ds)$ in Theorem~\ref{thm:rogers}\ref{it:r:c} are Bernstein measures of the completely monotone functions $\nu(x)$ and $\nu(-x)$ mentioned above. Furthermore, the constants $a$, $b$ and $c$ in~\eqref{eq:r:int} agree with those in~\eqref{eq:lk}.
\item
The correspondence between Rogers functions $f(\xi)$ and quadruplets $(a, b, c, \mu)$ satisfying the conditions of Theorem~\ref{thm:rogers}\ref{it:r:c} (or $(a, b, c, \nu)$ as in Theorem~\ref{thm:rogers}\ref{it:r:b}) is a bijection. Similarly, every non-zero Rogers function $f$ corresponds to a unique pair $(c, \ph)$ as in Theorem~\ref{thm:rogers}\ref{it:r:d}, if we agree to identify functions $\ph$ that are equal almost everywhere.
\item
In Theorem~\ref{thm:rogers}\ref{it:r:c} we may equivalently write, for a fixed $r > 0$,
\formula[eq:r:int:r]{
 f(\xi) & = a \xi^2 - i \tilde{b} \xi + c + \frac{1}{\pi} \int_{\R \setminus \{0\}} \biggl(\frac{\xi}{\xi + i s} + \frac{i \xi s}{r^2 + s^2}\biggr) \frac{\mu(ds)}{|s|}
}
for the same $a$, $c$ and $\mu(ds)$, and some $\tilde{b} \in \R$.
\item
The constants $a$, $b$, $c$ in Theorem~\ref{thm:rogers}\ref{it:r:c} and $\tilde{b}$ above are given by
\formula[eq:r:const]{
 a & = \lim_{\xi \to \infty} \frac{f(\xi)}{\xi^2} \, , & b & = \lim_{\xi \to \infty} \frac{-\im f(\xi)}{\xi} \, , & c & = \lim_{\xi \to 0^+} f(\xi), & \tilde{b} & = \im f(r) \, .
}
The measure $\mu(ds)$ satisfies
\formula[eq:r:mu]{
 \pi c \delta_0(ds) + \frac{\mu(ds)}{|s|} & = \lim_{t \to 0^+} \biggl(\re \frac{f(t - i s)}{t - i s} \, ds\biggr) ,
}
with the vague limit of measures in the right-hand side.
\item
If $f(\xi)$ is not identically zero, then, for almost every $s \in \R$, $-\ph(s) \sign s$ in Theorem~\ref{thm:rogers}\ref{it:r:d} is the non-tangential limit of $\Arg f(\xi)$ at $\xi = -i s$. In particular,
\formula[eq:r:ph]{
 -\ph(s) \sign s & = \lim_{t \to 0^+} \Arg f(t - i s) = \lim_{\eps \to 0^+} \Arg f(|s| e^{-i (\pi/2 - \eps) \sign s})
}
for almost all $s \in \R$.
\item Formula~\eqref{eq:r:int} will be referred to as \emph{Stieltjes representation} of~$f(\xi)$, while~\eqref{eq:r:exp} will be called the \emph{exponential representation} of~$f(\xi)$.
\end{enumerate}
\end{remark}

\begin{proof}
Suppose that condition~\ref{it:r:b} holds. By Bernstein's theorem, there is a unique non-negative Borel measure $\mu(ds)$ on $\R \setminus \{0\}$ such that for $x > 0$, $\nu(x)$ and $\nu(-x)$ are the Laplace transforms of $\ind_{(0, \infty)}(s) \mu(ds)$ and $\ind_{(0, \infty)}(s) \mu(-ds)$, respectively. It is easy to see that the integrability condition $\int_{-\infty}^\infty \min
\{1, x^2\} \nu(x) dx < \infty$ translates into $\int_{\R \setminus \{0\}} |s|^{-3} \min\{1, s^2\} \mu(ds) < \infty$, and furthermore~\eqref{eq:r:int} is equivalent to~\eqref{eq:lk}. This proves condition~\ref{it:r:c}. A very similar argument proves the converse, and thus conditions~\ref{it:r:b} and~\ref{it:r:c} are equivalent.

It is immediate to check that formula~\eqref{eq:r:int} defines a Rogers function. Conversely, if $f(\xi)$ is a Rogers function, then $\re(f(\xi) / \xi)$ is a non-negative harmonic function in $\hp$. Thus, by Herglotz's theorem, with $\xi = x + i y$ and $x > 0$, $y \in \R$, we have
\formula{
 \re \frac{f(\xi)}{\xi} & = a x + \frac{1}{\pi} \int_{\R} \frac{x}{x^2 + (y + s)^2} \, \tilde{\mu}(ds) = a \re \xi + \frac{1}{\pi} \int_{\R} \re \frac{1}{\xi + i s} \, \tilde{\mu}(ds)
}
for some $a \ge 0$ and some non-negative Borel measure $\tilde{\mu}(ds)$ on $\R$ which satisfies $\int_{\R} \min\{1, s^{-2}\} \tilde{\mu}(ds) < \infty$. It follows that for some $b \in \R$,
\formula{
 \frac{f(\xi)}{\xi} & = a \xi - i b + \frac{1}{\pi} \int_{\R} \biggl(\frac{1}{\xi + i s} + \frac{i \sign s}{1 + |s|}\biggr) \tilde{\mu}(ds) .
}
This implies~\eqref{eq:r:int}, with $\mu(ds) = |s| \tilde{\mu}(ds)$ and $c = \tfrac{1}{\pi} \tilde{\mu}(\{0\})$, and so conditions~\ref{it:r:a} and~\ref{it:r:c} are equivalent. Additionally, Herglotz's theorem asserts that
\formula{
 \tilde{\mu}(ds) & = \lim_{t \to 0^+} \re \frac{f(t - i s)}{t - i s} \, ds ,
}
which leads to~\eqref{eq:r:mu}. Formulae~\eqref{eq:r:const} follow easily by dominated convergence theorem.

Finally, it is equally simple to see that formula~\eqref{eq:r:exp} defines a Rogers function. Conversely, if $f(\xi)$ is a Rogers function and $f(\xi)$ is not identically equal to zero, then $\Arg f(\xi)$ is a bounded harmonic function in $\hp$, which takes values in $[-\pi, \pi]$. By Poisson's representation theorem, with $\xi = x + i y$ and $x > 0$, $y \in \R$, we have
\formula{
 \Arg f(\xi) & = \frac{1}{\pi} \int_{\R} \frac{x}{x^2 + (y + s)^2} \, \tilde{\phi}(s) ds = \frac{1}{\pi} \int_{-\infty}^\infty \im \frac{1}{\xi + i s} \, \tilde{\ph}(s) ds
}
for a Borel function $\tilde{\ph}(s)$ on $\R$ with values in $[-\pi, \pi]$. As in the previous part of the proof, it follows that for some $b \in \R$,
\formula{
 \log f(\xi) & = b + \frac{1}{\pi} \int_{\R} \biggl(\frac{i}{\xi + i s} - \frac{\sign s}{1 + |s|}\biggr) \tilde{\ph}(s) ds \\
 & = b - \frac{1}{\pi} \int_{\R} \biggl(\frac{\xi}{\xi + i s} - \frac{1}{1 + |s|}\biggr) \frac{\tilde{\ph}(s)}{s} \, ds ,
}
which is equivalent to~\eqref{eq:r:exp} with $c = e^b$ and $\ph(s) = -\tilde{\ph}(s) \sign s$. Furthermore,
\formula{
 \tilde{\ph}(s) & = \lim_{t \to 0^+} \Arg f(t - i s) = -\frac{\pi}{2} \, \sign s + \lim_{t \to 0^+} \Arg \frac{f(t - i s)}{t - i s}
}
for almost all $s \in \R$. Since $\Arg (f(\xi) / \xi) \in [-\tfrac{\pi}{2}, \tfrac{\pi}{2}]$, we conclude that $\tilde{\ph}(s) \in [-\pi, 0]$ for $s > 0$ and $\tilde{\ph}(s) \in [0, \pi]$ for $s > 0$. Equivalence of conditions~\ref{it:r:a} and~\ref{it:r:d} follows, and the proof is complete.
\end{proof}

Suppose that $f(\xi)$ is a Rogers function. Then $\re (f(\xi) / \xi)$ is non-negative and harmonic in $\hp$, and hence it is either everywhere positive or identically equal to $0$. In the former case, $f(\xi)$ is said to be \emph{non-degenerate}; otherwise, $f(\xi) = -i b \xi$ for some $b \in \R$, and $f(\xi)$ is said to be \emph{degenerate}. In particular, either $f(\xi) \ne 0$ for all $\xi \in \hp$, in which case we say that $f(\xi)$ is \emph{non-zero}, or $f(\xi)$ is identically zero in $\hp$. A non-zero Rogers function corresponds to a non-constant Lévy process (with completely monotone jumps), while a non-degenerate Rogers function is the characteristic exponent of a non-deterministic Lévy process.

We introduce two additional classes of Rogers functions. A Rogers function is said to be \emph{bounded} if it is a bounded function on $(0, \infty)$; note that a bounded Rogers function typically fails to be a bounded function on $\hp$. Furthermore, a Rogers function $f(\xi)$ is said to be \emph{symmetric} if $f(\overline{\xi}) = \overline{f(\xi)}$ for $\xi \in \hp$ (or, equivalently, $f(\xi)$ is real-valued for $\xi > 0$). Bounded and symmetric Rogers functions correspond to compound Poisson processes and symmetric Lévy processes, respectively.

Noteworthy, if the measure $\mu$ in Theorem~\ref{thm:rogers}\ref{it:r:b} is purely atomic, with atoms forming a discrete subset of $\R$, then $f$ is meromorphic, and it is the characteristic exponent of a meromorphic Lévy process, studied in detail in~\cite{bib:kkp13}.

\subsection{Complete Bernstein and Stieltjes functions}
\label{sec:cbf}

Recall that a function $f(\xi)$ holomorphic in $\C \setminus (-\infty, 0]$ is a \emph{complete Bernstein function} if and only if $f(\xi) \in [0, \infty)$ for $\xi > 0$ and $\im f(\xi) \ge 0$ when $\im \xi > 0$. Similarly, a function $f(\xi)$ holomorphic in $\C \setminus (-\infty, 0]$ is a \emph{Stieltjes function} if and only if $f(\xi) \in [0, \infty)$ for $\xi > 0$ and $\im f(\xi) \le 0$ when $\im \xi > 0$. In this section we recall some standard properties of these classes of functions.

\begin{theorem}[{see~\cite[Chapter~6]{bib:ssv10} and~\cite[Section~2.5]{bib:k11}}]
\label{thm:cbf}
Let $f(\xi)$ be a non-negative function on $(0, \infty)$. The following conditions are equivalent:
\begin{enumerate}[label=\rm (\alph*)]
\item\label{it:cbf:a}
$f(\xi)$ extends to a complete Bernstein function;
\item\label{it:cbf:b}
$f(\xi)$ is the characteristic (Laplace) exponent of a non-negative Lévy process with completely monotone jumps, possibly killed at a uniform rate; that is,
\formula[eq:cbf:lk]{
 f(\xi) = b \xi + c + \int_0^\infty (1 - e^{-\xi x}) \nu(x) dx
}
for all $\xi > 0$, where $b, c \ge 0$ and $\nu(x)$ is a completely monotone function on $(0, \infty)$ such that $\int_0^\infty \min\{1, x\} \nu(x) dx < \infty$;
\item\label{it:cbf:c}
we have
\formula[eq:cbf:int]{
 f(\xi) & = b \xi + c + \frac{1}{\pi} \int_{(0, \infty)} \frac{\xi}{\xi + s} \, \frac{\mu(ds)}{s}
}
for all $\xi > 0$, where $b, c \ge 0$ and $\mu(ds)$ is a non-negative Borel measure on $(0, \infty)$ such that $\int_{(0, \infty)} s^{-2} \min\{1, s\} \mu(ds) < \infty$;
\item\label{it:cbf:d}
either $f(\xi) = 0$ for all $\xi > 0$ or
\formula[eq:cbf:exp]{
 f(\xi) & = c \, \exp\biggl(\frac{1}{\pi} \int_0^\infty \biggl(\frac{\xi}{\xi + s} - \frac{1}{1 + s}\biggr) \frac{\ph(s)}{s} \, ds\biggr)
}
for all $\xi > 0$, where $c > 0$ and $\ph(s)$ is a Borel function on $(0, \infty)$ with values in $[0, \pi]$.
\end{enumerate}
\end{theorem}

\begin{theorem}[{see~\cite[Chapter~2 and Theorem~6.2]{bib:ssv10}}]
\label{thm:s}
Let $f(\xi)$ be a non-negative function on $(0, \infty)$. The following conditions are equivalent:
\begin{enumerate}[label=\rm (\alph*)]
\item\label{it:s:a}
$f(\xi)$ extends to a Stieltjes function;
\item\label{it:s:b}
$f(\xi)$ is, up to addition by a non-negative constant, the Laplace transform of a completely monotone function on $(0, \infty)$; that is,
\formula[eq:s:lk]{
 f(\xi) = c + \int_0^\infty e^{-\xi x} \nu(x) dx
}
for all $\xi > 0$, where $c \ge 0$ and $\nu(x)$ is a completely monotone function on $(0, \infty)$, locally integrable near $0$;
\item\label{it:s:c}
we have
\formula[eq:s:int]{
 f(\xi) & = \frac{b}{\xi} + c + \frac{1}{\pi} \int_{(0, \infty)} \frac{1}{\xi + s} \, \mu(ds)
}
for all $\xi > 0$, where $b, c \ge 0$ and $\mu(ds)$ is a non-negative Borel measure on $(0, \infty)$ such that $\int_{(0, \infty)} \min\{1, s^{-1}\} \mu(ds) < \infty$;
\item\label{it:s:d}
either $f(\xi) = 0$ for all $\xi > 0$ or
\formula[eq:s:exp]{
 f(\xi) & = c \, \exp\biggl(\frac{1}{\pi} \int_0^\infty \biggl(\frac{1}{\xi + s} - \frac{1}{1 + s}\biggr) \ph(s) ds\biggr)
}
for all $\xi > 0$, where $c > 0$ and $\ph(s)$ is a Borel function on $(0, \infty)$ with values in $[0, \pi]$.
\end{enumerate}
\end{theorem}

\begin{remark}
An analogue of Remark~\ref{rem:rogers} applies to Theorems~\ref{thm:cbf} and~\ref{thm:s}. In particular, as it was the case with Rogers functions, we always identify a complete Bernstein function $f(\xi)$ with its holomorphic extension to $\dom_f = \C \setminus (-\esssupp \ph)$, given by the exponential representation~\eqref{eq:cbf:exp}.
\end{remark}

\begin{proposition}[{\cite[Proposition~7.1 and Theorem~7.3]{bib:ssv10}}]
\label{prop:cbf:s}
For a non-zero function $f(\xi)$ holomorphic in $\C \setminus (-\infty, 0]$, the following conditions are equivalent: $f(\xi)$ is a complete Bernstein function; $f(\xi) / \xi$ is a Stieltjes functions; $1 / f(\xi)$ is a Stieltjes function; $\xi / f(\xi)$ is a complete Bernstein function.
\end{proposition}

\begin{proposition}[{see \cite[Theorem~6.2]{bib:ssv10} and Proposition~\ref{prop:cbf:s}}]
\label{prop:cbf:arg}
A non-zero function $f(\xi)$ holomorphic in $\C \setminus (-\infty, 0]$ is a complete Bernstein function if and only if $0 \le \Arg f(\xi) \le \Arg \xi$ whenever $\im \xi > 0$. A non-zero function $f(\xi)$ holomorphic in $\C \setminus (-\infty, 0]$ is a Stieltjes function if and only if $0 \ge \Arg f(\xi) \ge -\Arg \xi$ whenever $\im \xi > 0$.
\end{proposition}

\begin{proposition}[{see Theorems~\ref{thm:cbf} and~\ref{thm:s}}]
\label{prop:cbf:s:bounded}
If $f(\xi)$ is a complete Bernstein function, $C > 0$ and $f(\xi) \le C$ for all $\xi \in (0, \infty)$, then $C - f(\xi)$ is a Stieltjes function. Converesly, if $f(\xi)$ is a Stieltjes function, $C > 0$ and $f(\xi) \le C$ for all $\xi \in (0, \infty)$, then $C - f(\xi)$ is a complete Bernstein function.
\end{proposition}

\subsection{Basic properties of Rogers functions}

The following three results are direct consequences of the definition of a Rogers function, and we omit their proofs.

\begin{proposition}
\label{prop:r:arg}
A non-zero function $f(\xi)$ holomorphic in $\hp$ is a Rogers function if and only if $-\tfrac{\pi}{2} + \Arg \xi \le \Arg f(\xi) \le \tfrac{\pi}{2} + \Arg \xi$ for all $\xi \in \hp$.
\end{proposition}

\begin{proposition}
\label{prop:r:prop}
For all Rogers functions $f(\xi)$ and $g(\xi)$:
\begin{enumerate}[label=\rm (\alph*)]
\item\label{it:r:prop:a} $\xi^2 f(1 / \xi)$ is a Rogers function;
\item\label{it:r:prop:b} $\xi^2 / f(\xi)$ and $1 / f(1 / \xi)$ are Rogers functions if $f(\xi)$ is non-zero;
\item\label{it:r:prop:c} $\xi^{1 - \alpha} f(\xi^\alpha)$ is a Rogers function if $\alpha \in [-1, 1]$;
\item\label{it:r:prop:d} $g(\xi) f(\xi / g(\xi))$ is a Rogers function if $g(\xi)$ is non-zero;
\item\label{it:r:prop:e} $(f(\xi))^\alpha (g(\xi))^{1 - \alpha}$ is a Rogers function if $\alpha \in [0, 1]$;
\item\label{it:r:prop:f} $((f(\xi))^\alpha + (g(\xi))^\alpha)^{1 / \alpha}$ is a Rogers function if $\alpha \in [-1, 1] \setminus \{0\}$ and $f(\xi)$ and $g(\xi)$ are non-zero;
\item\label{it:r:prop:g} $a f(b \xi) + c$ is a Rogers function if $a, b, c \ge 0$.
\end{enumerate}
\end{proposition}

\begin{proposition}
\label{prop:r:composition}
If $f(\xi)$ is a Rogers function and $g(\xi)$ is a complete Bernstein function, then $g(f(\xi))$ is a Rogers function.
\end{proposition}

\begin{proposition}
\label{prop:r:limit}
If $f(\xi)$ is a Rogers function, then the limit $f(0^+) = \lim_{\xi \to 0^+} f(\xi)$ exists. More precisely, if $f(\xi)$ has Stieltjes representation~\eqref{eq:r:int}, then $f(0^+) = c$, and if $f(\xi)$ has the exponential representation~\eqref{eq:r:exp}, then
\formula[eq:r:limit:zero]{
 f(0^+) & = c \exp \biggl( -\frac{1}{\pi} \int_0^\infty \frac{1}{1 + |s|} \, \frac{\ph(s)}{|s|} \, ds \biggr) ,
}
where we understand that $\exp(-\infty) = 0$. Similarly, $f(\infty^-) = \lim_{\xi \to \infty} f(\xi)$ exists, and if $f(\xi)$ has the exponential representation~\eqref{eq:r:exp}, then
\formula[eq:r:limit:inf]{
 f(\infty^-) & = c \exp \biggl( \frac{1}{\pi} \int_0^\infty \frac{|s|}{1 + |s|} \, \frac{\ph(s)}{|s|} \, ds \biggr) ,
}
where we understand that $\exp(\infty) = \infty$
\end{proposition}

\begin{proof}
With the notation of~\eqref{eq:r:int}, we have
\formula{
 f(0^+) & = \lim_{\xi \to 0^+} \re \biggl(a \xi^2 - i b \xi + c + \frac{1}{\pi} \int_{\R \setminus \{0\}} \biggl(\frac{\xi}{\xi + i s} + \frac{i \xi \sign s}{1 + |s|}\biggr) \frac{\mu(ds)}{|s|}\biggr) \\
 & = c + \frac{1}{\pi} \lim_{\xi \to 0^+} \int_{\R \setminus \{0\}} \frac{\xi^2}{\xi^2 + s^2} \, \frac{\mu(ds)}{|s|} = c 
}
by the dominated convergence theorem. Existence of $f(\infty^-)$ follows now by the above argument applied to the Rogers function $g(\xi) = 1 / f(1 / \xi)$ (if $f(\xi)$ is non-zero).

With the notation of~\eqref{eq:r:exp}, for $\xi > 0$ we have in a similar way
\formula{
 |f(\xi)| & = c \exp \biggl(\frac{1}{\pi} \int_{-\infty}^\infty \biggl(\frac{\xi^2}{\xi^2 + s^2} - \frac{1}{1 + |s|}\biggr) \frac{\ph(s)}{|s|} \, ds\biggr) .
}
Formula~\eqref{eq:r:limit:zero} is obtained by passing to the limit as $\xi \to 0^+$ and using monotone convergence theorem. Formula~\eqref{eq:r:limit:inf} is obtained in a similar way by considering the limit as $\xi \to \infty$.
\end{proof}

One easily checks that a Rogers function $f(\xi)$ is bounded if and only if
\formula[eq:r:bounded]{
 f(\xi) & = c + \frac{1}{\pi} \int_{\R \setminus \{0\}} \frac{\xi}{\xi + i s} \, \frac{\mu(ds)}{|s|}
}
for $\xi \in \hp$, where $c \ge 0$ and $|s|^{-1} \mu(ds)$ is a finite measure. In this case $f(\infty^-) = \lim_{\xi \to \infty} f(\xi) = c + \int_{\R \setminus \{0\}} |s|^{-1} \mu(ds)$ is a finite non-negative number.

\begin{proposition}
\label{prop:r:bounded}
If $f(\xi)$ is a bounded Rogers function and $A \ge f(\infty^-)$, then $A - f(1 / \xi)$ is a (bounded) Rogers function of $\xi$.
\end{proposition}

\begin{proof}
By~\eqref{eq:r:bounded},
\formula{
 A - f(1 / \xi) & = A - c - \frac{1}{\pi} \int_{\R \setminus \{0\}} \frac{1}{1 + i \xi s} \, \frac{\mu(ds)}{|s|} \\
 & = \biggl(A - c - \frac{1}{\pi} \int_{\R \setminus \{0\}} \frac{\mu(ds)}{|s|}\biggr) + \frac{1}{\pi} \int_{\R \setminus \{0\}} \frac{i \xi s}{1 + i \xi s} \, \frac{\mu(ds)}{|s|} \\
 & = (A - f(\infty^-)) + \frac{1}{\pi} \int_{\R \setminus \{0\}} \frac{i \xi s}{1 + i \xi s} \, \frac{\mu(ds)}{|s|} \, .
}
Let $\tilde{\mu}(dt)$ be the push-forward of $\mu(ds)$ by the substitution $s = -1 / t$. We obtain
\formula{
 A - f(1 / \xi) & = (A - f(\infty^-)) + \frac{1}{\pi} \int_{\R \setminus \{0\}} \frac{\xi}{\xi + i t} \, \frac{\tilde{\mu}(dt)}{|t|} \, .
}
By~\eqref{eq:r:bounded}, this proves our claim.
\end{proof}

\begin{remark}\label{rem:r:convergence}
By the corresponding result for Nevanlinna--Pick functions, point-wise limit of a sequence of Rogers function, if finite at some point, is again a Rogers function. In addition, the convergence is uniform on compact subsets of $\hp$.

With the notation of the Stieltjes representation~\eqref{eq:cbf:int} in Theorem~\ref{thm:rogers}, a sequence $f_n(\xi)$ of Rogers functions converges to a finite limit $f(\xi)$ if and only if the corresponding coefficients $b_n$ converge to $b$, and the corresponding measures $|s|^{-3} \min\{1, s^2\} \mu_n(ds) + \tfrac{1}{\pi} c_n \delta_0(ds) + \tfrac{1}{\pi} a_n \delta_\infty(ds)$ converge to $|s|^{-3} \min\{1, s^2\} \mu(ds) + \tfrac{1}{\pi} c \delta_0(ds) + \tfrac{1}{\pi} a \delta_\infty(ds)$ vaguely on $\R \cup \{\infty\}$, the one-point compactification of $\R$. This property follows from the analogous characterisation of convergence of positive harmonic functions $\Re (f_n(\xi) / \xi)$, which have the Poisson integral representation with measure $\tilde{\mu}_n(ds) = |s|^{-1} \mu_n(ds) + \tfrac{1}{\pi} c_n \delta_0(ds)$ and constant $a_n$; see the proof of Theorem~\ref{thm:rogers} for further details.

Yet another equivalent condition for convergence of a sequence $f_n(\xi)$ of Rogers functions is given in terms of the exponential representation~\eqref{eq:cbf:exp} in Theorem~\ref{thm:rogers}: the corresponding coefficients $c_n$ must converge to $c$, and the corresponding measures $\ph_n(s) ds$ must converge to $\ph(s) ds$ vaguely on $\R$. This follows in a similar way from the Poisson integral representation of $\Arg f(\xi)$.

For further discussion, we refer to~\cite{bib:ssv10} and the references therein.
\end{remark}

\subsection{Estimates of Rogers functions}

The following simple estimate of a Rogers function follows from the Stieltjes representation~\eqref{eq:r:int:r}. A more refined result given in Proposition~\ref{prop:r:spine:bound} uses the exponential representation~\eqref{eq:r:exp}.

\begin{proposition}
\label{prop:r:bound}
If $f(\xi)$ is a Rogers function and $r > 0$, then
\formula[eq:r:bound]{
 \frac{1}{\sqrt{2}} \, \frac{|\xi|^2}{r^2 + |\xi|^2} \biggl(\frac{\re \xi}{|\xi|}\biggr) |f(r)| \le |f(\xi)| & \le \sqrt{2} \, \frac{r^2 + |\xi|^2}{r^2} \biggl(\frac{|\xi|}{\re \xi}\biggr) |f(r)|
}
for $\xi \in \hp$.
\end{proposition}

\begin{proof}
Suppose that $f(\xi)$ has Stieltjes representation~\eqref{eq:r:int:r}. Let $s \in \R$ and $\xi \in \hp$. Since $|r^2 + i \xi s|^2 \le (r^2 + |\xi|^2) (r^2 + s^2)$, we have
\formula{
 \abs{\frac{\xi}{\xi + i s} + \frac{i \xi s}{r^2 + s^2}} & = \frac{|\xi| |r^2 + i \xi s|}{|\xi + i s| (r^2 + s^2)} \le \frac{|\xi| \sqrt{r^2 + |\xi|^2}}{|\xi + i s| \sqrt{r^2 + s^2}} \, .
}
Let $\xi = x + i y$ with $x > 0$, $y \in \R$. By a simple calculation,
\formula{
 (x^2 + (y + s)^2)(r^2 + x^2 + y^2) - (r^2 + s^2) x^2 = (x^2 + y (y + s))^2 + (y + s)^2 \ge 0 ,
}
so that
\formula{
 |\xi + i s|^2 & = x^2 + (y + s)^2 \ge \frac{(r^2 + s^2) x^2}{r^2 + |\xi|^2} \, .
}
It follows that
\formula{
 \abs{\frac{\xi}{\xi + i s} + \frac{i \xi s}{r^2 + s^2}} & \le \frac{|\xi| (r^2 + |\xi|^2)}{x (r^2 + s^2)} \, .
}
Finally, $a |\xi|^2 + |\tilde{b}| |\xi| + c \le r^{-2} (r^2 + |\xi|^2) (a r^2 + \tilde{b} r + c)$. Therefore, we obtain
\formula{
 |f(\xi)| & \le \frac{|\xi| (r^2 + |\xi|^2)}{r^2 x} \biggl(a r^2 + |\tilde{b}| r + c + \frac{1}{\pi} \int_{\R \setminus \{0\}} \frac{r^2}{r^2 + s^2} \, \frac{\mu(ds)}{|s|}\biggr) .
}
On the other hand,
\formula{
 \frac{r}{r + i s} + \frac{i r s}{r^2 + s^2} & = \frac{r^2}{r^2 + s^2} ,
}
so that
\formula{
 \re f(r) + |\im f(r)| = a r^2 + |\tilde{b}| r + c + \frac{1}{\pi} \int_{\R \setminus \{0\}} \frac{r^2}{r^2 + s^2} \, \frac{\mu(ds)}{|s|} \, .
}
Since $\re f(r) + |\im f(r)| \le \sqrt{2} |f(r)|$, the upper bound for $|f(\xi)|$ follows. The lower bound for $|f(\xi)|$ is a consequence of the upper bound for the Rogers function $\xi^2 / f(\xi)$ (if $f(\xi)$ is non-zero).
\end{proof}

Suppose that $f(\xi)$ is a non-zero Rogers function with exponential representation~\eqref{eq:r:exp}. The derivative of the integrand in~\eqref{eq:r:exp} with respect to $\xi$ is equal to $i \ph(s) \sign s / (\xi + i s)^2$ and it is bounded by an absolutely integrable function of $s$ whenever $\xi$ is restricted to a compact subset of $\dom_f$. It follows that the expression for $\log f(\xi)$ can be differentiated under the integral sign. Thus,
\formula[eq:r:prime:exp]{
 \frac{f'(\xi)}{f(\xi)} & = (\log f)'(\xi) = \frac{1}{\pi} \int_{-\infty}^\infty \frac{i \sign s}{(\xi + i s)^2} \, \ph(s) ds
}
for all $\xi \in \dom_f$. Since the integral of $|\xi + i s|^{-2}$ over $s \in \R$ is equal to $\pi / \re \xi$, we have
\formula[eq:r:prime]{
 \abs{\frac{f'(\xi)}{f(\xi)}} & \le \frac{\pi}{\re \xi}
}
for $\xi \in \hp$. We will need the following improvement of the above estimate.

\begin{proposition}
\label{prop:r:spine:bound}
If $f$ is a non-zero Rogers function, $\xi \in \dom_f$ and $f(\xi) \in (0, \infty)$, then
\formula[eq:r:spine:prime]{
 \abs{\frac{f'(\xi)}{f(\xi)}} & \le \frac{\pi}{|\xi|} \, ,
}
and, for some $c > 0$,
\formula[eq:r:spine:log]{
 |\log f(\xi)| & \le |\log c| + \sqrt{2 \pi} \, \frac{1 + |\xi|}{\sqrt{|\xi|}} \, .
}
More precisely, $c$ is the constant in the exponential representation~\eqref{eq:r:exp} of $f(\xi)$, and with the notation of~\eqref{eq:r:exp}, we have
\formula[eq:r:spine:bound]{
 \frac{1}{\pi} \int_{-\infty}^\infty \abs{\frac{1}{\xi + i s} - \frac{1}{1 + |s|}} \frac{\ph(s)}{|s|} \, ds \le \sqrt{2 \pi} \, \frac{1 + |\xi|}{\sqrt{|\xi|}} \, .
}
\end{proposition}

\begin{proof}
Since $f(-\overline{\xi}) = \overline{f(\xi)}$, with no loss of generality we may assume that $\re \xi \ge 0$. Let $\xi = r e^{i \alpha}$, where $r = |\xi| > 0$ and $\alpha = \Arg \xi \in [-\tfrac{\pi}{2}, \tfrac{\pi}{2}]$. Then $|\xi + i s|^2 = r^2 + 2 r s \sin \alpha + s^2$ and
\formula{
 \im \frac{\xi}{\xi + i s} & = \frac{\im(|\xi|^2 - i \xi s)}{|\xi + i s|^2} = -\frac{r s \cos \alpha}{r^2 + 2 r s \sin \alpha + s^2} \, .
}
Using the above identity, the exponential representation~\eqref{eq:r:exp} and the identity $\Arg f(\xi) = \im \log f(\xi)$, we obtain
\formula[eq:r:real:arg]{
 \Arg f(r e^{i \alpha}) & = \frac{r \cos \alpha}{\pi} \int_{-\infty}^\infty \frac{-s}{r^2 + 2 r s \sin \alpha + s^2} \, \frac{\ph(s)}{|s|} \, ds .
}
However, $f(r e^{i \alpha}) \in (0, \infty)$, that is, $\Arg f(r e^{i \alpha}) = 0$. Therefore,
\formula{
 0 = \frac{\pi \Arg f(\xi)}{r \cos \alpha} & = \int_{-\infty}^\infty \frac{-s}{r^2 + 2 r s \sin \alpha + s^2} \, \frac{\ph(s)}{|s|} \, ds \\
 & = \int_{-\infty}^0 \frac{1}{r^2 + 2 r s \sin \alpha + s^2} \, \ph(s) ds - \int_0^\infty \frac{1}{r^2 + 2 r s \sin \alpha + s^2} \, \ph(s) ds .
}
It follows that the two integrals in the right-hand side are equal. Our goal is to estimate their sum: by~\eqref{eq:r:prime:exp}, we have
\formula{
 \frac{|f'(\xi)|}{f(\xi)} & \le \frac{1}{\pi} \int_{-\infty}^\infty \frac{1}{|\xi + i s|^2} \, \ph(s) ds = \frac{1}{\pi} \int_{-\infty}^\infty \frac{1}{r^2 + 2 r s \sin \alpha + s^2} \, \ph(s) ds .
}
Since the integrals over $(0, \infty)$ and $(-\infty, 0)$ are equal, it suffices to estimate one of them.

Suppose that $\alpha \ge 0$. Since $0 \le \ph(s) \le \pi$ for all $s \in \R$, we have
\formula{
 \int_0^\infty \frac{1}{r^2 + 2 r s \sin \alpha + s^2} \, \ph(s) ds & \le \pi \int_0^\infty \frac{1}{r^2 + 2 r s \sin \alpha + s^2} ds = \frac{\pi (\tfrac{\pi}{2} - \alpha)}{r \cos \alpha} \le \frac{\pi^2}{2 r} \, .
}
It follows that
\formula{
 \frac{|f'(\xi)|}{f(\xi)} & \le \int_{-\infty}^\infty \frac{1}{r^2 + 2 r s \sin \alpha + s^2} \, \ph(s) ds = 2 \int_0^\infty \frac{1}{r^2 + 2 r s \sin \alpha + s^2} \, \ph(s) ds \le \frac{\pi^2}{r} \, .
}
A similar argument, involving an estimate of the integral over $s \in (-\infty, 0)$, leads to the same bound for $\alpha < 0$, and the proof of~\eqref{eq:r:spine:prime} is complete.

Formula~\eqref{eq:r:spine:log} is a direct consequence of the exponential representation~\eqref{eq:r:exp} and~\eqref{eq:r:spine:bound}. Denote the integral in the left-hand side of~\eqref{eq:r:spine:bound} by $I$. Observe that
\formula{
 I & = \frac{1}{\pi} \int_{-\infty}^\infty \frac{|\xi - i \sign s|}{|\xi + i s| (1 + |s|)} \, \ph(s) \, ds .
}
Clearly, $|\xi - i \sign s| \le 1 + |\xi|$. By Cauchy--Schwarz inequality, we find that
\formula{
 I & \le \frac{1 + |\xi|}{\pi} \biggl(\int_{-\infty}^\infty \frac{1}{|\xi + i s|^2} \, \ph(s) \, ds\biggr)^{1/2} \biggl(\int_{-\infty}^\infty \frac{1}{(1 + |s|)^2} \, \ph(s) \, ds\biggr)^{1/2} .
}
Since $0 \le \ph(s) \le \pi$ for all $s \in \R$, the latter integral does not exceed $2 \pi$. The former one is bounded by $\pi^2 / |\xi|$ by the first part of the proof, and~\eqref{eq:r:spine:bound} follows.
\end{proof}

%
%

\section{Real values of Rogers functions}
\label{sec:real}

\subsection{Spine of a Rogers function}
\label{sec:real:spine}

The curve (or, more generally, the system of curves) along which a Rogers function takes positive real values plays a key role in our development.

\begin{definition}
\label{def:r:curve}
Suppose that $f(\xi)$ is a non-constant Rogers function. We define the \emph{spine} of $f$ by
\formula[eq:r:gamma]{
 \Gamma_f & = \{\zeta \in \hp : f(\zeta) \in (0, \infty)\} .
}
The orientation of each connected component of the spine is chosen in such a way that positive values of $\im f$ lie on the left-hand side of $\Gamma_f$.
\end{definition}

Since $f(\xi)$ does not take values in $(-\infty, 0]$, the spine is the nodal line of the harmonic function $\im f$. The following two theorems are main results of this section.

\begin{theorem}
\label{thm:r:real}
Let $f(\xi)$ be a non-constant Rogers function. There is a unique continuous complex-valued function $\zeta_f(r)$ on $(0, \infty)$ such that the following assertions hold:
\begin{enumerate}[label=\rm (\alph*)]
\item\label{it:r:real:a} We have $|\zeta_f(r)| = r$ and $\Arg \zeta_f(r) \in [-\tfrac{\pi}{2}, \tfrac{\pi}{2}]$ for all $r > 0$.
\item\label{it:r:real:b} If $\xi \in \hp$ and $r = |\xi|$, then
\formula{
 \sign \im f(\xi) & = \sign (\Arg \xi - \Arg \zeta_f(r)) .
}
\item\label{it:r:real:c}
The spine $\Gamma_f$ is the union of pairwise disjoint simple real-analytic curves, which begin and end at the imaginary axis or at complex infinity. Furthermore, $\Gamma_f$ has parameterisation
\formula[eq:r:spine]{
 \Gamma_f & = \{\zeta_f(r) : r \in \zero_f\} , 
}
where
\formula[eq:r:z]{
 \zero_f & = \{r \in (0, \infty) : \Arg \zeta_f(r) \in (-\tfrac{\pi}{2}, \tfrac{\pi}{2})\} .
}
\item\label{it:r:real:d}
For every $r > 0$, the spine $\Gamma_f$ restricted to the annular region $r \le |\xi| \le 2 r$ is a system of rectifiable curves of total length at most $C r$, where one can take $C = 300$. Furthermore, if $\zeta_f(r) = r e^{i \thet(r)}$ for $r \in \zero_f$, then
\formula[eq:r:real:curv]{
 |r (r \thet'(r))'| & \le \frac{9 ((r \thet'(r))^2 + 1)}{\cos \thet(r)}
}
for $r \in \zero_f$.
\end{enumerate}
\end{theorem}

\begin{theorem}
\label{thm:r:lambda}
Suppose that $f(\xi)$ is a non-constant Rogers function.
\begin{enumerate}[label=\rm (\alph*)]
\item\label{it:r:lambda:a} For every $r \in (0, \infty) \setminus \partial \zero_f$ we have $\zeta_f(r) \in \dom_f$.
\item\label{it:r:lambda:b} The function $\lambda_f(r)$, defined for $r \in (0, \infty) \setminus \partial \zero_f$ by
\formula[eq:r:lambda]{
 \lambda_f(r) & = f(\zeta_f(r)) ,
}
extends in a unique way to a continuous, strictly increasing function of $r \in (0, \infty)$, and $\lambda_f'(r) > 0$ for every $r \in (0, \infty) \setminus \partial \zero_f$.
\item\label{it:r:lambda:c} We have $\lambda(0^+) = f(0^+)$, and $\lambda(\infty^-) = f(\infty^-)$.
\end{enumerate}
\end{theorem}

We denote the extension of $\lambda_f(r)$ described in the above result by the same symbol. The notation $\Gamma_f$, $\zero_f$, $\zeta_f(r)$ and $\lambda_f(r)$ is kept throughout the paper. Whenever there is only one Rogers function involved, we omit the subscript $f$ (also in $\dom_f$), except in statements of results.

\begin{figure}
\centering
\begin{tabular}{cccc}
\includegraphics[width=0.22\textwidth]{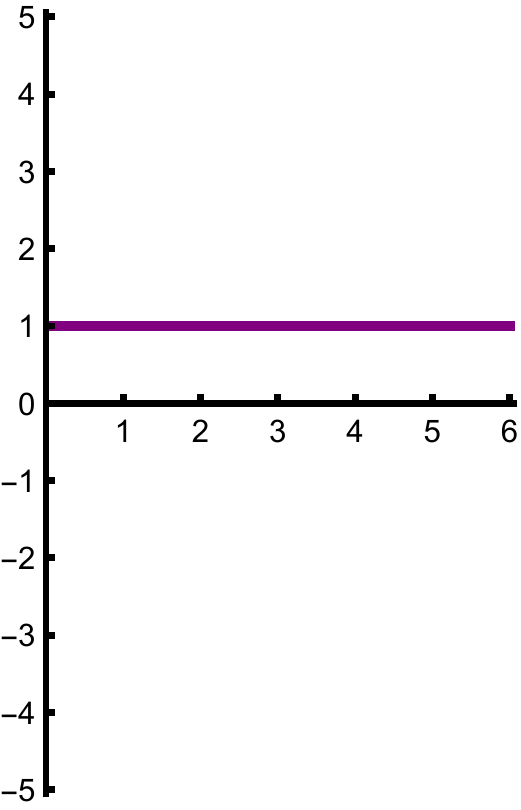}
&
\includegraphics[width=0.22\textwidth]{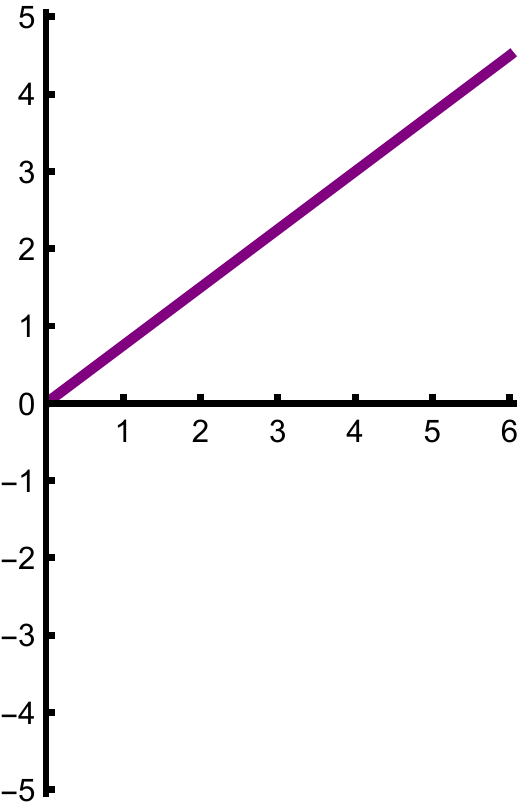}
&
\includegraphics[width=0.22\textwidth]{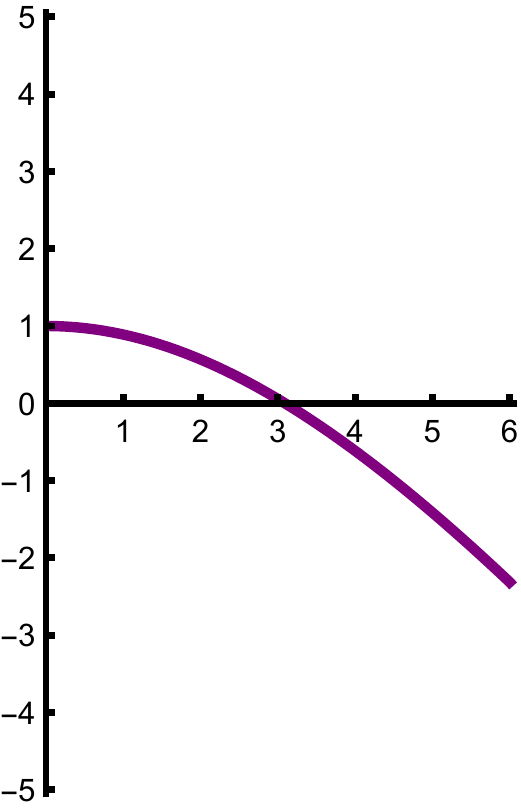}
&
\includegraphics[width=0.22\textwidth]{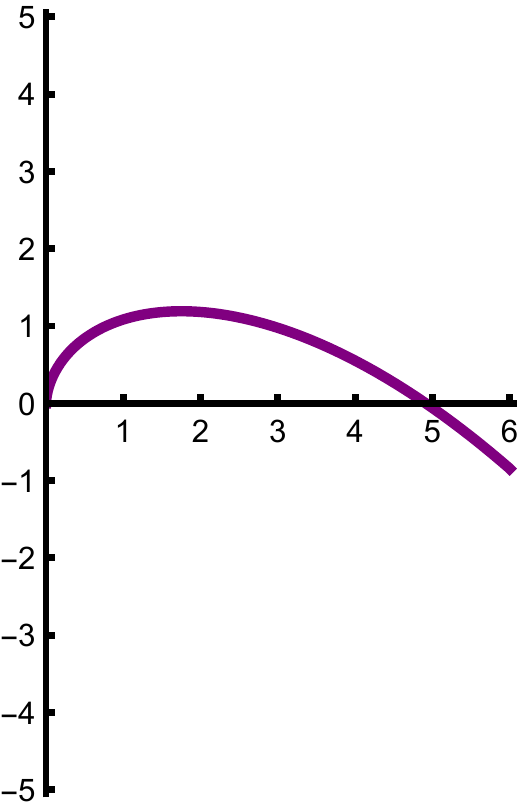}
\\[-0.2em]
(a)
&
(b)
&
(c)
&
(d)
\\[1em]
\includegraphics[width=0.22\textwidth]{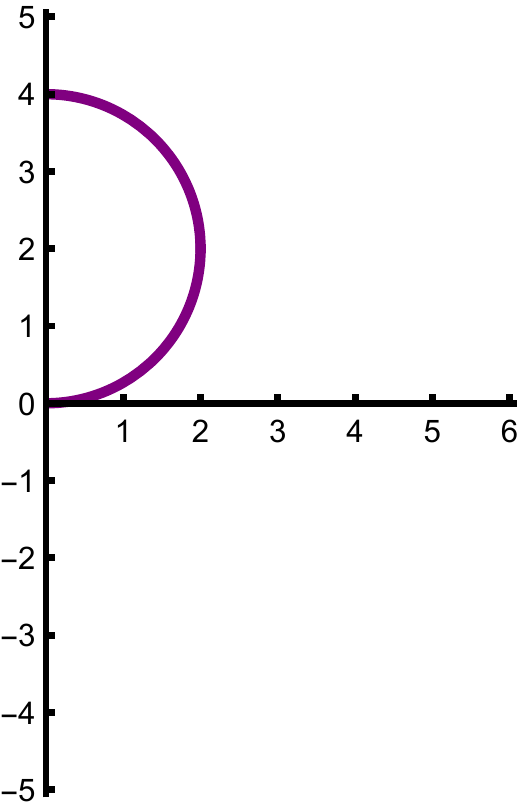}
&
\includegraphics[width=0.22\textwidth]{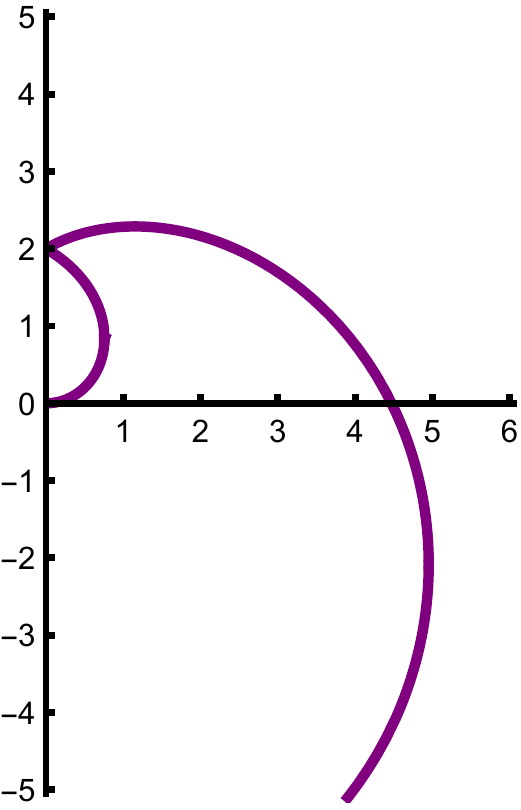}
&
\includegraphics[width=0.22\textwidth]{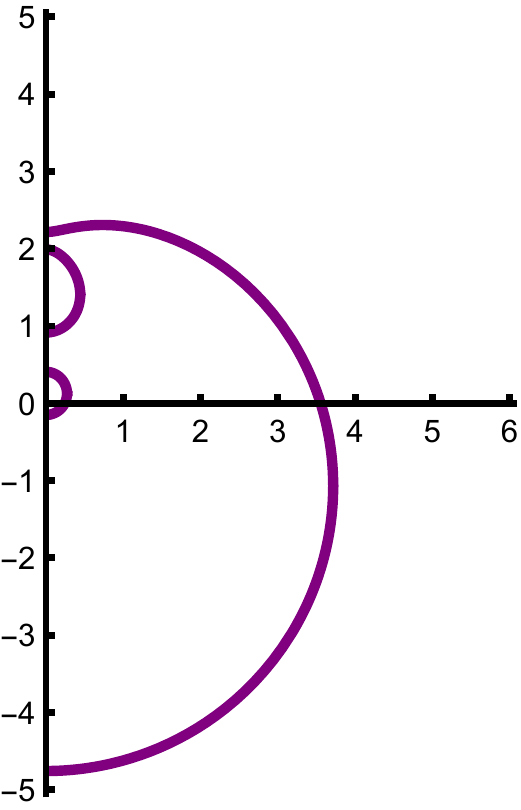}
&
\includegraphics[width=0.22\textwidth]{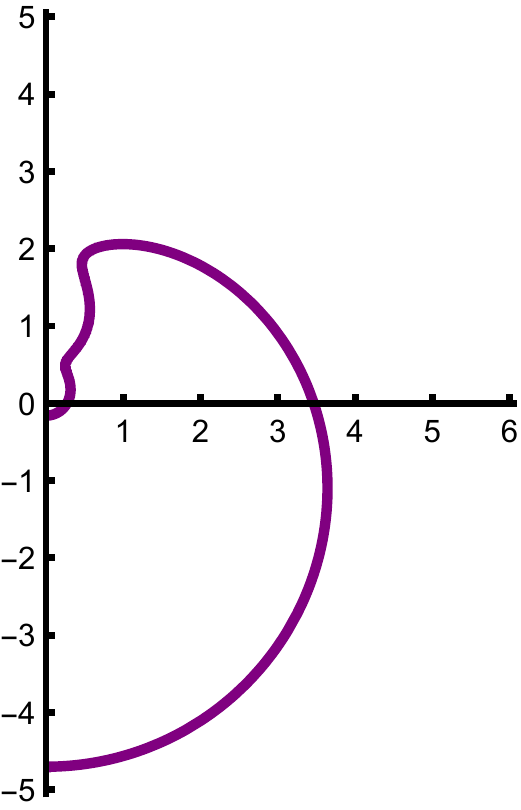}
\\[-0.2em]
(e)
&
(f)
&
(g)
&
(h)
\end{tabular}
\caption{Plot of the spine $\Gamma_f$ for eight Rogers functions $f(\xi)$: \newline
(a)~$f(\xi) = \tfrac{1}{2} \xi^2 - i \xi$ (Brownian motion with drift); \newline
(b)~$f(\xi) = 2 (-i \xi)^{1/2} + (i \xi)^{1/2}$ (strictly stable process); \newline
(c)~$f(\xi) = (-i \xi + 1)^{1/2} + 3 (i \xi + 19)^{1/2}$ (asymmetric tempered stable process); \newline
(d)~$f(\xi) = 8 (-i \xi)^{1/5} + (i \xi)^{4/5}$ (mixed stable process); \newline
(e)~$f(\xi) = \frac{\xi^2}{i \xi + 2}$; \newline
(f)~$f(\xi) = \frac{-i \xi}{-i \xi + 2} \, \frac{i \xi}{i \xi + 0.5} \, (i \xi + 14)$; \newline
(g)~$f(\xi) = \frac{-i \xi + 0.5}{-i \xi + 1} \, \frac{i \xi}{i \xi + 0.05} \, \frac{i \xi + 1.25}{i \xi + 1.26} (i \xi + 30)$; \newline
(h)~$f(\xi) = \frac{-i \xi + 0.5}{-i \xi + 1} \, \frac{i \xi}{i \xi + 0.06} \, \frac{i \xi + 0.95}{i \xi + 0.97} (i \xi + 30)$.
}
\label{fig:r:spine}
\end{figure}

Spines of a sample of Rogers functions are shown in Figure~\ref{fig:r:spine}.

\begin{proof}[Proof of Theorem~\ref{thm:r:real}]
Suppose that $f$ has the exponential representation~\eqref{eq:r:exp}, so that
\formula{
 \Arg f(\xi) & = \im \log f(\xi) = \frac{1}{\pi} \int_{-\infty}^\infty \im \frac{\xi}{\xi + i s} \, \frac{\ph(s)}{|s|} \, ds
}
for $\xi \in \hp$. We use polar coordinates: we write $\xi = r e^{i \alpha}$ with $r = |\xi| > 0$ and $\alpha = \Arg \xi \in (-\tfrac{\pi}{2}, \tfrac{\pi}{2})$. Recall that by~\eqref{eq:r:real:arg},
\formula{
 \Arg f(r e^{i \alpha}) & = \frac{r \cos \alpha}{\pi} \int_{-\infty}^\infty \frac{-s}{r^2 + 2 r s \sin \alpha + s^2} \, \frac{\ph(s)}{|s|} \, ds .
}
For every $r > 0$ and $s \in \R \setminus \{0\}$, the integrand is a strictly increasing function of $\alpha \in (-\tfrac{\pi}{2}, \tfrac{\pi}{2})$. Since $f$ is non-constant, $\ph$ is positive on a set of positive Lebesgue measure. It follows that $\Arg f(r e^{i \alpha}) / \cos \alpha$ is strictly increasing in $\alpha \in (-\tfrac{\pi}{2}, \tfrac{\pi}{2})$.

In particular, there is a unique $\thet(r) \in [-\tfrac{\pi}{2}, \tfrac{\pi}{2}]$ such that $\Arg f(r e^{i \alpha}) < 0$ if $\alpha < \thet(r)$ and $\Arg f(r e^{i \alpha}) > 0$ if $\alpha > \thet(r)$. It is easy to see that $\thet(r)$ is a continuous function of $r > 0$. We set $\zeta(r) = r e^{i \thet(r)}$. Parts~\ref{it:r:real:a} and~\ref{it:r:real:b} of the theorem follow.

Let $\zero$ be the set of those $r > 0$ for which $\thet(r) \in (-\tfrac{\pi}{2}, \tfrac{\pi}{2})$. By part~\ref{it:r:real:b}, the spine of $f(\xi)$ satisfies~\eqref{eq:r:spine}, that is, $\Gamma = \{\zeta(r) : r \in \zero\}$. Since $\Gamma$ is the nodal line of the harmonic function $\im f(\xi)$, it is a union of (at most countably many) simple real-analytic curves. These curves necessarily begin and end at the imaginary axis or converge to complex infinity, and part~\ref{it:r:real:b} asserts that they do not intersect each other. This completes the proof of part~\ref{it:r:real:c}.

The proof of item~\ref{it:r:real:d} is rather long and technical, and it is deferred to Section~\ref{sec:proof}.
\end{proof}

\begin{proof}[Proof of Theorem~\ref{thm:r:lambda}]
Suppose that $f(\xi)$ has exponential representation~\eqref{eq:r:exp}. Clearly, $\zeta(r) \in \hp \sub \dom$ for $r \in \zero$. Observe that $(0, \infty) \setminus \partial \zero = \zero \cup \Int ((0, \infty) \setminus \zero)$, where $\Int$ denotes the interior of a set. Therefore, in order to prove part~\ref{it:r:lambda:a}, it remains to show that $\zeta(r) \in \dom$ for every $r \in \Int ((0, \infty) \setminus \zero)$.

Whenever $s \notin \zero$, we have either $\zeta(s) = i s$ or $\zeta(s) = -i s$. Let $r \in \Int ((0, \infty) \setminus \zero)$, and suppose that $\zeta(r) = i r$. By the definition of $r$ and continuity of $\zeta(s)$, we have $\zeta(s) = i s$ for $s$ in some neighbourhood of~$r$. From~\eqref{eq:r:ph} it follows that for almost all $s > 0$, the number $\ph(-s) \in [0, \pi]$ is the limit of $\Arg f(s e^{i (\pi/2 - \eps)})$ as $\eps \to 0^+$. However, if $\zeta(s) = i s$, then $\Arg f(s e^{i (\pi/2 - \eps)}) \le 0$, and consequently $\ph(-s) \le 0$ for almost all $s$ in a neighbourhood of~$r$. We have thus proved that $\ph(-s) = 0$ for almost all $s$ in some neighbourhood of~$r$, and so $\zeta(r) = i r \in \dom$. Similar argument shows that if $r \in \Int ((0, \infty) \setminus \zero)$ and $\zeta(r) = -i r$, then $\zeta(r) \in \dom$. Part~\ref{it:r:lambda:a} is proved, and it follows that $\lambda(r) = f(\zeta(r))$ is well-defined for $r \in (0, \infty) \setminus \partial \zero$.

We need the following observation. Suppose that $g(\xi)$ is a holomorphic function in the unit disk $\disk$ and $\im g(\xi)$ is a bounded, positive function on $\disk$. Then, by Poisson's representation theorem, $\im g(\xi)$ has a non-tangential limit $h(\xi)$ for almost every $\xi \in \partial \disk$, and $\im g(\xi)$ is given by the Poisson integral of $h$. Therefore, $-\re g(\xi)$ is the conjugate Poisson integral of $h$. It follows that $\re g(\xi)$ extends to a continuous function on $(\Cl \disk) \setminus \esssupp h$. Furthermore, if this extension is denoted by the same symbol, then on every interval $(\alpha_1, \alpha_2)$ such $h(e^{i \alpha}) = 0$ for almost all $\alpha \in (\alpha_1, \alpha_2)$, the function $\re g(e^{i \alpha})$ is continuous and has positive derivative.

Consider a connected component $U$ of the set $\{\xi \in \hp : \im f(\xi) > 0\}$. From Theorem~\ref{thm:r:real} it follows that $U$ is simply connected, and the boundary of $U$ is a Jordan curve on the Riemann sphere $\C \cup \{\infty\}$, which consists of the curve $\{\zeta(r) : r \in (r_1, r_2)\}$ and the interval $[i r_1, i r_2]$ for some $(r_1, r_2) \sub (0, \infty)$. By Carathéodory's theorem, the Riemann map $\Phi$ between $U$ and $\disk$ extends to a homeomorphism of the boundaries of these domains (as subsets of the Riemann sphere). We apply the property discussed in the previous paragraph to $g(\xi) = \log f(\Phi^{-1}(\xi))$.

We already know that the non-tangential limit of $\im \log f(\xi) = \Arg f(\xi)$ is equal to zero almost everywhere on $\{\zeta(r) : r \in (r_1, r_2)\}$: this is obvious at $\zeta(r)$ for $r \in (r_1, r_2) \cap \zero$, and it was already established in the first part of the proof for $r \in (r_1, r_2) \setminus \zero$ (in the latter case necessarily $\zeta(r) = -i r$). It follows that $\re \log f(\xi)$ extends from $U$ to a continous function on $\{\zeta(r) : r \in (r_1, r_2)\}$. Furthermore, if this extension is denoted again by the same symbol, then $\re \log f(\zeta(r))$ is strictly increasing in $r \in (r_1, r_2)$, because $\Phi(\zeta(r))$ follows an arc of $\partial \disk$ in a counter-clockwise direction. We conclude that $\lambda(r)$ extends to a strictly increasing, continuous function on $(r_1, r_2)$.

The same argument applies to connected components $U$ of the set $\{\xi \in \hp : \im f(\xi) < 0\}$. In this case $-\im f(\xi) > 0$ for $\xi \in U$, and $\Phi(\zeta(r))$ follows an arc of $\partial \disk$ in a clockwise fashion, so again $\lambda(r)$ extends to a strictly increasing continuous function on the appropriate interval $(r_1, r_2)$.

Observe that the intervals $(r_1, r_2)$ corresponding to connected components $U$ of $\{\xi \in \hp : \im f(\xi) > 0\}$ and $\{\xi \in \hp : \im f(\xi) < 0\}$ fully cover $(0, \infty)$. This proves that $\lambda(r)$ extends to a strictly increasing continuous function on $(0, \infty)$. Uniqueness of this extension follows from density of $(0, \infty) \setminus \partial \zero$ in $(0, \infty)$.

To complete the proof of part~\ref{it:r:lambda:b}, it remains to show that $\lambda'(r) \ne 0$ for $r \in (0, \infty) \setminus \partial \zero$. This follows from the properties of the Riemann map. Indeed, suppose that $r \in (0, \infty) \setminus \partial \zero$ and that $U$ is a connected component of $\{\xi \in \hp : \im f(\xi) > 0\}$ or $\{\xi \in \hp : \im f(\xi) < 0\}$ such that $\zeta(r) \in \partial U$. Then the boundary of $U$ is smooth in a neighbourhood of $\zeta(r)$. Thus, the Riemann map $\Phi$ is differentiable at $\zeta(r)$, and $\Phi'(\zeta(r)) \ne 0$. Since we already know that $g'(\Phi(\zeta(r))) \ne 0$, we conclude that $(\log f)'(\zeta(r)) = g'(\Phi(\zeta(r))) \Phi'(\zeta(r)) \ne 0$, that is, $\lambda'(r) = \lambda(r) (\log \lambda)'(r) = \lambda(r) (\log f)'(\zeta(r)) \zeta'(r) \ne 0$, as desired.

To prove part~\ref{it:r:lambda:c}, observe first that $g(\xi) = f(\xi) - f(0^+)$ is a Rogers function, and $0 \le \lambda_g(r) = \lambda_f(r) - f(0^+)$ for $r > 0$. Therefore, $\lambda_f(0^+) \ge f(0^+)$. On the other hand, if $f(\xi)$ has the exponential representation~\eqref{eq:r:exp}, then for all $r \in \zero$ we have $\im \log f(\zeta(r)) = 0$, and therefore
\formula{
 \log \lambda(r) & = \re \log f(\zeta(r)) - \frac{\im \zeta(r)}{\re \zeta(r)} \, \im \log f(\zeta(r)) \\
 & = \log c + \frac{1}{\pi} \int_{-\infty}^\infty \biggl(\re \biggl(\frac{\zeta(r)}{\zeta(r) + i s} - \frac{1}{1 + |s|}\biggr) - \frac{\im \zeta(r)}{\re \zeta(r)} \, \im \frac{\zeta(r)}{\zeta(r) + i s} \biggr) \frac{\ph(s)}{|s|} \, ds\biggr) .
}
An explicit calculation leads to (we omit the details)
\formula{
 \log \lambda(r) & = \log c + \frac{1}{\pi} \int_{-\infty}^\infty \biggl(1 - \frac{s^2}{|\zeta(r) + i s|^2} - \frac{1}{1 + |s|}\biggr) \frac{\ph(s)}{|s|} \, ds .
}
By Fatou's lemma, we have
\formula{
 \liminf_{r \to 0^+} \int_{-1}^1 \frac{s^2}{|\zeta(r) + i s|^2} \, \frac{\ph(s)}{|s|} \, ds & \ge \int_{-1}^1 \frac{\ph(s)}{|s|} \, ds ,
}
while dominated convergence theorem implies that
\formula{
 \lim_{r \to 0^+} \int_{\R \setminus (-1, 1)} \biggl(\frac{s^2}{|\zeta(r) + i s|^2} - 1 \biggr) \frac{\ph(s)}{|s|} \, ds & = 0 .
}
Therefore, if $0 \in \Cl \zero$, we obtain
\formula{
 \log \lambda(0^+) & \le \log c - \frac{1}{\pi} \int_{-\infty}^\infty \frac{1}{1 + |s|} \, \frac{\ph(s)}{|s|} \, ds = \log f(0^+) ,
}
as desired; here we understand that $\log 0 = -\infty$. If $0 \notin \Cl \zero$, then there is $\eps > 0$ such that either $\zeta(r) = i r$ and $\ph(-r) = 0$ for $r \in (0, \eps)$, or $\zeta(r) = -i r$ and $\ph(r) = 0$ for $r \in (0, \eps)$. Both cases are very similar, so we discuss only the former one. We have then
\formula{
 \log \lambda(0^+) & = \log c + \frac{1}{\pi} \lim_{r \to 0^+} \int_{\R \setminus (-\eps, 0)} \biggl(\biggl(\frac{r}{r + s} - \frac{1}{1 + |s|}\biggr) \biggr) \frac{\ph(s)}{|s|} \, ds\biggr) .
}
The desired equality $\log \lambda(0^+) = \log f(0^+)$ follows by an application of the monotone convergence theorem for the integral over $r \in (0, \infty)$ and the dominated convergence theorem for the integral over $r \in (-\infty, -\eps]$.

The proof of the other identity, $\lambda(\infty^-) = f(\infty^-)$, is very similar, and we omit the details.
\end{proof}

\subsection{Symmetrised spine of a Rogers function}
\label{sec:real:symspine}

If $f(\xi)$ is a Rogers function, then we denote by $\Gamma_f^\star$ the union of $\Gamma_f$, all endpoints of $\Gamma_f$, and the mirror image $-\overline{\Gamma}_f$ of $\Gamma_f$ with respect to the imaginary axis. We also extend the definition of $\zeta_f(r)$ to all $r \in \R$ so that $\zeta_f(0) = 0$ and $\zeta_f(-r) = -\overline{\zeta_f(r)}$. The orientation of $-\overline{\Gamma}_f$ is chosen in such a way that $\zeta_f(r)$, $r \in -\zero$, is its parameterisation. Thus, $\Gamma_f^\star$ consists of at most one unbounded simple curve and at most countably many simple closed curves, and any two of them can only touch on the imaginary axis.

The system of curves $\Gamma_f^\star$ naturally divides the complex plane into two open sets, $\dom_f^+$ and $\dom_f^-$, lying above and below $\Gamma_f^\star$. Namely, if $\zeta_f(r) = r e^{i \thet(r)}$ with $\thet(r) \in [-\tfrac{\pi}{2}, \tfrac{\pi}{2}]$ for $r > 0$ and $\thet(0) = 0$, then
\formula{
 \dom_f^+ & = \Int \{r e^{i \alpha} \in \C : r \ge 0 , \alpha \in [\thet(r), \pi - \thet(r)]\} , \\
 \dom_f^- & = \Int \{r e^{i \alpha} \in \C : r \ge 0 , \alpha \in [-\pi - \thet(r), \thet(r)]\} ,
}
where $\Int$ denotes the interior of a set.

We remark that $\dom_f^+ \cap (\C \setminus i \R)$ is the set of those $\xi \in \C \setminus i \R$ for which $\im f(\xi) > 0$, and $\dom_f^- \cap (\C \setminus i \R)$ is the set of $\xi \in \C \setminus i \R$ for which $\im f(\xi) < 0$. Furthermore, for $r > 0$, $i r \in \dom_f^+$ if and only if $\zeta_f(r) \ne i r$; similarly, $-i r \in \dom_f^-$ if and only if $\zeta_f(r) \ne -i r$. On the other hand, for $r > 0$, $-i r \in \dom_f^+$ if and only if $\zeta_f(s) = -i s$ for $s$ in a neighbourhood of~$r$, and $i r \in \dom_f^-$ if and only if $\zeta_f(s) = i s$ for $s$ in a neighbourhood of~$r$.

Note that the closure of $\Gamma_f^\star$ is the boundary of both $\dom_f^+$ and $\dom_f^-$. However, in general $\Gamma_f^\star$ need not be a closed set: its closure may contain additional points on the imaginary axis.

\begin{figure}
\centering
\begin{tabular}{cccc}
\includegraphics[width=0.31\textwidth]{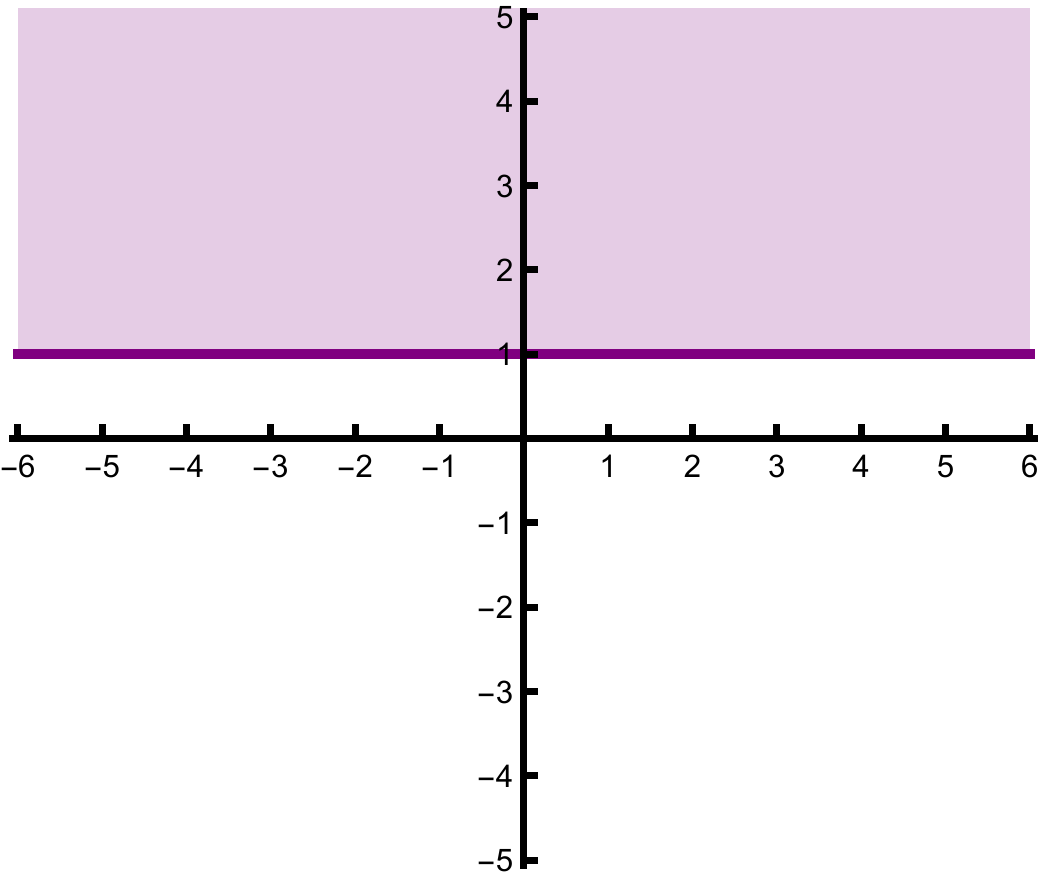}
&
\includegraphics[width=0.31\textwidth]{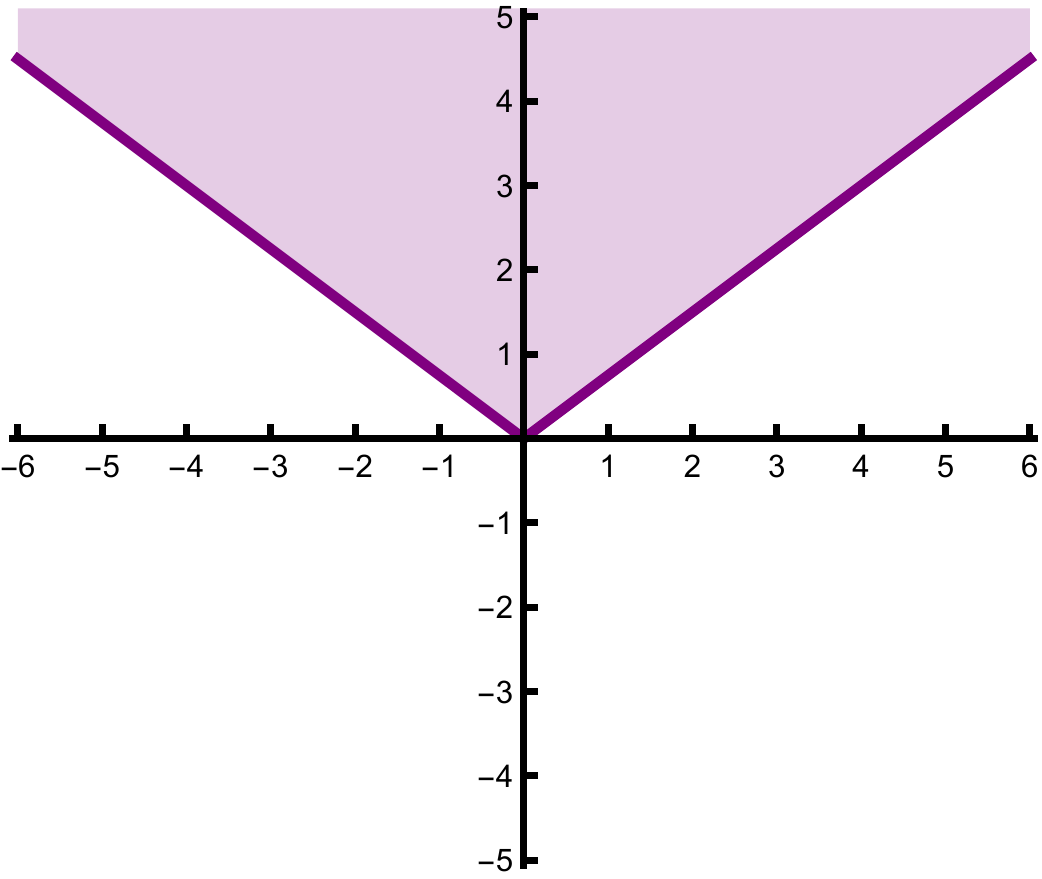}
&
\includegraphics[width=0.31\textwidth]{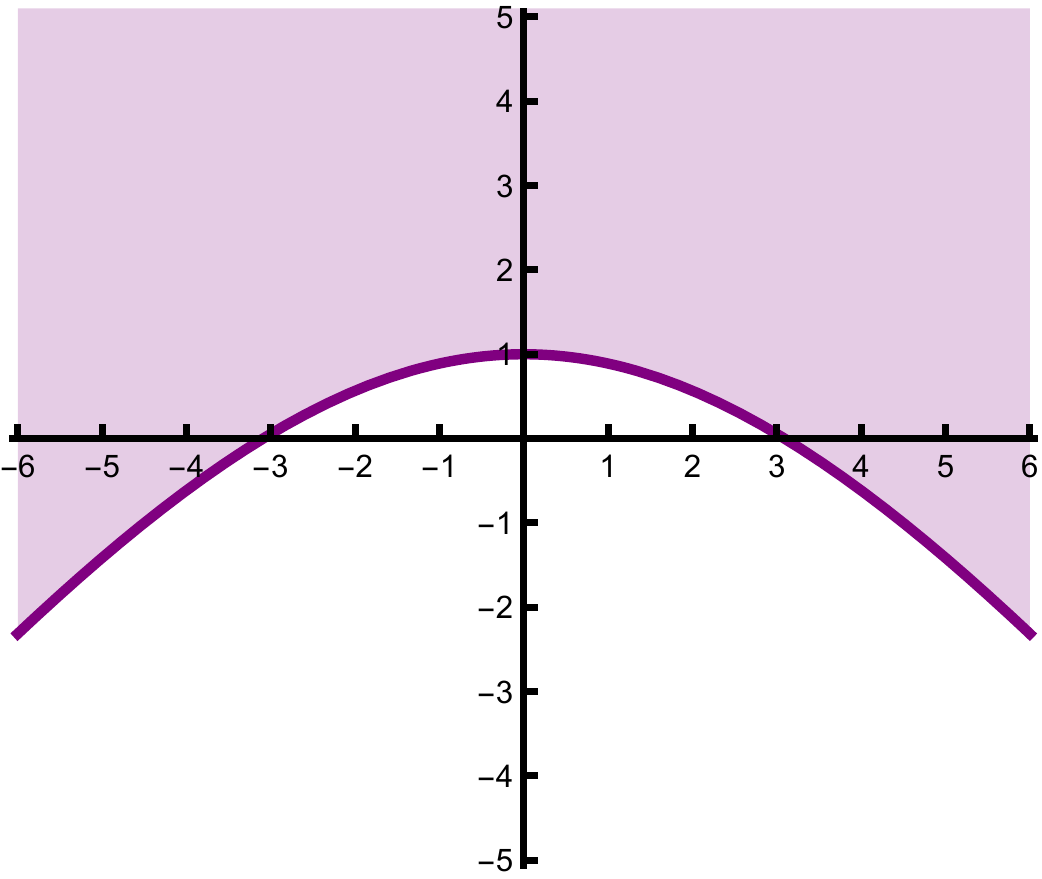}
\\[-0.2em]
(a)
&
(b)
&
(c)
\\[1em]
\includegraphics[width=0.31\textwidth]{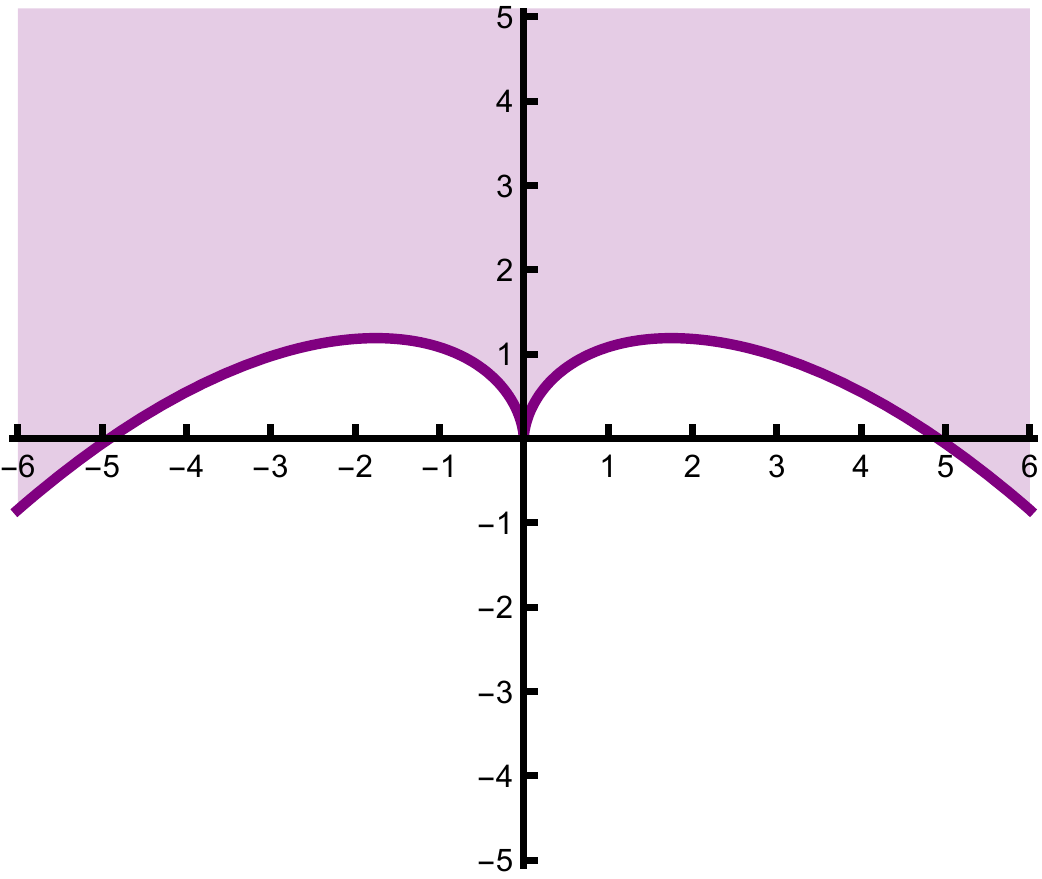}
&
\includegraphics[width=0.31\textwidth]{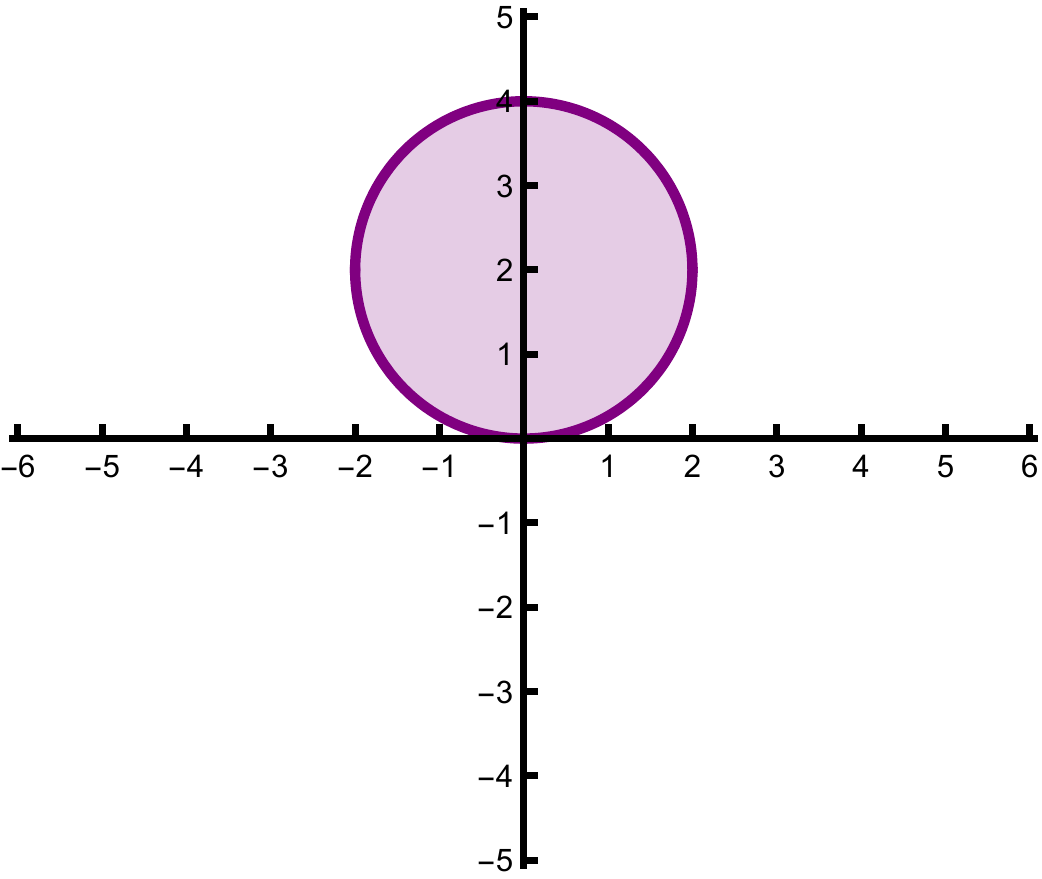}
&
\includegraphics[width=0.31\textwidth]{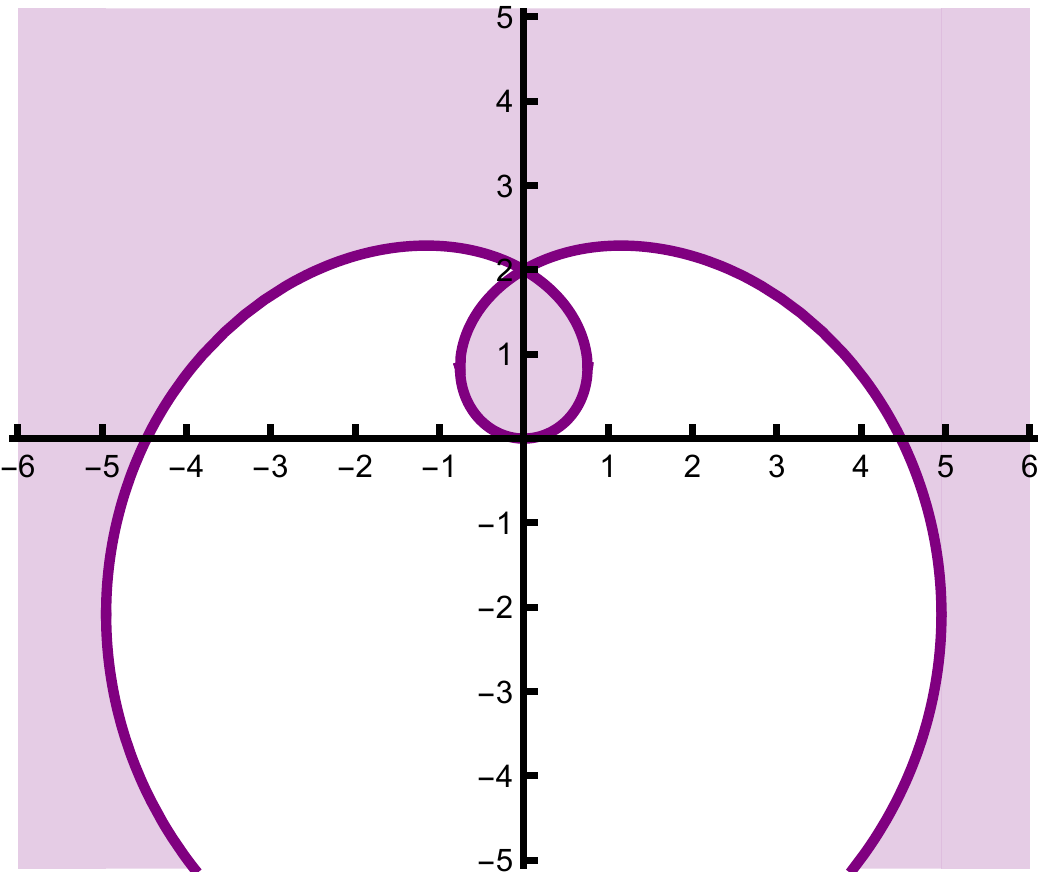}
\\[-0.2em]
(d)
&
(e)
&
(f)
\\[1em]
\includegraphics[width=0.31\textwidth]{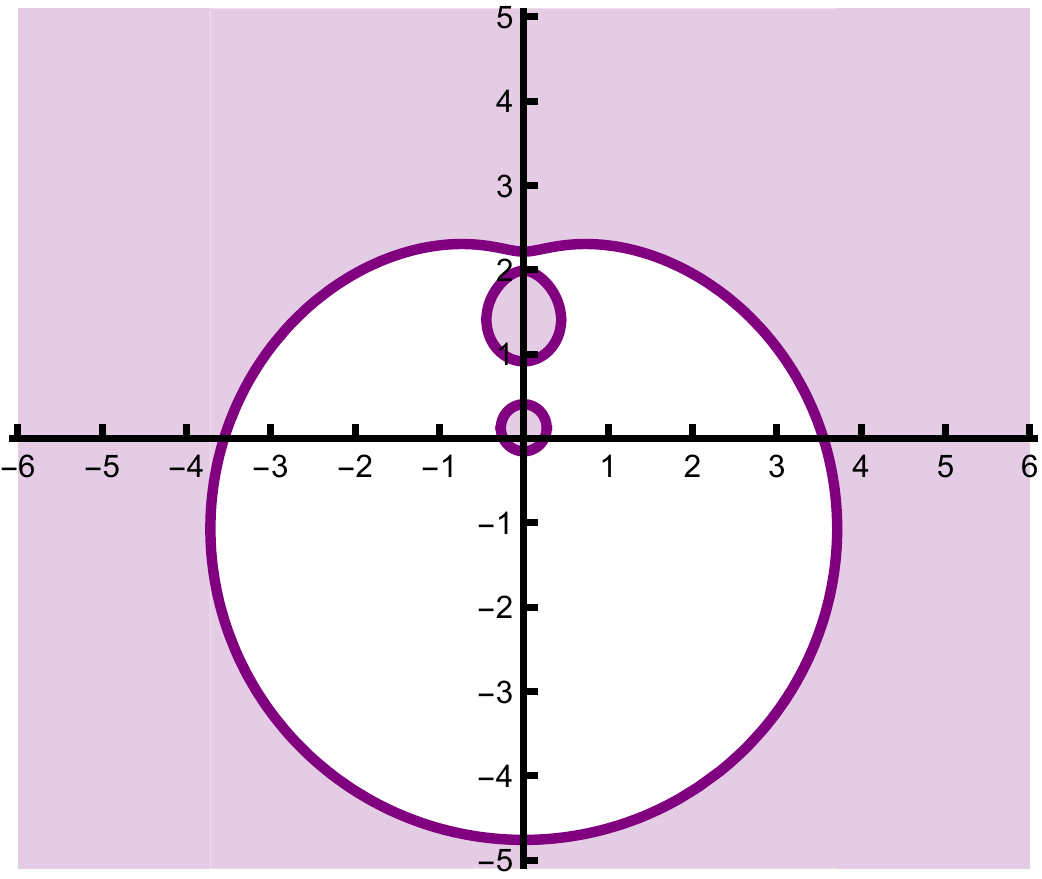}
&
\includegraphics[width=0.31\textwidth]{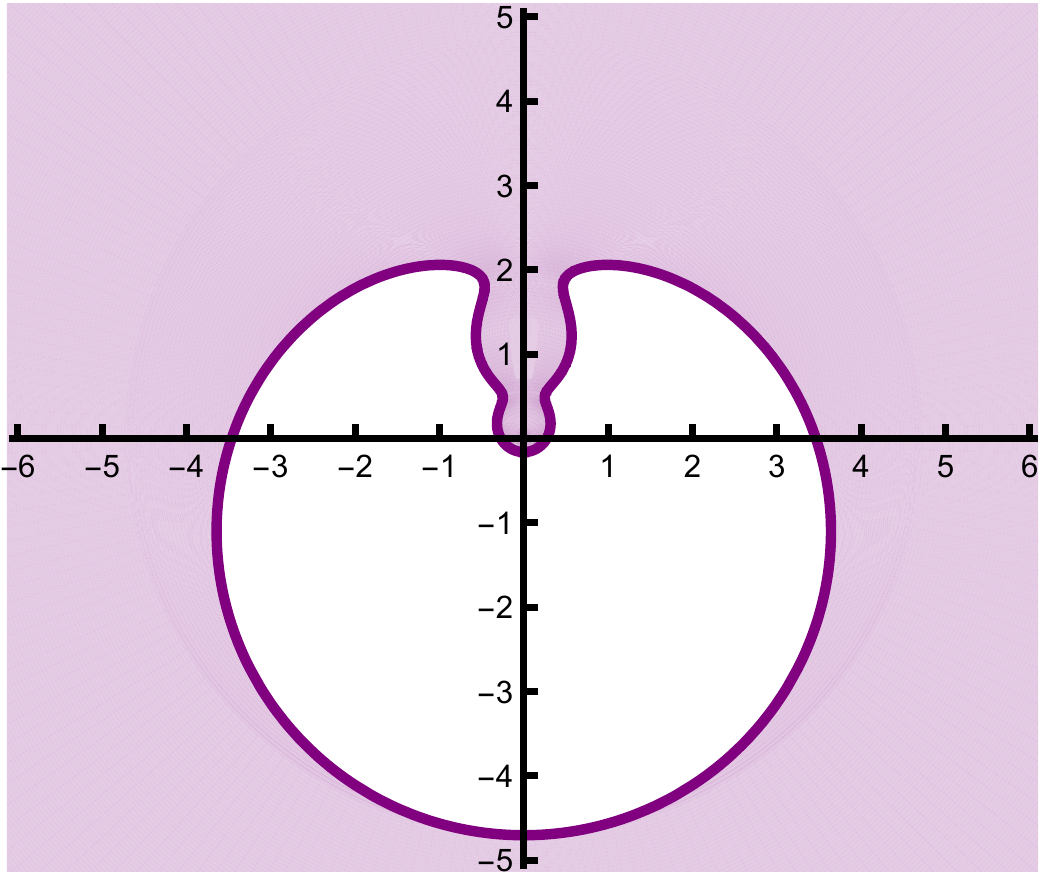}
&
\\[-0.2em]
(g)
&
(h)
&
\end{tabular}
\caption{Plot of the symmetrised spine $\Gamma_f^\star$ (thick purple line) and the open sets $\dom_f^+$ (purple) and $\dom_f^-$ (white) for the eight Rogers functions introduced in Figure~\ref{fig:r:spine}.}
\label{fig:r:dom}
\end{figure}

The notation $\Gamma_f^\star$, $\dom_f^+$ and $\dom_f^-$ is kept throughout the paper. The sets $\dom_f^+$ and $\dom_f^-$ for sample Rogers functions are depicted in Figure~\ref{fig:r:dom}.

%
%

\section{Wiener--Hopf factorisation}
\label{sec:wh}

\subsection{Wiener--Hopf factorisation theorem}

The proof that the Wiener--Hopf factors of a Rogers function are complete Bernstein functions was essentially given in~\cite{bib:r83}, where it is shown that $f$ is a Rogers function if and only if $f(\xi) + 1 = f^+(-i \xi) f^-(i \xi)$ for some complete Bernstein functions $f^+(\xi)$ and $f^-(\xi)$. The following statement is a minor modification. For completeness, we provide a simplified version of the proof from~\cite{bib:r83}.

\begin{theorem}[{see~\cite[Theorem~2]{bib:r83}}]
\label{th:r:wh}
A function $f(\xi)$ holomorphic in $\hp$ is a non-zero Rogers function if and only if it admits a Wiener--Hopf factorisation
\formula[eq:r:wh:fact]{
 f(\xi) & = f^+(-i \xi) f^-(i \xi)
}
for some non-zero complete Bernstein functions $f^+(\xi)$, $f^-(\xi)$ and for all $\xi \in \hp$, or, equivalently, $\xi \in \dom_f$. The factors $f^+(\xi)$ and $f^-(\xi)$ are defined uniquely, up to multiplication by a constant.
\end{theorem}

\begin{proof}
Suppose that $f(\xi)$ is a non-zero Rogers function with exponential representation~\eqref{eq:r:exp}, and define
\formula[eq:r:wh]{
 f^+(\xi) & = c_+ \, \exp \biggl(\frac{1}{\pi} \int_0^\infty \biggl(\frac{\xi}{\xi + s} - \frac{1}{1 + s}\biggr) \frac{\ph(s)}{s} \, ds\biggr) \\
 f^-(\xi) & = c_- \, \exp \biggl(\frac{1}{\pi} \int_0^\infty \biggl(\frac{\xi}{\xi + s} - \frac{1}{1 + s}\biggr) \frac{\ph(-s)}{s} \, ds\biggr) ,
}
where $c_+, c_- > 0$ satisfy $c_+ c_- = c$. The desired factorisation~\eqref{eq:r:wh:fact} for $\xi \in \hp$ follows directly from~\eqref{eq:r:exp}, and by Theorem~\ref{thm:cbf}\ref{it:cbf:c}, $f^+(\xi)$ and $f^-(\xi)$ are indeed complete Bernstein functions of $\xi$. Extension to $\xi \in \dom_f$ is immediate.

Conversely, if $f^+(\xi)$ and $f^-(\xi)$ are complete Bernstein functions with exponential representation~\eqref{eq:r:wh}, then $f(\xi) = f^+(-i \xi) f^-(i \xi)$ is given by~\eqref{eq:r:exp}, and therefore it is a Rogers function.

Finally, to prove uniqueness of the Wiener--Hopf factors, recall that the pair $(f^+, f^-)$ corresponds in a one-to-one way to the triple $(c_+, c_-, \ph)$, while $f$ corresponds in a one-to-one way to the pair $(c, \ph)$; here we identify functions $\ph(s)$ equal almost everywhere.
\end{proof}

Note that the Wiener--Hopf factors $f^+(\xi)$, $f^-(\xi)$, defined by~\eqref{eq:r:wh}, extend to holomorphic functions in $\C \setminus ((-\esssupp \ph) \cap (-\infty, 0])$ and $\C \setminus ((\esssupp \ph) \cap (-\infty, 0])$, respectively. The constants $c_+$ and $c_-$ do not play an essential role and we do not specify their values. Note, however, that the quantities $f^+(\xi_1) / f^+(\xi_2)$, $f^-(\xi_1) / f^-(\xi_2)$ and $f^+(\xi_1) f^-(\xi_2)$ do not depend on the choice of $c_+$ and $c_-$.

Closed form expressions for the Wiener--Hopf factors are rarely available. Below we list two most important examples.

\begin{example}
\label{ex:r:wh}
\begin{enumerate}[label=\rm (\alph*)]
\item
\label{it:r:wh:bm}
The Wiener--Hopf factors for the characteristic exponent $f(\xi) = \frac{1}{2} \xi^2 - i b \xi$ of the Brownian motion with drift are given by
\formula{
 f^+(\xi) & = c_+ \xi, & f^-(\xi) & = c_- (\xi + b)
}
if $b \ge 0$, and by
\formula{
 f^+(\xi) & = c_+ (\xi - b) , & f^-(\xi) & = c_- \xi 
}
if $b \le 0$, where $c_+ c_- = \tfrac{1}{2}$.
\item
\label{it:r:wh:st}
When $f(\xi) = c \xi^\alpha$ is the characteristic exponent of a strictly stable process (here $\alpha \in (0, 2]$ and $|\Arg c| \le \tfrac{\pi}{2} \min\{\alpha, 2 - \alpha\}$), the Wiener--Hopf factors are given by
\formula{
 f^+(\xi) & = c_+ \xi^{\alpha \ro} , & f^-(\xi) & = c_- \xi^{\alpha (1 - \ro)} , 
}
where $c_+ c_- = |c|$ and $\ro = \tfrac{1}{2} - \tfrac{1}{\alpha \pi} \Arg c$ is the \emph{positivity parameter} of the corresponding strictly stable Lévy process $X_t$: we have $\ro = \pr(X_t > 0)$ for every $t > 0$.
\end{enumerate}
\end{example}

\subsection{Baxter--Donsker formulae}

Recall that every Rogers function $f$ is automatically extended from $\hp$ to $\C \setminus i \R$ in such a way that $f(-\overline{\xi}) = \overline{f(\xi)}$ for $\xi \in \C \setminus i \R$. This extension is again given by the Stieltjes representation~\eqref{eq:r:int} and, for a non-zero Rogers function $f$, by the exponential representation~\eqref{eq:r:exp}. 

We begin with a Baxter--Donsker-type expression, similar to the one found in~\cite{bib:bd57}. A~simpler proof of this result can be given, which uses Cauchy's integral formula. However, we choose a more technical argument involving Fubini's theorem, in order to illustrate the key idea of the proof of Theorem~\ref{thm:r:curve}.

\begin{proposition}
\label{prop:r:bd}
If $f(\xi)$ is a non-constant Rogers function, then
\formula[eq:r:bd]{
 \hspace*{7em} & \hspace*{-7em} \exp \biggl(\frac{1}{2 \pi i} \int_{-\infty}^\infty \biggl(\frac{1}{z - \xi_1} - \frac{1}{z - \xi_2}\biggr) \log f(z) dz\biggr) \\
 & = \begin{cases}
  f^+(-i \xi_1) / f^+(-i \xi_2) & \text{if $\im \xi_1, \im \xi_2 > 0$,} \\ 
  f^-(i \xi_2) / f^-(i \xi_1) & \text{if $\im \xi_1, \im \xi_2 < 0$,} \\ 
  f^+(-i \xi_1) f^-(-i \xi_2) & \text{if $\im \xi_1 > 0$, $\im \xi_2 < 0$.}
 \end{cases}
}
\end{proposition}

\begin{proof}
Let the Rogers function $f(\xi)$ and its Wiener--Hopf factors $f^+(\xi)$, $f^-(\xi)$ have exponential representations~\eqref{eq:r:exp} and~\eqref{eq:r:wh}, respectively. Suppose that $\im \xi_1, \im \xi_2 > 0$. By~\eqref{eq:r:wh}, we have
\formula{
 \log \frac{f^+(-i \xi_1)}{f^+(-i \xi_2)} & = \frac{1}{\pi} \int_0^\infty \biggl(\frac{-i \xi_1}{-i \xi_1 + s} - \frac{-i \xi_2}{-i \xi_2 + s}\biggr) \frac{\ph(s)}{s} \, ds .
}
On the other hand, let $I$ denote the logarithm of the left-hand side of~\eqref{eq:r:bd}. Using~\eqref{eq:r:exp}, we find that
\formula[eq:r:bd:aux]{
 I & = \frac{1}{2 \pi i} \int_{-\infty}^\infty \biggl(\frac{1}{z - \xi_1} - \frac{1}{z - \xi_2}\biggr) \biggl(\log c + \frac{1}{\pi} \int_{-\infty}^\infty \biggl(\frac{z}{z + i s} - \frac{1}{1 + |s|}\biggr) \frac{\ph(s)}{|s|}\biggr) dz .
}
The integral of $(z - \xi_1)^{-1} - (z -\xi_2)^{-1}$ over $z \in \R$ is absolutely convergent and equal to zero by Cauchy's integral formula; we omit the details. It follows that
\formula{
 I & = \frac{1}{2 \pi^2 i} \int_{-\infty}^\infty \int_{-\infty}^\infty \biggl(\frac{1}{z - \xi_1} - \frac{1}{z - \xi_2}\biggr) \biggl(\frac{z}{z + i s} - \frac{1}{1 + |s|}\biggr) \frac{\ph(s)}{|s|} \, ds dz .
}
The integrand in the right-hand side is an absolutely integrable function:
\formula[eq:r:fubini]{
 \hspace*{7em} & \hspace*{-7em} \abs{\biggl(\frac{1}{z - \xi_1} - \frac{1}{z - \xi_2}\biggr) \biggl(\frac{z}{z + i s} - \frac{1}{1 + |s|}\biggr) \frac{\ph(s)}{|s|}} \\
 & = \frac{|\xi_1 - \xi_2|}{|z - \xi_1| |z - \xi_2|} \, \frac{|s| |z - i \sign s|}{|z + i s| (1 + |s|)} \, \frac{\ph(s)}{|s|} \\
 & \le C(\xi_1, \xi_2) \, \frac{1}{(1 + |z|)^2} \, \frac{1 + |z|}{(|z| + |s|) (1 + |s|)} \\
 & \le C(\xi_1, \xi_2) \, \frac{1}{(1 + |z|) \sqrt{|z|}} \, \frac{1}{(1 + |s|) \sqrt{|s|}}
}
for some positive number $C(\xi_1, \xi_2)$. Therefore, by Fubini's theorem,
\formula{
 I & = \frac{1}{\pi} \int_{-\infty}^\infty \biggl(\frac{1}{2 \pi i} \int_{-\infty}^\infty \biggl(\frac{1}{z - \xi_1} - \frac{1}{z - \xi_2}\biggr) \biggl(\frac{z}{z + i s} - \frac{1}{1 + |s|}\biggr) dz\biggr) \frac{\ph(s)}{|s|} \, ds .
}
The inner integral can be evaluated explicitly by the residue theorem: the integrand is a meromorphic function in the lower complex half-plane, which decays faster than $|z|^{-1}$ as $|z| \to \infty$. If $s > 0$, it has one pole, located at $z = -i s$, with residue
\formula{
 \biggl(\frac{1}{\xi_1 + i s} - \frac{1}{\xi_2 + i s}\biggr) (-i s) & = -\frac{s}{-i \xi_1 + s} + \frac{s}{-i \xi_2 + s} = \frac{-i \xi_1}{-i \xi_1 + s} - \frac{-i \xi_2}{-i \xi_2 + s} \, ;
}
for $s < 0$, there are no poles. We conclude that
\formula{
 I & = \frac{1}{\pi} \int_0^\infty \biggl(\frac{-i \xi_1}{-i \xi_1 + s} - \frac{-i \xi_2}{-i \xi_2 + s}\biggr) \frac{\ph(s)}{|s|} \, ds , 
}
and the first part of~\eqref{eq:r:bd} follows. The second one is proved in a very similar way.

The proof of the third part requires some modifications. Suppose that $\im \xi_1 > 0$ and $\im \xi_2 < 0$. In this case the logarithm of the left-hand side of~\eqref{eq:r:bd} is again given by~\eqref{eq:r:bd:aux}, but the integral of $(z - \xi_1)^{-1} - (z - \xi_2)^{-1}$ over $z \in \R$ is absolutely convergent and equal to $2 \pi i$ rather than $0$; again we omit the details. It follows that
\formula{
 I & = \log c + \frac{1}{2 \pi^2 i} \int_{-\infty}^\infty \int_{-\infty}^\infty \biggl(\frac{1}{z - \xi_1} - \frac{1}{z - \xi_2}\biggr) \biggl(\frac{z}{z + i s} - \frac{1}{1 + |s|}\biggr) \frac{\ph(s)}{|s|} \, ds dz .
}
As before, we may use Fubini's theorem, and we obtain
\formula{
 I & = \log c + \frac{1}{\pi} \int_{-\infty}^\infty \biggl(\frac{1}{2 \pi i} \int_{-\infty}^\infty \biggl(\frac{1}{z - \xi_1} - \frac{1}{z - \xi_2}\biggr) \biggl(\frac{z}{z + i s} - \frac{1}{1 + |s|}\biggr) dz\biggr) \frac{\ph(s)}{|s|} \, ds .
}
Again, the inner integral can be evaluated explicitly by the residue theorem. If $s > 0$, the integrand is a meromorphic function in the upper complex half-plane, which decays faster than $|z|^{-1}$ as $|z| \to \infty$. It has a single pole, located at $z = \xi_1$, with residue $\xi_1 / (\xi_1 + i s) - 1 / (1 + |s|)$. On the other hand, if $s < 0$, the integrand is a meromorphic function in the lower complex half-plane, which decays faster than $|z|^{-1}$ as $|z| \to \infty$. It has a single pole, located at $z = \xi_2$, with residue $-\xi_2 / (\xi_2 + i s) + 1 / (1 + |s|)$. Therefore,
\formula{
 I & = \log c + \frac{1}{\pi} \int_{-\infty}^\infty \biggl(\frac{\xi_1}{\xi_1 + i s} \ind_{(0, \infty)}(s) + \frac{\xi_2}{\xi_2 + i s} \ind_{(-\infty, 0)}(s) - \frac{1}{1 + |s|}\biggr) \frac{\ph(s)}{|s|} \, ds .
}
By the definition~\eqref{eq:r:wh} of Wiener--Hopf factors $f^+(-i \xi_1)$, $f^-(i \xi_2)$, we conclude that $I = \log f^+(-i \xi_1) + \log f^-(i \xi_2)$, and the third part of formula~\eqref{eq:r:bd} follows.
\end{proof}

\subsection{Contour deformation in Baxter--Donsker formulae}
\label{sec:wh:bd}

The contour of integration in the expression given in Proposition~\ref{prop:r:bd} need not be $\R$: it can be deformed to a more general one. For our purposes it is important to replace it by the symmetrised spine $\Gamma_f^\star$. In this case it is in fact easier to repeat the proof of Proposition~\ref{prop:r:bd} rather than deform the contour of integration. Before we state the result, however, we need a technical lemma.

\begin{lemma}
\label{lem:r:winding}
Suppose that $f(\xi)$ is a non-degenerate Rogers function and $\xi_1, \xi_2 \in \dom_f^+ \cup \dom_f^-$. Then
\formula[eq:r:winding]{
 \frac{1}{2 \pi i} \int_{\Gamma_f^\star} \biggl(\frac{1}{z - \xi_1} - \frac{1}{z - \xi_2}\biggr) dz & = \ind_{\dom_f^+}(\xi_1) - \ind_{\dom_f^+}(\xi_2) .
}
\end{lemma}

\begin{proof}
For simplicity, we omit the subscript $f$. The integral is absolutely convergent by Theorem~\ref{thm:r:real}\ref{it:r:real:d} and the fact that the integrand is bounded on $\Gamma^\star$ by $C(\xi_1, \xi_2) (1 + |z|)^{-2}$ for some $C(\xi_1, \xi_2)$ that depends continuously on $\xi_1, \xi_2 \in \dom^+ \cup \dom^-$. By dominated convergence theorem, the integral is a continuous function of $\xi_1, \xi_2 \in \dom^+ \cup \dom^-$, and therefore it is sufficient to prove the result when $\xi_1, \xi_2$ do not lie on the imaginary axis, that is, $\xi_1, \xi_2 \in (\dom^+ \cup \dom^-) \setminus i \R$. Denote the left-hand side of~\eqref{eq:r:winding} by $I$. We claim that
\formula{
 I & = \frac{1}{2 \pi i} \int_{-\infty}^\infty \biggl(\frac{1}{\zeta(r) - \xi_1} - \frac{1}{\zeta(r) - \xi_2}\biggr) \zeta'(r) dr .
}
Indeed, the curve parameterised by $\zeta(r)$, $r \in \R$, consists of $\Gamma_f^\star$ and a part of the imaginary axis, which is added twice, each time with an opposite orientation, and thus it does not contribute to the integral.

We write, as usual, $\zeta(r) = r e^{i \thet(r)}$ and $\zeta(-r) = -r e^{-i \thet(r)}$ for $r > 0$, where $\thet(r) \in [-\tfrac{\pi}{2}, \tfrac{\pi}{2}]$. For $p \in (0, 1)$ we define a deformed contour
\formula{
 \zeta_p(r) & = r e^{i p \thet(r)} , & \zeta_p(-r) & = -r e^{-i p \thet(r)}
}
when $r > 0$, and $\zeta_p(0) = 0$. The function $\zeta_p(r)$, $r \in \R$, is a parameterisation of a simple curve $\Gamma_p^\star$, that divides the complex plane into two sets, $D_p^+$ and $D_p^-$, namely,
\formula{
 D_p^+ & = \{r e^{i \alpha} : r > 0 , \alpha \in (p \thet(r), \pi - p \thet(r))\} , \\
 D_p^- & = \{r e^{i \alpha} : r > 0 , \alpha \in (-\pi - p \thet(r), p \thet(r))\} .
}
Since $\xi_1, \xi_2$ do not lie on the imaginary axis, there is $\eps > 0$ such that if $p \in [1 - \eps, 1)$, then $\ind_{D_p^+}(\xi_1) = \ind_{D_p}(\xi_1)$ and $\ind_{D_p^+}(\xi_2) = \ind_{D_p}(\xi_2)$. Furthermore, one easily finds that $|\zeta_p'(r)| \le |\zeta'(r)|$. Thus, by dominated convergence theorem,
\formula[eq:r:winding:aux]{
 I & = \lim_{p \to 1^-} \biggl(\frac{1}{2 \pi i} \int_{-\infty}^\infty \biggl(\frac{1}{\zeta_p(r) - \xi_1} - \frac{1}{\zeta_p(r) - \xi_2}\biggr) \zeta_p'(r) dr\biggr) \\
 & = \lim_{p \to 1^-} \biggl(\frac{1}{2 \pi i} \int_{\Gamma_p^\star} \biggl(\frac{1}{z - \xi_1} - \frac{1}{z - \xi_2}\biggr) dz\biggr) ;
}
indeed, the integrand is bounded by $C(\xi_1, \xi_2) (1 + |z|)^{-2}$ uniformly with respect to $p \in [1 - \eps, 1)$. In order to complete the proof, we only need to show that for $p \in [1 - \eps, 1)$ the expression under the limit in~\eqref{eq:r:winding:aux} is equal to $\ind_{\dom^+}(\xi_1) - \ind_{\dom^+}(\xi_2)$.

Since $\Gamma_p^\star$ is a simple curve, the following standard argument applies. Fix $R > \max\{|\xi_1|, |\xi_2|\}$ and let $\disk(R)$ be the disk $\{z \in \C : |z| < R\}$. Then $\partial (D_p^+ \cap \disk(R))$ is a simple closed rectifiable curve, which consists of $\Gamma_p^\star \cap \disk(R)$ and an arc of the circle $\partial \disk(R)$. If the orientation of $\partial (D_p^+ \cap \disk(R))$ agrees with that of $\Gamma_p^\star$, we obtain, by the residue theorem,
\formula{
 \frac{1}{2 \pi i} \int_{\partial (D_p^+ \cap \disk(R))} \biggl(\frac{1}{z - \xi_1} - \frac{1}{z - \xi_2}\biggr) dz & = \ind_{D_p^+ \cap \disk(R)}(\xi_1) - \ind_{D_p^+ \cap \disk(R)}(\xi_2) \\
 & = \ind_{D_p^+}(\xi_1) - \ind_{D_p^+}(\xi_2) = \ind_{\dom^+}(\xi_1) - \ind_{\dom^+}(\xi_2) .
}
The integral over $\Gamma_p^\star \cap \disk(R)$ converges to the integral over $\Gamma_p^\star$ as $R \to \infty$ (by dominated convergence theorem: the latter integral is absolutely integrable). Since the integrand decays faster than $|z|^{-1}$ as $|z| \to \infty$, the integral over the arc of the circle $\partial \disk(R)$ converges to zero as $R \to \infty$. Thus, the expression under the limit in~\eqref{eq:r:winding:aux} is indeed equal to $\ind_{\dom^+}(\xi_1) - \ind_{\dom^+}(\xi_2)$, and the proof is complete.
\end{proof}

\begin{theorem}
\label{thm:r:curve}
If $f(\xi)$ is a non-degenerate Rogers function and $\xi_1, \xi_2 \in \dom_f^+ \cup \dom_f^-$, then
\formula[eq:r:curve]{
 \hspace*{5em} & \hspace*{-5em} \exp \biggl(\frac{1}{2 \pi i} \int_{\Gamma_f^\star} \biggl(\frac{1}{z - \xi_1} - \frac{1}{z - \xi_2}\biggr) \log f(z) dz\biggr) \\
 & = \begin{cases}
  f^+(-i \xi_1) / f^+(-i \xi_2) & \text{if $\xi_1, \xi_2 \in \dom_f^+$,} \\ 
  f^-(i \xi_2) / f^-(i \xi_1) & \text{if $\xi_1, \xi_2 \in \dom_f^-$,} \\ 
  f^+(-i \xi_1) f^-(-i \xi_2) & \text{if $\xi_1 \in \dom_f^+$, $\xi_2 \in \dom_f^-$.}
 \end{cases}
}
\end{theorem}

\begin{proof}
The proof is very similar to that of Proposition~\ref{prop:r:bd}: if $\Gamma_f$ is contained in the region $\{\xi \in \C : |\Arg \xi| < \tfrac{\pi}{2} - \eps\}$ for some $\eps > 0$, then essentially no changes are required. In the general case, however, technical problems arise, and so we provide full details.

For simplicity, we drop the subscript $f$ from the notation. Let $I$ denote the logarithm of the left-hand side of~\eqref{eq:r:curve}, and suppose that $\xi_1, \xi_2 \in \dom^+$. Using~\eqref{eq:r:exp}, we find that
\formula{
 I & = \frac{1}{2 \pi i} \int_{\Gamma^\star} \biggl(\frac{1}{z - \xi_1} - \frac{1}{z - \xi_2}\biggr) \biggl(\log c + \frac{1}{\pi} \int_{-\infty}^\infty \biggl(\frac{z}{z + i s} - \frac{1}{1 + |s|}\biggr) \frac{\ph(s)}{|s|}\biggr) dz .
}
By Lemma~\ref{lem:r:winding}, we have
\formula{
 \frac{1}{2 \pi i} \int_{\Gamma^\star} \biggl(\frac{1}{z - \xi_1} - \frac{1}{z - \xi_2}\biggr) dz & = 0 .
}
It follows that
\formula{
 I & = \frac{1}{2 \pi i} \int_{\Gamma^\star} \biggl(\frac{1}{\pi} \int_{-\infty}^\infty \biggl(\frac{1}{z - \xi_1} - \frac{1}{z - \xi_2}\biggr) \biggl(\frac{z}{z + i s} - \frac{1}{1 + |s|}\biggr) \frac{\ph(s)}{|s|} \, ds\biggr) dz .
}
An analogue of the estimate~\eqref{eq:r:fubini} of the integrand is found using Proposition~\ref{prop:r:spine:bound}: we have
\formula{
 \hspace*{3em} & \hspace*{-3em} \int_{\Gamma^\star} \int_{-\infty}^\infty \abs{\biggl(\frac{1}{z - \xi_1} - \frac{1}{z - \xi_2}\biggr) \biggl(\frac{z}{z + i s} - \frac{1}{1 + |s|}\biggr) \frac{\ph(s)}{|s|}} ds dz \\
 & \le \sqrt{2 \pi} \int_{\Gamma^\star} \frac{|\xi_1 - \xi_2|}{|z - \xi_1| |z - \xi_2|} \, \frac{1 + |z|}{\sqrt{|z|}} \, dz \le C(\xi_1, \xi_2) \int_{\Gamma^\star} \frac{1}{(1 + |z|) \sqrt{|z|}} \, dz < \infty .
}
This allows us to apply Fubini's theorem in the expression for $I$. We obtain that
\formula[eq:r:curve:aux]{
 I & = \frac{1}{\pi} \int_{-\infty}^\infty \biggl(\frac{1}{2 \pi i} \int_{\Gamma^\star} \biggl(\frac{1}{z - \xi_1} - \frac{1}{z - \xi_2}\biggr) \biggl(\frac{z}{z + i s} - \frac{1}{1 + |s|}\biggr) dz\biggr) \frac{\ph(s)}{|s|} \, ds .
}
Simplification of the above expression requires a few steps. If $-i s \in \dom_f^- \cup \dom_f^+$, the inner integral in~\eqref{eq:r:curve:aux} can be evaluated explicitly: we have
\formula{
 \biggl(\frac{1}{z - \xi_1} - \frac{1}{z - \xi_2}\biggr) \biggl(\frac{z}{z + i s} - \frac{1}{1 + |s|}\biggr) & = \biggl(\frac{-i \xi_1}{-i \xi_1 + s} - \frac{1}{1 + |s|}\biggr) \biggl(\frac{1}{z - \xi_1} - \frac{1}{z + i s}\biggr) \\
 & \hspace*{2em} - \biggl(\frac{-i \xi_2}{-i \xi_2 + s} - \frac{1}{1 + |s|}\biggr) \biggl(\frac{1}{z - \xi_2} - \frac{1}{z + i s}\biggr) ,
}
and so, by Lemma~\ref{lem:r:winding},
\formula{
 \hspace*{7em} & \hspace*{-7em} \frac{1}{2 \pi i} \int_{\Gamma^\star} \biggl(\frac{1}{z - \xi_1} - \frac{1}{z - \xi_2}\biggr) \biggl(\frac{z}{z + i s} - \frac{1}{1 + |s|}\biggr) dz \\
 & = \biggl(\frac{-i \xi_1}{-i \xi_1 + s} - \frac{1}{1 + |s|}\biggr) (\ind_{\dom^+}(\xi_1) - \ind_{\dom^+}(-i s)) \\
 & \hspace*{10em} - \biggl(\frac{-i \xi_2}{-i \xi_2 + s} - \frac{1}{1 + |s|}\biggr) (\ind_{\dom^+}(\xi_2) - \ind_{\dom^+}(-i s)) \\
 & = \biggl(\frac{-i \xi_1}{-i \xi_1 + s} - \frac{-i \xi_2}{-i \xi_2 + s}\biggr) (1 - \ind_{\dom^+}(-i s))
}
(in the last equality we used the fact that $\xi_1, \xi_2 \in \dom^+$).

As it was observed in the proof of Theorem~\ref{thm:r:lambda}, for every $s < 0$ such that $-i s \notin \dom^+$ we have $\zeta(-s) = -i s$, and consequently $\ph(s) = 0$ for almost all $s < 0$ such that $-i s \notin \dom^+$. On the other hand, if $s < 0$ and $-i s \in \dom^+$, then we have already found that the inner integral in~\eqref{eq:r:curve:aux} is zero. Therefore,
\formula{
 I & = \frac{1}{\pi} \int_0^\infty \biggl(\frac{1}{2 \pi i} \int_{\Gamma^\star} \biggl(\frac{1}{z - \xi_1} - \frac{1}{z - \xi_2}\biggr) \biggl(\frac{z}{z + i s} - \frac{1}{1 + |s|}\biggr) dz\biggr) \frac{\ph(s)}{|s|} \, ds .
}
In a similar way, $\ph(s) = 0$ for almost all $s > 0$ such that $-i s \notin \dom^-$. On the other hand, if $-i s \in \dom^-$, then we already evaluated the inner integral. We conclude that
\formula{
 I & = \frac{1}{\pi} \int_0^\infty \biggl(\frac{-i \xi_1}{-i \xi_1 + s} - \frac{-i \xi_2}{-i \xi_2 + s}\biggr) (1 - \ind_{\dom^+}(-i s)) \ind_{\dom^-}(-i s) \, \frac{\ph(s)}{|s|} \, ds \\
 & = \frac{1}{\pi} \int_0^\infty \biggl(\frac{-i \xi_1}{-i \xi_1 + s} - \frac{-i \xi_2}{-i \xi_2 + s}\biggr) \frac{\ph(s)}{|s|} \, ds .
}
Combined with~\eqref{eq:r:wh}, this leads to the first part of~\eqref{eq:r:curve}. The proof of the secod part is very similar, and we omit the details.

The proof of the last part of formula~\eqref{eq:r:curve} is also alike, but here some of the necessary modifications are not as straightforward, and we discuss them below. By Lemma~\ref{lem:r:winding}, if $\xi_1 \in \dom^+$ and $\xi_2 \in \dom^-$, then
\formula{
 \frac{1}{2 \pi i} \int_{\Gamma^\star} \biggl(\frac{1}{z - \xi_1} - \frac{1}{z - \xi_2}\biggr) dz & = 1 .
}
Therefore, the logarithm $I$ of the left-hand side of~\eqref{eq:r:curve} is given by
\formula{
 I & = \log c + \frac{1}{2 \pi i} \int_{\Gamma^\star} \biggl(\frac{1}{\pi} \int_{-\infty}^\infty \biggl(\frac{1}{z - \xi_1} - \frac{1}{z - \xi_2}\biggr) \biggl(\frac{z}{z + i s} - \frac{1}{1 + |s|}\biggr) \frac{\ph(s)}{|s|} \, ds\biggr) dz .
}
Fubini's theorem is applicable by the same argument as in the previous case. After changing the order of integration, the inner integral is simplified using the identity
\formula{
 \biggl(\frac{1}{z - \xi_1} - \frac{1}{z - \xi_2}\biggr) \biggl(\frac{z}{z + i s} - \frac{1}{1 + |s|}\biggr) & = \biggl(\frac{-i \xi_1}{-i \xi_1 + s} - \frac{1}{1 + |s|}\biggr) \biggl(\frac{1}{z - \xi_1} - \frac{1}{z + i s}\biggr) \\
 & \hspace*{2em} - \biggl(\frac{i \xi_2}{i \xi_2 - s} - \frac{1}{1 + |s|}\biggr) \biggl(\frac{1}{z - \xi_2} - \frac{1}{z + i s}\biggr) ,
}
which, by Lemma~\ref{lem:r:winding}, leads to
\formula{
 \hspace*{3em} & \hspace*{-3em} \frac{1}{2 \pi i} \int_{\Gamma^\star} \biggl(\frac{1}{z - \xi_1} - \frac{1}{z - \xi_2}\biggr) \biggl(\frac{z}{z + i s} - \frac{1}{1 + |s|}\biggr) dz \\
 & = \biggl(\frac{-i \xi_1}{-i \xi_1 + s} - \frac{1}{1 + |s|}\biggr) (1 - \ind_{\dom^+}(-i s)) - \biggl(\frac{i \xi_2}{i \xi_2 - s} - \frac{1}{1 + |s|}\biggr) (0 - \ind_{\dom^+}(-i s))
}
when $-i s \in \dom^+ \cup \dom^-$. Arguing as in the first part of the proof, we eventually find that
\formula{
 I & = \frac{1}{\pi} \int_0^\infty \biggl(\frac{-i \xi_1}{-i \xi_1 + s} - \frac{1}{1 + |s|}\biggr) \frac{\ph(s)}{|s|} \, ds + \frac{1}{\pi} \int_{-\infty}^0 \biggl(\frac{i \xi_2}{i \xi_2 - s} - \frac{1}{1 + |s|}\biggr) \frac{\ph(s)}{|s|} \, ds ,
}
which, combined with~\eqref{eq:r:wh}, gives the desired result stated in the third part of~\eqref{eq:r:curve}.
\end{proof}

Recall that $\Gamma_f^\star$ is parameterised by $\zeta_f(r)$, $r \in (-\zero_f) \cup \zero_f$, and $\zeta_f(-r) = -\overline{\zeta_f(r)}$ for $r > 0$. We use the notation $\lambda_f(r) = f(\zeta_f(r))$. The following corollary of Theorem~\ref{thm:r:curve} is almost immediate.

\begin{corollary}
\label{cor:r:curve}
If $f(\xi)$ is a non-degenerate Rogers function, $\xi_1, \xi_2 \in i \R \cap (\dom_f^+ \cup \dom_f^-)$, then
\formula[eq:r:curve:alt]{
 \hspace*{5em} & \hspace*{-5em} \exp \biggl(\frac{1}{\pi} \int_{\zero_f} \im \biggl(\frac{\zeta_f'(r)}{\zeta_f(r) - \xi_1} - \frac{\zeta_f'(r)}{\zeta_f(r) - \xi_2}\biggr) \log \lambda_f(r) dr\biggr) \\
 & = \begin{cases}
  f^+(-i \xi_1) / f^+(-i \xi_2) & \text{if $\xi_1, \xi_2 \in \dom_f^+$,} \\ 
  f^-(i \xi_2) / f^-(i \xi_1) & \text{if $\xi_1, \xi_2 \in \dom_f^-$,} \\ 
  f^+(-i \xi_1) f^-(-i \xi_2) & \text{if $\xi_1 \in \dom_f^+$, $\xi_2 \in \dom_f^-$.}
 \end{cases}
}
\end{corollary}

\begin{proof}
Substituting $z = \zeta_f(r)$ in~\eqref{eq:r:curve}, we obtain
\formula{
 \frac{f^+(-i \xi_1)}{f^+(-i \xi_2)} & = \exp \biggl(\frac{1}{2 \pi i} \int_{(-\zero_f) \cup \zero_f} \biggl(\frac{\zeta_f'(r)}{\zeta_f(r) - \xi_1} - \frac{\zeta_f'(r)}{\zeta_f(r) - \xi_2}\biggr) \log \lambda_f(r) dr\biggr) .
}
In order to prove formula~\eqref{eq:r:curve:alt}, it suffices to observe that $\lambda_f(-r) = \lambda_f(r)$, $\zeta_f'(-r) = \overline{\zeta_f'(r)}$ and $\zeta_f(-r) - \xi_1 = -(\overline{\zeta_f(r) - \xi_1})$.
\end{proof}

Our final result in this section is obtained from the above corollary by integration by parts. Recall that $\lambda_f(r)$ is a strictly increasing continuous function of $r \in (0, \infty)$ such that $\lambda_f(r) = f(\zeta_f(r))$ for $r \in (0, \infty) \setminus \partial \zero_f$.

\begin{theorem}
\label{thm:r:parts}
If $f$ is a non-zero Rogers function and $\xi_1, \xi_2 \in (0, i \infty)$, then
\formula[eq:r:parts:p]{
 \frac{f^+(-i \xi_1)}{f^+(-i \xi_2)} & = \exp \biggl(-\frac{1}{\pi} \int_0^\infty \bigl(\Arg (\zeta_f(r) - \xi_1) - \Arg(\zeta_f(r) - \xi_2)\bigr) \frac{d\lambda_f(r)}{\lambda_f(r)} \biggr) ,
}
where the integral is an (absolutely convergent) Riemann--Stieltjes integral. Similarly,
\formula[eq:r:parts:m]{
 \frac{f^-(i \xi_1)}{f^-(i \xi_2)} & = \exp \biggl(\frac{1}{\pi} \int_0^\infty \bigl(\Arg (\zeta_f(r) - \xi_1) - \Arg(\zeta_f(r) - \xi_2)\bigr) \frac{d\lambda_f(r)}{\lambda_f(r)} \biggr)
}
when $\xi_1, \xi_2 \in (-i \infty, 0)$. Finally, if $R > 0$, $\xi_1 \in (0, i \infty)$ and $\xi_2 \in (-i \infty, 0)$, then
\formula[eq:r:parts:pm]{
 f^+(-i \xi_1) f^-(i \xi_2) & = \lambda_f(R) \exp \biggl(-\frac{1}{\pi} \int_0^\infty \bigl(\Arg (\zeta_f(r) - \xi_1) - \Arg(\zeta_f(r) - \xi_2) \\
 & \hspace*{17.3em} + \pi \ind_{(0, R)}(r)\bigr) \frac{d\lambda_f(r)}{\lambda_f(r)} \biggr) .
}
When $f(0^+) > 0$, then one can additionally set $R = 0$ in~\eqref{eq:r:parts:pm}, with the convention that in this case $\lambda_f(R) = f(0^+)$.
\end{theorem}

\begin{proof}
As usual, we drop subscript $f$ from the notation. The assertion clearly holds true for degenerate Rogers functions, so we only consider the case when $f(\xi)$ is non-degenerate. Suppose that $\xi_1, \xi_2 \in (0, i \infty) \cap (\dom^+ \cup \dom^-)$ and $\im \xi_1 < \im \xi_2$. Our starting point is the integral in~\eqref{eq:r:curve}, and we will deduce~\eqref{eq:r:parts:p} by integration by parts.

For $r \in (0, \infty)$ we define
\formula{
 g(r) & = \Arg(\zeta(r) - \xi_1) - \Arg(\zeta(r) - \xi_2) ,
}
except possibly at $r = |\xi_1|$ and $r = |\xi_2|$. More precisely, if $\xi_1 \in \dom^-$, then $\zeta(r) = i r$ for $r$ in a neighbourhood of $|\xi_1|$, and so $g(r) = 0$ for $r$ in some left neighbourhood of $|\xi_1|$, and $g(r) = \pi$ for $r$ in some right neighbourhood of $|\xi_1|$. Similarly, if $\zeta \in \dom^-$, then $g(r) = \pi$ for $r$ in some left neighbourhood of $|\xi_2|$, and $g(r) = 0$ in some right neighbourhood of $|\xi_2|$. Except possibly for these two jump discontinuities, $g(r)$ is continuous on $(0, \infty)$. We need two more properties of $g(r)$.

The number $g(r)$ is the measure of the angle at $\zeta(r)$ in the triangle with vertices $\xi_1$, $\xi_2$ and $\zeta(r)$. By elementary geometry, one easily finds that $0 \le g(r) \le \pi$, and
\formula[eq:r:parts:g]{
 0 & \le g(r) \le C(\xi_1, \xi_2) \min\{r, r^{-1}\}
}
for some constant $C(\xi_1, \xi_2)$; we omit the details.

Note that $\Arg(z - \xi_1) - \Arg(z - \xi_2)$ is a continuously differentiable function of $z \in (\Cl \hp) \setminus \{\xi_1, \xi_2\}$, and $\zeta(r)$ is a locally absolutely continuous function. Therefore, the composition $g(r)$ of these two functions is locally absolutely continuous, except possibly at $r = |\xi_1|$ (if $\xi_1 \in \dom^-$) and $r = |\xi_2|$ (if $\xi_2 \in \dom^-$). Since $\Arg z = \im \log z$, for $r \in \zero$ we have
\formula[eq:r:parts:dg]{
 g'(r) & = \im \bigl(\log (\zeta(r) - \xi_1) - \log(\zeta(r) - \xi_2)\bigr)' = \im \biggl(\frac{\zeta'(r)}{\zeta(r) - \xi_1} - \frac{\zeta'(r)}{\zeta(r) - \xi_2}\biggr) .
}
Furthermore, $g(r) \in \{0, \pi\}$ for all $r \in (0, \infty) \setminus \zero$, and so if $r$ is a density point of $(0, \infty) \setminus \zero$ and $g'(r)$ exists, then $g'(r) = 0$. Therefore, $g'(r) = 0$ for almost all $r \in (0, \infty) \setminus \zero$.

Observe that $h(r) = \log \lambda(r)$ is a continuous increasing function and, by Proposition~\ref{prop:r:spine:bound},
\formula[eq:r:parts:h]{
 |h(r)| & = |\log f(\zeta_f(r))| \le |\log c| + \sqrt{2 \pi} \, \frac{1 + r}{\sqrt{r}} \le C(f) (r^{1/2} + r^{-1/2})
}
for some constant $C(f)$. Estimates~\eqref{eq:r:parts:g} and~\eqref{eq:r:parts:h} imply that
\formula[eq:r:parts:lim]{
 \lim_{r \to \infty} g(r) h(r) & = 0 , & \lim_{r \to 0^+} g(r) h(r) & = 0 .
}
Since $g(r)$ and $h(r)$ have locally bounded variation on $(0, \infty)$ and no common discontinuities, integration by parts and~\eqref{eq:r:parts:lim} lead to
\formula[eq:r:parts:int]{
 -\int_0^\infty g(r) dh(r) & = \int_0^\infty h(r) dg(r) ,
}
provided that either integral exists.

Note that $dh(r) = d\lambda(r) / \lambda(r)$, so that the integral in the left-hand side of~\eqref{eq:r:parts:int} coincides with the one in~\eqref{eq:r:parts:p}. Since $h(r)$ is increasing and $g(r)$ is non-negative, the integral, if convergent, is automatically absolutely convergent.

We now evaluate the integral in the right-hand side of~\eqref{eq:r:parts:lim}. Recall that $g(r)$ may have two jump discontinuities: at $r = |\xi_1|$, of size $\pi$, if $\xi_1 \in \dom^-$; and at $r = |\xi_2|$, of size $-\pi$, if $\xi_2 \in \dom^-$. Otherwise, $g(r)$ is absolutely continuous, with derivative $g'(r)$ given by~\eqref{eq:r:parts:dg} for $r \in \zero$ and equal to zero almost everywhere in $(0, \infty) \setminus \zero$. It follows that
\formula[eq:r:parts:int2]{
 \frac{1}{\pi} \int_0^\infty h(r) dg(r) & = \frac{1}{\pi} \int_Z h(r) g'(r) dr + h(|\xi_1|) \ind_{\dom^-}(\xi_1) - h(|\xi_2|) \ind_{\dom^-}(\xi_2) .
}
The integral in the right-hand side is identical to the one in Corollary~\ref{cor:r:curve}; in particular it is absolutely convergent. Furthermore, if $\xi_1 \in \dom^-$, then $\zeta(|\xi_1|) = \xi_1$, and so $h(|\xi_1|) = \log f(\zeta(|\xi_1|)) = \log f(\xi_1)$. Similarly, if $\xi_2 \in \dom^-$, then $h(|\xi_2|) = \log f(\xi_2)$.

Let us summarise what we have found so far: the integral in~\eqref{eq:r:parts:p} is convergent, and by combining~\eqref{eq:r:parts:int}, \eqref{eq:r:parts:int2} and Corollary~\ref{cor:r:curve} we obtain
\formula{
 \hspace*{7em} & \hspace*{-7em} -\frac{1}{\pi} \int_0^\infty \bigl(\Arg (\zeta_f(r) - \xi_1) - \Arg(\zeta_f(r) - \xi_2)\bigr) \frac{d\lambda(r)}{\lambda(r)} \\
 & =
 \begin{cases}
  \log (f^+(-i \xi_1) / f^+(-i \xi_2)) & \text{if $\xi_1, \xi_2 \in \dom_f^+$,} \\ 
  \log (f^-(i \xi_2) / f^-(i \xi_1)) + \log f(\xi_1) - \log f(\xi_2) & \text{if $\xi_1, \xi_2 \in \dom_f^-$,} \\ 
  \log (f^+(-i \xi_1) f^-(-i \xi_2)) - \log f(\xi_2) & \text{if $\xi_1 \in \dom_f^+$, $\xi_2 \in \dom_f^-$,} \\
  -\log (f^-(-i \xi_1) f^+(-i \xi_2)) + \log f(\xi_1) & \text{if $\xi_1 \in \dom_f^-$, $\xi_2 \in \dom_f^+$.}
 \end{cases}
}
Since $f(\xi_1) = f^+(-i \xi_1) f^-(i \xi_1)$ and $f(\xi_2) = f^+(-i \xi_2) f^-(i \xi_2)$, in each case the right-hand side is equal to $\log (f^+(-i \xi_1) / f^+(-i \xi_2))$, and consequently~\eqref{eq:r:parts:p} is proved when $\xi_1, \xi_2 \in (0, i \infty) \cap (\dom^+ \cup \dom^-)$ and $\im \xi_1 < \im \xi_2$. The case $\im \xi_1 > \im \xi_2$ follows by symmetry. Finally, extension to the case when $\xi_1$ or $\xi_2$ is in $(0, i \infty) \setminus (\dom^+ \cup \dom^-)$ follows by continuity: $(0, i \infty) \cap (\dom^+ \cup \dom^-)$ is dense in $(0, i \infty)$, and both sides of~\eqref{eq:r:parts:p} are continuous functions of $\xi_1, \xi_2 \in (0, i \infty)$. Indeed, continuity of the left-hand side is obvious, while for the right-hand side continuity is a consequence of dominated convergence theorem; we omit the details.

The proof of~\eqref{eq:r:parts:m} is very similar. In the proof of~\eqref{eq:r:parts:pm}, the definition of $g(r)$ is different: we consider $\xi_1 \in (0, i \infty) \cap (\dom^+ \cup \dom^-)$ and $\xi_2 \in (-i \infty, 0) \cap (\dom^+ \cup \dom^-)$, and we define
\formula{
 g(r) & = \Arg(\zeta(r) - \xi_2) - \Arg(\zeta(r) - \xi_1) - \pi \ind_{(0, R)}(r) .
}
Note that since $\im \xi_2 < 0 < \im \xi_1$, we have $0 \le \Arg(\zeta(r) - \xi_2) - \Arg(\zeta(r) - \xi_1) \le \pi$, so that $-\pi \le g(r) \le \pi$. By a slightly more involved geometric argument, estimate~\eqref{eq:r:parts:g} is again satisfied. The remaining part of the proof is very similar, except that $g(r)$ may have up to three jump discontinuities, at $r = |\xi_1|$ (if $\xi_1 \in \dom^-$), at $r = |\xi_2|$ (if $\xi_2 \in \dom^+$) and at $r = R$ (always; two of these discontinuities may cancel out, though). The last discontinuity gives rise to the additional factor $1/\lambda(R)$ in the left-hand side of~\eqref{eq:r:parts:pm}. We omit the details.

If $f(0^+) > 0$, then we can follow the above argument with $R = 0$, with the convention that in this case $\lambda(R) = \lambda(0^+) = f(0^+)$ (see Theorem~\ref{thm:r:lambda}). Estimate~\eqref{eq:r:parts:g} no longer holds, and $g(0^+) = -\pi$. Thus, $g(0^+) h(0^+) = -\pi \log \lambda(0^+)$ must be added to the right-hand side of~\eqref{eq:r:parts:int}. Again, we omit the details.
\end{proof}

%
%

\section{Space-time Wiener--Hopf factorisation}
\label{sec:xwh}

In this section we return to our original problem and prove Theorem~\ref{thm:main}. Recall that we consider a non-constant Lévy process $X_t$ with completely monotone jumps, its characteristic exponent $f(\xi)$, and the Wiener--Hopf factors $\kappa^+(\tau, \xi)$ and $\kappa^-(\tau, \xi)$. By Theorem~\ref{thm:rogers}, $f(\xi)$ extends to a non-zero Rogers function. 

For $\tau \ge 0$ we denote $f_\tau(\xi) = \tau + f(\xi)$. Then $f_\tau$ is a Rogers function, the spine $\Gamma_{f_\tau}$ of $f_\tau$ does not depend on $\tau \ge 0$, and we have $\zeta_{f_\tau}(r) = \zeta_f(r)$ and $\lambda_{f_\tau}(r) = \tau + \lambda_f(r)$ for $r > 0$.

We divide the proof into five of steps. First, however, we need an auxiliary lemma.

\begin{lemma}
\label{lem:key}
If $f(\xi)$ is a non-zero Rogers function, $f_\tau(\xi) = \tau + f(\xi)$ and $0 \le \xi_1 \le \xi_2$, then
\formula{
 & \frac{f_\tau^+(\xi_1)}{f_\tau^+(\xi_2)} && \text{and} && \frac{f_\tau^-(\xi_1)}{f_\tau^-(\xi_2)}
}
are complete Bernstein functions of $\tau$. Similarly, if $\xi_1, \xi_2 \ge 0$, then
\formula{
 f_\tau^+(\xi_1) f_\tau^-(\xi_2)
}
is a complete Bernstein function of $\tau$.
\end{lemma}

\begin{proof}
With no loss of generality we may assume that $f(0^+) = 0$: the general case follows then by applying the result to the Rogers function $g(\xi) = f(\xi) - f(0^+)$. By our assumption, $\lambda_f(0^+) = f(0^+) = 0$.

Suppose that $0 < \xi_1 < \xi_2$. Since $i \xi_1, i \xi_2 \in (0, i \infty)$, by Theorem~\ref{thm:r:parts} we have
\formula{
 \frac{f_\tau^+(\xi_1)}{f_\tau^+(\xi_2)} & = \exp \biggl(-\frac{1}{\pi} \int_0^\infty \bigl(\Arg (\zeta_f(r) - i \xi_1) - \Arg(\zeta_f(r) - i \xi_2)\bigr) \frac{d\lambda_f(r)}{\lambda_f(r) + \tau} \biggr) .
}
Write $\ph(\lambda_f(r)) = \Arg (\zeta_f(r) - i \xi_1) - \Arg(\zeta_f(r) - i \xi_2)$ for $r > 0$ (except perhaps $r = \xi_1$ and $r = \xi_2$). In the proof of Theorem~\ref{thm:r:parts} we noticed that $\ph(\lambda_f(r))$ takes values in $[0, \pi]$. It follows that
\formula{
 \biggl(\frac{f_\tau^+(\xi_1)}{f_\tau^+(\xi_2)}\biggr)^{-1} & = \exp \biggl(\frac{1}{\pi} \int_0^\infty \ph(\lambda_f(r)) \, \frac{d\lambda_f(r)}{\tau + \lambda_f(r)} \biggr) \\
 & = \exp \biggl(\frac{1}{\pi} \int_0^{\lambda_f(\infty^-)} \ph(s) \, \frac{1}{\tau + s} \, ds \biggr) .
}
The right-hand side is essentially the exponential representation~\eqref{eq:s:exp} of a Stieltjes function of $\tau$. Since the reciprocal of a Stieltjes function is a complete Bernstein function, the first part of the lemma is proved when $0 < \xi_1 < \xi_2$. The general case $0 \le \xi_1 \le \xi_2$ follows by continuity of the Wiener--Hopf factors and the fact that a point-wise limit of complete Bernstein functions, if finite, is again a complete Bernstein function.

In a similar way one shows that $f_\tau^-(\xi_1) / f_\tau^-(\xi_2)$ is a complete Bernstein function of $\tau$. To prove the last statement, we again use Theorem~\ref{thm:r:parts}, with $R = 0$: since $\lambda_{f_\tau}(0^+) = \tau + f(0^+) = \tau$, for $\xi_1, \xi_2 > 0$ we have
\formula{
 f_\tau^+(\xi_1) f_\tau^-(\xi_2) & = \tau \exp \biggl(\frac{1}{\pi} \int_0^\infty \ph(\lambda(r)) \, \frac{d\lambda_f(r)}{\tau + \lambda_f(r)} \biggr) ,
}
where $\ph(\lambda_f(r)) = \Arg (\zeta_f(r) + i \xi_2) - \Arg(\zeta_f(r) - i \xi_1)$ takes values in $[0, \pi]$. As in the first part of the proof, the exponential defines a Stieltjes function of $\tau$, and therefore $f_\tau^+(\xi_1) f_\tau^-(\xi_2)$ is a complete Bernstein function of $\tau$, as desired. The extension to $\xi_1 = 0$ or $\xi_2 = 0$ again follows by continuity.
\end{proof}

Before we proceed with the proof of Theorem~\ref{thm:main}, we clarify one aspect of the statement of the theorem. If $X_t$ is a compound Poisson process, then, according to our definitions, the expressions $\kappa^+(\tau, \xi_1) \kappa^-(\tau, \xi_2)$ and $\kappa^\circ(\tau) \kappa^+(\tau, \xi_1) \kappa^-(\tau, \xi_2)$ are different. Below we prove that both of them define a complete Bernstein function of $\tau$.

\begin{proof}[Proof of Theorem~\ref{thm:main}]
\emph{Step 1: the properties of $\kappa^+(\tau, \xi_1) / \kappa^+(\tau, \xi_2)$, $\kappa^-(\tau, \xi_1) / \kappa^-(\tau, \xi_2)$ and $\kappa^\circ(\tau) \kappa^+(\tau, \xi_1) \kappa^-(\tau, \xi_2)$ as functions of $\tau$.}
We use the notation $f_\tau(\xi) = \tau + f(\xi)$ introduced earlier in this section. According to Baxter--Donsker formulae~\eqref{eq:bd:p}, \eqref{eq:bd:m}, \eqref{eq:bd:pm} and Proposition~\ref{prop:r:bd}, we have
\formula[eq:bd:kappa:r:1]{
 \frac{\kappa^+(\tau, \xi_1)}{\kappa^+(\tau, \xi_2)} & = \frac{f_\tau^+(\xi_1)}{f_\tau^+(\xi_2)} \, , \qquad & \frac{\kappa^-(\tau, \xi_1)}{\kappa^-(\tau, \xi_2)} & = \frac{f_\tau^-(\xi_1)}{f_\tau^-(\xi_2)} ,
}
and
\formula[eq:bd:kappa:r:2]{
 \kappa^\circ(\tau) \kappa^+(\tau, \xi_1) \kappa^-(\tau, \xi_2) & = \frac{f_\tau^+(\xi_1) f_\tau^-(\xi_2)}{1 + f(0^+)}
}
for $\tau \ge 0$ and $\xi_1, \xi_2 > 0$. By Lemma~\ref{lem:key}, both expressions in~\eqref{eq:bd:kappa:r:1} are complete Bernstein functions of $\tau$ when $0 \le \xi_1 \le \xi_2$, and the expression in~\eqref{eq:bd:kappa:r:2} is a complete Bernstein function of $\tau$ when $\xi_1, \xi_2 \ge 0$, as desired.

\emph{Step 2: the properties $\kappa^+(\tau, \xi)$ and $\kappa^-(\tau, \xi)$ as functions of $\tau$.}
By~\eqref{eq:lim} we have
\formula{
 \lim_{\xi \to \infty} \frac{\kappa^+(\tau_1, \xi)}{\kappa^+(\tau_2, \xi)} & = 1
}
when $\tau_1, \tau_2 \ge 0$. It follows that
\formula{
 \kappa^+(\tau, \xi) & = \lim_{\eta \to \infty} \frac{\kappa^+(\tau, \xi) \kappa^+(1, \eta)}{\kappa^+(\tau, \eta)} \, .
}
We already know that, for fixed $\xi, \eta \ge 0$, the function $\kappa^+(\tau, \xi) \kappa^+(1, \eta) / \kappa^+(\tau, \eta)$ is a complete Bernstein function of $\tau$. Since a point-wise limit of complete Bernstein functions is again a complete Bernstein function, we conclude that $\kappa^+(\tau, \xi)$ is a complete Bernstein function of $\tau$. In a similar way one proves that $\kappa^-(\tau, \xi)$ is a complete Bernstein function of $\tau$.

\emph{Step 3: the properties of $\kappa^+(\tau_1, \xi) / \kappa^+(\tau_2, \xi)$ and $\kappa^+(\tau_1, \xi) / \kappa^+(\tau_2, \xi)$ as functions of $\xi$.}
Suppose that $0 \le \tau_1 \le \tau_2$, and define
\formula{
 g(\xi) & = \frac{\tau_1 + f(\xi)}{\tau_2 + f(\xi)} \, .
}
The following are Rogers functions of $\xi$ by Proposition~\ref{prop:r:prop}: $\tau_1 + f(\xi)$, $\xi^2 / (\tau_1 + f(\xi))$, $\xi^2 + (\tau_2 - \tau_1) \xi^2 / (\tau_1 + f(\xi))$, and finally
\formula{
 g(\xi) & = \frac{\tau_1 + f(\xi)}{\tau_2 + f(\xi)} = \xi^2 \biggl(\xi^2 + (\tau_2 - \tau_1) \frac{\xi^2}{\tau_1 + f(\xi)}\biggr)^{-1} .
}
By Baxter--Donsker formula~\eqref{eq:bd:0} and Proposition~\ref{prop:r:bd}, for all $\xi, \eta > 0$ we have
\formula{
 \frac{g^+(\xi)}{g^+(\eta)} & = \frac{\kappa^+(\tau_1, \xi) \kappa^+(\tau_2, \eta)}{\kappa^+(\tau_1, \eta) \kappa^+(\tau_2, \xi)} \, .
}
Since $g^+(\xi)$ is a complete Bernstein function, it follows that $\kappa^+(\tau_1, \xi) / \kappa^+(\tau_2, \xi)$ is a complete Bernstein function of $\xi$, as desired. A similar argument shows that $\kappa^-(\tau_1, \xi) / \kappa^-(\tau_2, \xi)$ is a complete Bernstein function of $\xi$.

\emph{Step 4: the properties of $\kappa^+(\tau, \xi)$ and $\kappa^-(\tau, \xi)$ as functions of $\xi$.}
As in step 2, by~\eqref{eq:lim} we have
\formula{
 \lim_{\tau \to \infty} \frac{\kappa^+(\tau, \xi_1)}{\kappa^+(\tau, \xi_2)} & = 1
}
when $\xi_1, \xi_2 \ge 0$. It follows that
\formula{
 \kappa^+(\tau, \xi) & = \lim_{\sigma \to \infty} \frac{\kappa^+(\tau, \xi) \kappa^+(\sigma, 1)}{\kappa^+(\sigma, \xi)} \, .
}
We already proved that, for fixed $\tau, \sigma \ge 0$, the function $\kappa^+(\tau, \xi) \kappa^+(\sigma, 1) / \kappa^+(\sigma, \xi)$ is a complete Bernstein function of $\xi$. A point-wise limit of complete Bernstein functions is again a complete Bernstein function, and so $\kappa^+(\tau, \xi)$ is a complete Bernstein function of $\xi$, as desired. By a similar argument, also $\kappa^-(\tau, \xi)$ is a complete Bernstein function of $\xi$.

\emph{Step 5: the properties of $\kappa^+(\tau, \xi_1) \kappa^-(\tau, \xi_2)$ as functions of $\tau$.}
We already know that $g_0(\tau) = \kappa^\circ(\tau)$, $g_1(\tau) = \kappa^+(\tau, \xi_1)$, $g_2(\tau) = \kappa^-(\tau, \xi_2)$, as well as $g_3(\tau) = g_0(\tau) g_1(\tau) g_2(\tau)$ are complete Bernstein functions of $\tau$. Let $\ph_j(s)$ denote the function $\ph(s)$ in the exponential representation~\eqref{eq:cbf:exp} for the complete Bernstein function $g_j(\xi)$, where $j = 0, 1, 2, 3$. Then $g_3(\tau)$ has exponential representation~\eqref{eq:cbf:exp} with function $\ph(s)$ equal either to $\ph_3(s)$ or to $\ph_0(s) + \ph_1(s) + \ph_2(s)$. By uniqueness of this representation, we necessarily have $\ph_0(s) + \ph_1(s) + \ph_2(s) = \ph_3(s)$ for almost all $s \in (0, \infty)$. In particular, $0 \le \ph_1(s) + \ph_2(s) \le \ph_3(s) \le \pi$. This, however, implies that the function $g_1(\tau) g_2(\tau)$ has exponential representation~\eqref{eq:cbf:exp} with function $\ph(s)$ equal almost everywhere to $\ph_1(s) + \ph_2(s) \in [0, \pi]$. Therefore, $g_1(\tau) g_2(\tau) = \kappa^+(\tau, \xi_1) \kappa^-(\tau, \xi_2)$ is a complete Bernstein function of $\tau$, and the proof is complete.
\end{proof}

%
%

\section{Local rectifiability of the spine}
\label{sec:proof}

This section contains the proof of Theorem~\ref{thm:r:real}\ref{it:r:real:d}. Our strategy is as follows. We first prove (in Lemma~\ref{lem:r:real:1}) inequality~\eqref{eq:r:real:curv}, which can be thought of as an upper bound for the curvature of the spine $\Gamma$ of $f(\xi)$. Next, we use this bound to prove (in Lemma~\ref{lem:r:real:2}) that $\Arg \zeta(r)$ cannot oscillate too rapidly away from the imaginary axis. To prove that $\Arg \zeta(r)$ does not oscillate between $\pm \tfrac{\pi}{2}$ too quickly, we show (in Lemma~\ref{lem:r:real:3}) that the zeroes of $\Arg \zeta(r)$ are separated when the derivative of $\Arg \zeta(r)$ is large. All these auxiliary results are used to prove a variant of Theorem~\ref{thm:r:real}\ref{it:r:real:d} in Lemma~\ref{lem:r:real:4}.

Throughout this section, we use the notation $\Gamma_f$, $\zeta_f(r)$ and $\zero_f$ introduced in Theorem~\ref{thm:r:real}, and for simplicity we omit the subscript $f$. We use logarithmic polar coordinates $\xi = e^{R + i \alpha}$ rather than the usual polar coordinates $\xi = r e^{i \alpha}$; the two are clearly related by the relation $r = e^R$. We write $\tilde{\zero} = \log \zero$, so that $R \in \tilde{\zero}$ if and only if $r = e^R \in \zero$. We also define $\Theta(R) = \thet(e^R) = \Arg \zeta(e^R)$, so that $\zeta(e^R) = e^{R + i \Theta(R)}$.

We begin with the proof of an equivalent form of formula~\eqref{eq:r:real:curv}.

\begin{lemma}
\label{lem:r:real:1}
For every $R \in \tilde{\zero}$, we have
\formula{
 |\Theta''(R)| & \le \frac{9 ((\Theta'(R))^2 + 1)}{\cos \Theta(R)} \, .
}
\end{lemma}

\begin{proof}
We denote $H(R, \alpha) = \Arg f(e^{R + i \alpha}) / \cos \alpha$. By~\eqref{eq:r:real:arg},
\formula[eq:r:real:1:arg]{
 H(R, \alpha) & = \frac{1}{\pi} \int_{-\infty}^\infty \frac{-e^R s}{e^{2 R} + 2 e^R s \sin \alpha + s^2} \, \frac{\ph(s)}{|s|} \, ds \\
 & = \frac{1}{\pi} \int_{-\infty}^\infty \frac{-t}{1 + 2 t \sin \alpha + t^2} \, \frac{\ph(e^R t)}{|t|} \, dt ,
}
where the second equality follows by a substitution $s = e^R t$. The former integral in the above display can be differentiated under the integral sign with respect to $R$ and $\alpha$; for example, we have
\formula[eq:r:real:angular]{
 \partial_\alpha H(R, \alpha) & = \frac{1}{\pi} \int_{-\infty}^\infty \frac{2 e^{2 R} s^2 \cos \alpha}{(e^{2 R} + 2 e^R s \sin \alpha + s^2)^2} \, \frac{\ph(s)}{|s|} \, ds \\
 & = \frac{1}{\pi} \int_{-\infty}^\infty \frac{2 t^2 \cos \alpha}{(1 + 2 t \sin \alpha + t^2)^2} \, \frac{\ph(e^R t)}{|t|} \, dt ,
}
where again the second equality is obtained by a substitution $s = e^R t$.

For every $R \in \tilde{\zero}$, we have $H(R, \Theta(R)) = 0$ and $\partial_\alpha H(R, \Theta(R)) > 0$, and therefore
\formula{
 \Theta'(R) & = -\frac{\partial_R H(R, \Theta(R))}{\partial_\alpha H(R, \Theta(R))} \, .
}
Similarly, we find that
\formula{
 \Theta''(R) & = -\frac{\partial_{RR} H + 2 \Theta'(R) \partial_{R\alpha} H + (\Theta'(R))^2 \partial_{\alpha\alpha} H}{\partial_\alpha H} \, ,
}
where to improve clarity we omitted the argument $(R, \Theta(R))$ of the partial derivatives of $H$. Since $H(R, \Theta(R)) = 0$, we have
\formula{
 \Theta''(R) & = -\frac{\partial_{RR} H + 2 \Theta'(R) \partial_{R\alpha} H + (\Theta'(R))^2 \partial_{\alpha\alpha} H - H}{\partial_\alpha H} \, .
}
The partial derivatives of $H$ in the right-hand side can be evaluated as in~\eqref{eq:r:real:angular}. After a lengthy calculation (that we omit here), we arrive at
\formula{
 \Theta''(R) & = \frac{1}{\pi} \, \frac{1}{\partial_\alpha H} \int_{-\infty}^\infty \frac{2 t^2 (A (\Theta'(R))^2 + B \Theta'(R) + C)}{(1 + 2 t \sin \Theta(R) + t^2)^3} \, \frac{\ph(e^R t)}{|t|} \, dt ,
}
with
\formula{
 A & = (1 + 2 t \sin \Theta(R) + t^2) \sin \Theta(R) + 4 t \cos^2 \Theta(R) , \\
 B & = 4 (1 - t^2) \cos \Theta(R) , \\
 C & = -3 (1 + 2 t \sin \Theta(R) + t^2) \sin \Theta(R) - 4 t \cos^2 \Theta(R) .
}
Using the inequalities $\cos^2 \alpha = 1 - \sin^2 \alpha \le 2 - 2 |\sin \alpha|$ and $2 |t| \le 1 + t^2$, we find that
\formula{
 |4 t \cos^2 \alpha| & \le 8 |t| - 8 |t \sin \alpha| \le 4 (1 + t^2) - 8 t \sin \alpha = 4 (1 + 2 t \sin \alpha + t^2) .
}
It follows that
\formula{
 |A| & \le 5 (1 + 2 t \sin \Theta(R) + t^2) , & |C| \le 7 (1 + 2 t \sin \Theta(R) + t^2) .
}
Furthermore, $(1 + 2 t \sin \alpha + t^2) + (1 - t^2) \cos \alpha = 2 (t \cos \tfrac{\alpha}{2} + \sin \tfrac{\alpha}{2})^2 \ge 0$ and $(1 + 2 t \sin \alpha + t^2) - (1 - t^2) \cos \alpha = 2 (t \sin \tfrac{\alpha}{2} + \cos \tfrac{\alpha}{2})^2 \ge 0$, and therefore
\formula{
 |B| & = |4 (1 - t^2) \cos \Theta(R)| \le 4 (1 + 2 t \sin \Theta(R) + t^2) .
}
The above estimates imply that
\formula{
 |\Theta''(R)| & \le \frac{1}{\pi} \, \frac{1}{\partial_\alpha H} \int_{-\infty}^\infty \frac{2 t^2 (5 (\Theta'(R))^2 + 4 \Theta'(R) + 7)}{(1 + 2 t \sin \Theta(R) + t^2)^2} \, \frac{\ph(e^R t)}{|t|} \, dt .
}
Comparing the right-hand side with~\eqref{eq:r:real:angular}, we conclude that
\formula{
 |\Theta''(R)| & \le \frac{5 (\Theta'(R))^2 + 4 \Theta'(R) + 7}{\cos \Theta(R)} \le \frac{9 ((\Theta'(R))^2 + 1)}{\cos \Theta(R)} \, ,
}
as desired.
\end{proof}

\begin{lemma}
\label{lem:r:real:2}
If $R_0 \in \tilde{\zero}$, $|\Theta'(R_0)| \le 1$ and $|R - R_0| < \tfrac{1}{90} \cos \Theta(R_0)$, then $R \in \tilde{\zero}$ and $|\Theta'(R)| \le 2$.
\end{lemma}

\begin{proof}
Let $h$ be the largest number with the following property: if $|R - R_0| < h$, then $R \in \tilde{\zero}$ and $|\Theta'(R)| \le 2$. Suppose, contrary to the assertion of the lemma, that $h < \tfrac{1}{90} \cos \Theta(R_0)$. If $|R_1 - R_0| \le h$, we have
\formula{
 |\Theta(R_1) - \Theta(R_0)| & = \abs{\int_{R_0}^{R_1} \Theta'(R) dR} \le 2 |R_1 - R_0| \le 2 h \le \tfrac{1}{2} \cos \Theta(R_0) ,
}
and therefore
\formula{
 \cos \Theta(R_1) & \ge \cos \Theta(R_0) - |R_1 - R_0| \ge \tfrac{1}{2} \cos \Theta(R_0) .
}
In particular, if $|R_1 - R_0| = h$, then $\cos \Theta(R_1) > 0$, and hence $R_1 \in \tilde{\zero}$.

Using the above estimates, the inequality $9 ((\Theta'(R))^2 + 1) \le 45$ and Lemma~\ref{lem:r:real:1}, we find that if $|R_1 - R_0| = h$, then
\formula{
 |\Theta'(R_1)| & = \abs{\Theta'(R_0) + \int_{R_0}^{R_1} \Theta''(R) dR} \\
 & \le 1 + \int_{R_0}^{R_1} \frac{9 ((\Theta'(R))^2 + 1)}{\cos \Theta(R)} \, dR \\
 & \le 1 + |R_1 - R_0| \, \frac{45}{\tfrac{1}{2} \cos \Theta(R_0)} < 2 .
}
The above inequality contradicts the maximality of $h$, and the proof is complete.
\end{proof}

\begin{lemma}
\label{lem:r:real:3}
If $R_0 \in \tilde{\zero}$, $\Theta(R_0) = 0$, $|\Theta'(R_0)| \ge 1$ and $0 < |R - R_0| < \log(1 + \sqrt{2})$, then $|\Theta(R)| \ne 0$.
\end{lemma}

\begin{proof}
Let $H(R, \alpha)$ be defined as in the proof of Lemma~\ref{lem:r:real:1}, and let $h \in \R$. Since $H(R_0, 0) = 0$ and $\partial_R H(R_0, 0) + \Theta'(R_0) \partial_\alpha H(R_0, 0) = 0$, we have
\formula{
 H(R_0 + h, 0) & = H(R_0 + h) - H(R_0, 0) \cosh h \\
 & \hspace*{8em} - (\partial_R H(R_0, 0) + \Theta'(R_0) \partial_\alpha H(R_0, 0)) \sinh h .
}
We evaluate the right-hand side using~\eqref{eq:r:real:1:arg}, \eqref{eq:r:real:angular} and a similar expression for the derivative with respect to $R$; in the expresison for $H(R_0 + h, 0)$ we use the same substitution $s = e^{R_0} t$ rather than $s = e^{R_0 + h} t$. This leads to
\formula{
 H(R_0 + h, 0) & = \frac{1}{\pi} \int_{-\infty}^\infty \biggl(\frac{-e^h t}{e^{2 h} + t^2} - \frac{-t}{1 + t^2} \, \cosh h \\
 & \hspace*{7em} - \frac{t (1 - t^2) + 2 t^2 \Theta'(R_0)}{(1 + t^2)^2} \, \sinh h\biggr) \frac{\ph(e^{R_0} t)}{|t|} \, dt .
}
After a lengthy calculation (that we omit here), we find that
\formula{
 H(R_0 + h, 0) & = -\frac{2 \sinh h}{\pi} \int_{-\infty}^\infty \frac{t^2 (\Theta'(R_0) (e^{2 h} + t^2) - (e^{2 h} - 1) t)}{(e^{2 h} + t^2)(1 + t^2)^2} \, \frac{\ph(e^{R_0} t)}{|t|} \, dt .
}
Suppose that $\Theta'(R_0) \ge 1$ and $|h| < \log(1 + \sqrt{2})$. Since $|t| \le \tfrac{1}{2} e^{-h} (e^{2 h} + t^2)$, we have
\formula{
 \Theta'(R_0) (e^{2 h} + t^2) - (e^{2 h} - 1) t & \ge (e^{2 h} + t^2) - |e^{2 h} - 1| |t| \\
 & \ge (e^{2 h} + t^2) (1 - \tfrac{1}{2} e^{-h} |e^{2 h} - 1|) \\
 & = (e^{2 h} + t^2) (1 - |\sinh h|) > 0 .
}
It follows that $H(R_0 + h, 0) < 0$ if $h > 0$ and $H(R_0 + h, 0) > 0$ if $h < 0$. When $\Theta'(R_0) \le -1$, the calculations are very similar: we find that $H(R_0 + h, 0) > 0$ if $h > 0$ and $H(R_0 + h, 0) < 0$ if $h < 0$. In particular, in either case we have $H(R_0 + h, 0) \ne 0$ when $0 < |h| < \log(1 + \sqrt{2})$, as desired.
\end{proof}

\begin{lemma}
\label{lem:r:real:4}
Let $h = \log(1 + \sqrt{2})$. For every $R_0$ we have
\formula{
 \int_{\tilde{\zero} \cap [R_0, R_0 + h]} |\Theta'(R)| dR & \le 140 .
}
\end{lemma}

\begin{proof}
Define $\tilde{\zero}_0$ to be the set of $R \in \tilde{\zero}$ for which $|\Theta'(R)| \le 1$. Clearly,
\formula[eq:r:real:4:1]{
 \int_{\tilde{\zero}_0 \cap [R_0, R_0 + h]} |\Theta'(R)| dR & \le |\tilde{\zero}_0 \cap [R_0, R_0 + h]| \le h .
}
Let $\tilde{\zero}_1$ be the union of those connected components $(R_1, R_2)$ of $\tilde{\zero} \setminus \tilde{\zero}_0$ on which $\Theta(R)$ takes value $0$, and let $\tilde{\zero}_2$ be the union of the remaining connected components of $\tilde{\zero} \setminus \tilde{\zero}_0$.

Note that on each connected component $(R_1, R_2)$ of $\tilde{\zero}_1$ or $\tilde{\zero}_2$, we have $\Theta'(R) \ne 0$ for $R \in (R_1, R_2)$, and so $\Theta(R)$ is monotone on $(R_1, R_2)$. If $(R_1, R_2)$ is a connected component of $\tilde{\zero}_1$, then
\formula{
 \int_{R_1}^{R_2} |\Theta'(R)| dR & = \abs{\int_{R_1}^{R_2} \Theta'(R) dR} = |\Theta(R_2) - \Theta(R_1)| \le \pi .
}
However, by Lemma~\ref{lem:r:real:3}, at most one connected component of $\tilde{\zero}_1$ is fully contained in $[R_0, R_0 + h]$, and so at most three connected components of $\tilde{\zero}_1$ intersect $[R_0, R_0 + h]$. It follows that
\formula[eq:r:real:4:2]{
 \int_{\tilde{\zero}_1 \cap [R_0, R_0 + h]} |\Theta'(R)| dR & \le 3 \pi .
}
Suppose now that $(R_1, R_2)$ is a connected component of $\tilde{\zero}_2$. Since $\Theta(R)$ is monotone on $(R_1, R_2)$ and $\Theta(R) \ne 0$ for $R \in (R_1, R_2)$, the number $\delta = \max\{\cos \Theta(R_1), \cos \Theta(R_2)\}$ is strictly positive. We assume that $\delta = \cos \Theta(R_1)$; the other case is very similar. We consider two scenarios. If $R_2 - R_1 \ge \tfrac{1}{90} \delta$, then
\formula{
 \int_{R_1}^{R_2} |\Theta'(R)| dR & = \abs{\int_{R_1}^{R_2} \Theta'(R) dR} = |\Theta(R_2) - \Theta(R_1)| \\
 & \le \tfrac{\pi}{2} - |\Theta(R_1)| \le \tfrac{\pi}{2} \cos \Theta(R_1) = \tfrac{\pi}{2} \delta \le 45 \pi (R_2 - R_1) .
}
On the other hand, if $R_2 - R_1 < \tfrac{1}{90} \delta$, then, by Lemma~\ref{lem:r:real:2},
\formula{
 \int_{R_1}^{R_2} |\Theta'(R)| dR & \le 2 (R_2 - R_1) .
}
Taking into account two connected components of $\tilde{\zero}_2$ which may intersect the boundary of $[R_0, R_0 + h]$, we conclude that 
\formula[eq:r:real:4:3]{
 \int_{\tilde{\zero}_2 \cap [R_0, R_0 + h]} |\Theta'(R)| dR & \le 45 \pi |\tilde{\zero}_2 \cap [R_0, R_0 + h]| + \pi .
}
The desired results follows by combining the three estimates~\eqref{eq:r:real:4:1}, \eqref{eq:r:real:4:2} and~\eqref{eq:r:real:4:3} and the inequality $h + 3 \pi + (45 \pi h + \pi) \le 140$.
\end{proof}

\begin{proof}[Proof of Theorem~\ref{thm:r:real}\ref{it:r:real:d}]
Formula~\eqref{eq:r:real:curv} is an equivalent form of Lemma~\ref{lem:r:real:1}, after substitution $\thet(r) = \Theta(\log r)$. The estimate for the length is a consequence of Lemma~\ref{lem:r:real:4}: since $\zeta(r) = r e^{i \Theta(\log r)}$, we have
\formula{
 |\zeta'(r)| & = |e^{i \Theta(\log r)}(1 + i \Theta'(\log r))| = ((\Theta'(\log r))^2 + 1)^{1/2} \le |\Theta'(\log r)| + 1 ,
}
and so
\formula{
 \int_{r_0}^{2 r_0} |\zeta'(r)| dr & \le \int_{r_0}^{2 r_0} (|\Theta'(\log r)| + 1) dr \\
 & = \int_{\log r_0}^{\log r_0 + \log 2} e^R |\Theta'(R) dR + r_0 \\
 & \le 2 r_0 \int_{\log r_0}^{\log r_0 + \log (1 + \sqrt{2})} |\Theta'(R)| dR + r_0 \le 300 r_0 ,
}
as desired.
\end{proof}

%
%

\bigskip
\section*{Acknowledgements}

I thank Sonia Fourati, Wissem Jedidi, Panki Kim, Alexey Kuznetsov, Pierre Patie, René Schilling and Zoran Vondra\v{c}ek for inspiring discussions on the subject of the article.

%
%

%
%


\begin{thebibliography}{00}

\bibitem{bib:a04}
D.~Applebaum,
\emph{L{\'e}vy Processes and Stochastic Calculus}.
Cambridge University Press, Cambridge, 2004.

\bibitem{bib:bnmr01}
O.~E.~Barndorff-Nielsen, T.~Mikosch, S.~I.~Resnick (Eds.),
\emph{Lévy Processes: Theory and Applications}.
Birkh{\"a}user, Boston, 2001.

\bibitem{bib:bd57}
G.~Baxter, M.~D.~Donsker,
\emph{On the distribution of the supremum functional for processes with stationary independent increments}.
Trans. Amer. Math. Soc. 85 (1957) 73--87.

\bibitem{bib:bdp08}
V.~Bernyk, R.~C.~Dalang, G.~Peskir,
\emph{The law of the supremum of a stable Lévy process with no negative jumps}.
Ann. Probab. 36(5) (2008): 1777--1789.

\bibitem{bib:b96}
J.~Bertoin,
\emph{Lévy Processes}.
Cambridge Univ. Press, Melbourne, New York, 1996.

\bibitem{bib:b73}
N.~H.~Bingham,
\emph{Maxima of sums of random variables and suprema of stable processes}.
Z.~Wahrscheinlichkeitstheorie Verw. Gebiete 26 (1973): 273--296.

\bibitem{bib:bbkrsv09}
K.~Bogdan, T.~Byczkowski, T.~Kulczycki, M.~Ryznar, R.~Song, Z.~Vondra\v{c}ek,
\emph{Potential Analysis of Stable Processes and its Extensions}.
Lecture Notes in Mathematics 1980, Springer, 2009.

\bibitem{bib:bgr13}
K.~Bogdan, T.~Grzywny, M.~Ryznar,
\emph{Density and tails of unimodal convolution semigroups}.
J.~Funct. Anal. 266(6) (2014): 3543--3571.

\bibitem{bib:bgr14}
K.~Bogdan, T.~Grzywny, M.~Ryznar,
\emph{Barriers, exit time and survival probability for unimodal Lévy processes}.
Probab. Theory Related Fields 162(1--2) (2015): 155--198.

\bibitem{bib:bgr14a}
K.~Bogdan, T.~Grzywny, M.~Ryznar,
\emph{Dirichlet heat kernel for unimodal Lévy processes}.
Stoch. Proc. Appl. 124(11) (2014): 3612--3650.

\bibitem{bib:bkkk14}
J.~Burridge, A.~Kuznetsov, A.~E.~Kyprianou, M.~Kwaśnicki,
\emph{New families of subordinators with explicit transition probability semigroup}.
Stoch. Proc. Appl. 124(10) (2014): 3480--3495.

\bibitem{bib:cks16}
Z.-Q.~Chen, P.~Kim, R.~Song,
\emph{Dirichlet Heat Kernel Estimates for Subordinate Brownian Motions with Gaussian Components}.
J.~Reine Angewandte Math. 711 (2016): 111--138.

\bibitem{bib:c15}
G.~Coqueret,
\emph{On the supremum of the spectrally negative stable process with drift}.
Stat. Probab. Lett. 107 (2015): 333--340.

\bibitem{bib:cgt17}
W.~Cygan, T.~Grzywny, B.~Trojan,
\emph{Asymptotic behavior of densities of unimodal convolution semigroups}.
Trans. Amer. Math. Soc. 369(8) (2017): 5623--5644.

\bibitem{bib:d56}
D.~A.~Darling,
\emph{The maximum of sums of stable random variables}.
Trans. Amer. Math. Soc. 83 (1956) 164--169.

\bibitem{bib:d87}
R.~A.~Doney,
\emph{On Wiener-Hopf factorisation and the distribution of extrema for certain stable processes}.
Ann. Probab. 15(4) (1987) 1352--1362.

\bibitem{bib:d07}
R.~A.~Doney,
\emph{Fluctuation Theory for Lévy Processes}.
Lecture Notes in Math. 1897, Springer, Berlin, 2007.

\bibitem{bib:dr11}
R.~A.~Doney, V.~Rivero,
\emph{Asymptotic behaviour of first passage time distributions for Lévy processes}.
Probab. Theory Related Fields 157(1) (2013): 1--45.

\bibitem{bib:ds10}
R.~A.~Doney, M.~S.~Savov,
\emph{The asymptotic behavior of densities related to the supremum of a stable process}.
Ann. Probab. 38(1) (2010): 316--326.

\bibitem{bib:fjw15}
S.~Fotopoulos, V.~Jandhyala, J.~Wang,
\emph{On the joint distribution of the supremum functional and its last occurrence for subordinated linear Brownian motion}.
Stat. Probab. Lett. 106 (2015): 149--156.

\bibitem{bib:f74}
B.~E.~Fristedt,
\emph{Sample functions of stochastic processes with stationary, independent increments}.
In: \emph{Advances in Probability and Related Topics}, vol. 3, Dekker, New York, 1974, 241--396.

\bibitem{bib:gj10}
P.~Graczyk, T.~Jakubowski,
\emph{On Wiener--Hopf factors of stable processes}.
Ann. Inst. Henri Poincaré (B) 47(1) (2010): 9--19.

\bibitem{bib:gj12}
P.~Graczyk, T.~Jakubowski,
\emph{On exit time of stable processes}.
Stoch. Proc. Appl. 122(1) (2012): 31--41.

\bibitem{bib:g12}
T.~Grzywny,
\emph{Potential theory of one-dimensional geometric stable processes}.
Colloq. Math. 129(1) (2012): 7--40. 

\bibitem{bib:g13}
T.~Grzywny,
\emph{On Harnack inequality and H\"older regularity for isotropic unimodal Lévy processes}.
Potential Anal. 41(1) (2014): 1--29.

\bibitem{bib:gr17}
T.~Grzywny, M.~Ryznar,
\emph{Hitting Times of Points and Intervals for Symmetric Lévy Processes}.
Potential Anal. 46(4) (2017): 739--777.

\bibitem{bib:gs19}
T.~Grzywny, K.~Szczypkowski,
\emph{Estimates of heat kernels of non-symmetric Lévy processes}.
Preprint, 2017, arXiv:1710.07793.

\bibitem{bib:gs19a}
T.~Grzywny, K.~Szczypkowski,
\emph{Heat kernels of non-symmetric Lévy-type operators}.
Preprint, 2018, arXiv:1804.01313.

\bibitem{bib:hk16}
D.~Hackmann, A.~Kuznetsov,
\emph{Approximating Lévy processes with completely monotone jumps}.
Ann. Appl. Probab. 26(1) (2016): 328--359.

\bibitem{bib:h69}
C.~C.~Heyde,
\emph{On the maximum of sums of random variables and the supremum functional for stable processes}.
J.~Appl. Probab. 6 (1969): 419--429.

\bibitem{bib:hk11}
F.~Hubalek, A.~Kuznetsov,
\emph{A convergent series representation for the density of the supremum of a stable process}.
Elect. Comm. Probab. 16 (2011): 84--95.

\bibitem{bib:jk14}
T.~Juszczyszyn, M.~Kwaśnicki,
\emph{Hitting times of points for symmetric Lévy processes with completely monotone jumps}.
Electron. J.~Probab. 20(48) (2015): 1--24.

\bibitem{bib:km14}
P.~Kim, A.~Mimica,
\emph{Green function estimates for subordinate Brownian motions: stable and beyond}.
Trans. Amer. Math. Soc. 366(8) (2014): 4383--4422.

\bibitem{bib:km18a}
P.~Kim, A.~Mimica,
\emph{Estimates of Dirichlet heat kernels for subordinate Brownian motions}.
Electron. J.~Probab. 23(64) (2018): 1--45.

\bibitem{bib:ksv12}
P.~Kim, R.~Song, Z.~Vondra\v{c}ek,
\emph{Potential theory of subordinate Brownian motions revisited}.
In: T.~Zhang, X.~Zhou (Eds.), \emph{Stochastic Analysis and Applications to Finance--Essays in Honour of Jia-an Yan}, World Scientific, 2012, 243--290.

\bibitem{bib:ksv13}
P.~Kim, R.~Song, Z.~Vondra\v{c}ek,
\emph{Potential theory of subordinate Brownian motions with Gaussian components}.
Stoch. Proc. Appl. 123 (2013): 764--795.

\bibitem{bib:ksv14}
P.~Kim, R.~Song, Z.~Vondra\v{c}ek,
\emph{Global uniform boundary Harnack principle with explicit decay rate and its application}.
Stoch. Proc. Appl. 124 (2014): 235--267

\bibitem{bib:ksv14a}
P.~Kim, R.~Song, Z.~Vondra\v{c}ek,
\emph{Boundary Harnack principle and Martin boundary at infinity for subordinate Brownian motions}.
Potential Anal. 41(2) (2014): 407--441.

\bibitem{bib:ksv18}
P.~Kim, R.~Song, Z.~Vondra\v{c}ek,
\emph{Heat kernels of non-symmetric jump processes: beyond the stable case}.
Potential Anal. 49(1) (2018): 37--90.

\bibitem{bib:kk17}
V.~Knopova, A.~Kulik,
\emph{Intrinsic compound kernel estimates for the transition probability density of a Lévy type processes and their applications}.
Probab. Math. Stat. 37(1) (2017): 53--100.

\bibitem{bib:kkms10}
T.~Kulczycki, M.~Kwaśnicki, J.~Małecki, A.~Stós,
\emph{Spectral properties of the Cauchy process on half-line and interval}.
Proc. London Math. Soc. 101(2) (2010): 589--622.

\bibitem{bib:ku10a}
A.~Kuznetsov,
\emph{Analytic Proof of Pecherski\u{\i}--Rogozin Identity and Wiener--Hopf Factorization}.
Theory Probab. Appl. 55(3) (2010): 432--443.

\bibitem{bib:ku10}
A.~Kuznetsov,
\emph{Wiener-Hopf factorization and distribution of extrema for a family of Lévy processes}.
Ann. Appl. Prob. 20(5) (2010): 1801--1830.

\bibitem{bib:ku11}
A.~Kuznetsov,
\emph{On extrema of stable processes}.
Ann. Probab. 39(3) (2011): 1027--1060.

\bibitem{bib:ku13}
A.~Kuznetsov,
\emph{On the density of the supremum of a stable process}.
Stoch. Proc. Appl. 123(3) (2013): 986--1003.

\bibitem{bib:kkp13}
A.~Kuznetsov, A.~Kyprianou, J.~C.~Pardo,
\emph{Meromorphic Lévy processes and their fluctuation identities}.
Ann. Appl. Probab. 22(3) (2012): 1101--1135.

\bibitem{bib:kk18}
A.~Kuznetsov, M.~Kwaśnicki,
\emph{Spectral analysis of stable processes on the positive half-line}.
Electron. J.~Probab. 23(10) (2018): 1--29.

\bibitem{bib:k06}
A.~E.~Kyprianou,
\emph{Introductory Lectures on Fluctuations of Lévy Processes with Applications}.
Universitext, Springer-Verlag, Berlin, 2006.

\bibitem{bib:k16:deep}
A.~E.~Kyprianou,
\emph{Deep factorisation of the stable process}.
Electron. J.~Probab. 21(23) (2016): 1--28.

\bibitem{bib:k18:deep}
A.~E.~Kyprianou, V.~Rivero, B.~Şengül,
\emph{Deep factorisation of the stable process II: Potentials and applications}.
Ann. Inst. H.~Poincaré Probab. Statist. 54(1) (2018): 343--362.

\bibitem{bib:k19:deep}
A.~E.~Kyprianou, V.~Rivero, W.~Satitkanitkul
\emph{Deep factorisation of the stable process III: Radial excursion theory and the point of closest reach}.
Preprint, 2017, arXiv:1706.09924.

\bibitem{bib:kpw14}
A.~E.~Kyprianou, J.~C.~Pardo, A.~R.~Watson,
\emph{The extended hypergeometric class of Lévy processes}.
J.~Appl. Probab. 51(A) (2014): 391--408.

\bibitem{bib:k11}
M.~Kwaśnicki,
\emph{Spectral analysis of subordinate Brownian motions on the half-line}.
Studia Math. 206(3) (2011): 211--271.

\bibitem{bib:k12a}
M.~Kwaśnicki,
\emph{Eigenvalues of the fractional Laplace operator in the interval}.
J.~Funct. Anal. 262(5) (2012): 2379--2402.

\bibitem{bib:k12}
M.~Kwaśnicki,
\emph{Spectral theory for one-dimensional symmetric Lévy processes killed upon hitting the origin}.
Electron. J. Probab. 17 (2012), 83:1--29.

\bibitem{bib:k18}
M.~Kwaśnicki,
\emph{A new class of bell-shaped functions}.
Preprint, 2017, arXiv:1710.11023.

\bibitem{bib:k:pre}
M.~Kwaśnicki,
\emph{Rogers functions and fluctuation theory}.
Unpublished, 2013, arXiv:1312.1866.

\bibitem{bib:kmr12}
M.~Kwaśnicki, J.~Małecki, M.~Ryznar,
\emph{First passage times for subordinate Brownian motions}.
Stoch. Proc. Appl 123 (2013): 1820--1850.

\bibitem{bib:kmr13}
M.~Kwaśnicki, J.~Małecki, M.~Ryznar,
\emph{Suprema of Lévy processes}.
Ann. Probab. 41(3B) (2013): 2047--2065.

\bibitem{bib:km18}
M.~Kwaśnicki, J.~Mucha,
\emph{Extension technique for complete Bernstein functions of the Laplace operator}.
J.~Evol. Equ. 18(3) (2018): 1341--1379.

\bibitem{bib:lm08}
A.L.~Lewis, E.~Mordecki,
\emph{Wiener--Hopf Factorization for Lévy Processes Having Positive Jumps with Rational Transforms}.
J.~Appl. Probab. 45(1) (2008): 118--134.

\bibitem{bib:m13}
Z.~Michna,
\emph{Explicit formula for the supremum distribution of a spectrally negative stable process}.
Electron. Commun. Probab. 18 (2013), 10:1--6.

\bibitem{bib:p14}
H.~Pantí,
\emph{On Lévy processes conditioned to avoid zero}.
ALEA, Lat. Am. J.~Probab. Math. Stat. 14 (2017): 657--690.

\bibitem{bib:ps17}
P.~Patie, M.~Savov,
\emph{Cauchy problem of the non-self-adjoint Gauss–Laguerre semigroups and uniform bounds for generalized Laguerre polynomials}.
J.~Spectral Theory 7(3) (2017): 797--846.

\bibitem{bib:ps19}
P.~Patie, M.~Savov,
\emph{Spectral expansions of non-self-adjoint generalized Laguerre semigroups}.
Mem. Amer. Math. Soc., to appear, arXiv:1506.01625.

\bibitem{bib:psz19}
P.~Patie, M.~Savov, Y.~Zhao,
\emph{Intertwining, Excursion Theory and Krein Theory of Strings for Non-self-adjoint Markov Semigroups}.
Preprint, 2017, arXiv:1706.08995.

\bibitem{bib:pz17}
P.~Patie, Y.~Zhao,
\emph{Spectral decomposition of fractional operators and a reflected stable semigroup}.
J.~Differ. Equations 262(3) (2017): 1690--1719.

\bibitem{bib:pr69}
E.A.~Pecherski, B.A.~Rogozin,
\emph{The joint distributions of random variables associated to fluctuations of a process with independent increments}.
Teor. Veroyatnost. Primenen. 14(3) (1969): 431--444; English transl. in Theory Probab. Appl. 14(3) (1969): 410--423.

\bibitem{bib:r83}
L.C.G.~Rogers,
\emph{Wiener--Hopf factorization of diffusions and Lévy processes}.
Proc. London Math. Soc. 47(3) (1983): 177--191.

\bibitem{bib:r66}
B.A.~Rogozin
\emph{Distribution of certain functionals related to boundary problems for processes with independent increments}.
Teor. Veroyatnost. i Primenen. 11(4) (1966): 656--670.

\bibitem{bib:s99}
K.~Sato,
\emph{Lévy Processes and Infinitely Divisible Distributions}.
Cambridge Univ. Press, Cambridge, 1999.

\bibitem{bib:ssv10}
R.~Schilling, R.~Song, Z.~Vondra\v{c}ek,
\emph{Bernstein Functions: Theory and Applications}.
De Gruyter, Studies in Math. 37, Berlin, 2012.

\bibitem{bib:tt02}
Y.~Tamura, H.~Tanaka,
\emph{On a fluctuation identity for multidimensional Lévy processes}.
Tokyo J.~Math. 25 (2002): 363--380.

\bibitem{bib:tt08}
Y.~Tamura, H.~Tanaka,
\emph{On a formula on the potential operators of absorbing Lévy processes in the half space}.
Stoch. Proc. Appl 118 (2008): 199--212.

\bibitem{bib:y12}
K.~Yano,
\emph{On harmonic function for the killed process upon hitting zero of asymmetric Lévy processes}.
J.~Math-for-Industry 5 (2013): 17--24.

\end{thebibliography}
\end{document}